\documentclass[11pt,reqno]{amsart}
\usepackage{bm,amsmath,amsthm,amssymb,mathtools,verbatim,amsfonts,tikz-cd,mathrsfs,diagbox,makecell}
\usepackage{enumitem}
\usepackage{float}
\usepackage{xfrac}
\usepackage{url}
\usepackage[hidelinks,colorlinks=true,linkcolor=blue, citecolor=black,linktocpage=true]{hyperref}
\usepackage{graphicx}
\usepackage[alphabetic]{amsrefs}
\usepackage{caption}
\usepackage{morefloats}
\usepackage[top=1.4in, bottom=1.2in, left=1.4in, right=1.4in,marginpar=1in]{geometry}
\usepackage{subfigure}
\usepackage{tikz}
\usepackage{adjustbox}
\usetikzlibrary{cd}
\usetikzlibrary{fit, patterns}
\usetikzlibrary{matrix,arrows,decorations.pathmorphing,decorations.pathreplacing,calligraphy}

\usepackage{chngcntr}
\counterwithin{figure}{section}

\newtheorem{thm}{Theorem}[section]
\newtheorem{prop}[thm]{Proposition}

\newtheorem{cor}[thm]{Corollary}
\newtheorem{lem}[thm]{Lemma}

\newtheorem{claim}{Claim}

\theoremstyle{definition}
\newtheorem{define}[thm]{Definition}

\theoremstyle{remark}
\newtheorem{rem}[thm]{Remark}
\newtheorem{example}[thm]{Example}

\newcommand{\ve}[1]{\boldsymbol{\mathbf{#1}}}

\newcommand{\R}{\mathbb{R}}

\newcommand{\Z}{\mathbb{Z}}

\renewcommand{\d}{\partial}
\renewcommand{\subset}{\subseteq}

\renewcommand{\tilde}{\widetilde}

\newcommand{\iso}{\cong}

\DeclareMathOperator{\alg}{{alg}}

\DeclareMathOperator{\free}{{free}}

\DeclareMathOperator{\gr}{{gr}}

\DeclareMathOperator{\id}{{id}}

\DeclareMathOperator{\im}{{im}}

\DeclareMathOperator{\rel}{{rel}}

\DeclareMathOperator{\sign}{{sign}}

\DeclareMathOperator{\tors}{{tors}}

\DeclareMathOperator{\Tors}{{Tors}}

\newcommand{\lk}{\mathrm{lk}}

\newcommand{\bE}{\mathbb{E}}
\newcommand{\bF}{\mathbb{F}}

\newcommand{\bH}{\mathbb{H}}
\newcommand{\bI}{\mathbb{I}}
\newcommand{\bJ}{\mathbb{J}}

\newcommand{\bM}{\mathbb{M}}

\newcommand{\bX}{\mathbb{X}}

\newcommand{\bZ}{\mathbb{Z}}

\newcommand{\cA}{\mathcal{A}}
\newcommand{\cB}{\mathcal{B}}
\newcommand{\cC}{\mathcal{C}}
\newcommand{\cD}{\mathcal{D}}

\newcommand{\cH}{\mathcal{H}}

\newcommand{\cK}{\mathcal{K}}

\newcommand{\cS}{\mathcal{S}}

\newcommand{\cX}{\mathcal{X}}
\newcommand{\cY}{\mathcal{Y}}

\newcommand{\frC}{\mathfrak{C}}

\newcommand{\scC}{\mathscr{C}}

\newcommand{\scE}{\mathscr{E}}
\newcommand{\scF}{\mathscr{F}}

\newcommand{\scJ}{\mathscr{J}}
\newcommand{\scK}{\mathscr{K}}
\newcommand{\scM}{\mathscr{M}}
\newcommand{\scN}{\mathscr{N}}

\newcommand{\scS}{\mathscr{S}}

\newcommand{\scW}{\mathscr{W}}

\newcommand{\cCFL}{\mathcal{C\!F\!L}}
\newcommand{\cCFK}{\mathcal{C\hspace{-.5mm}F\hspace{-.3mm}K}}
\newcommand{\cHFL}{\mathcal{H\hspace{-.5mm}F\hspace{-.3mm}L}}
\newcommand{\cHFK}{\mathcal{H\hspace{-.5mm}F\hspace{-.3mm} K}}
\newcommand{\CF}{\mathit{CF}}

\newcommand{\HF}{\mathit{HF}}

\newcommand{\HFK}{\mathit{HFK}}

\newcommand{\CFL}{\mathit{CFL}}
\newcommand{\HFL}{\mathit{HFL}}

\newcommand{\xs}{\ve{x}}
\newcommand{\ys}{\ve{y}}
\newcommand{\zs}{\ve{z}}
\newcommand{\ws}{\ve{w}}

\renewcommand{\a}{\alpha}
\renewcommand{\b}{\beta}

\newcommand{\veps}{\varepsilon}

\usepackage{url}
\usepackage{leftidx}

\DeclareMathOperator{\Cone}{{Cone}}

\newcommand{\Ss}[1]{\scriptstyle{#1}}

\numberwithin{equation}{section}

\newcommand{\ar}{\mathrm{a.r.}}

\allowdisplaybreaks

\newcommand{\Wh}{\mathrm{Wh}}

\title{Applications of the L-space satellite formula}

\author{Daren Chen}
\address{Department of Mathematics\\California Institute of Technology\\ Pasadena, CA, USA}
\email{darenc@caltech.edu}

\author{Ian Zemke}
\address{Department of Mathematics\\University of Oregon\\  Eugene, OR, USA}
\email{izemke@uoregon.edu}

\author{Hugo Zhou}
\address{Department of Mathematics\\University of Michigan\\  Ann Arbor, MI, USA}
\email{hugozhou@umich.edu}

\begin{document}
	\maketitle
	\begin{abstract}
		We give a formula for the $\tau$-invariant of a satellite knot $P(K,n)$ when $P$ is an L-space satellite operator. Our formula holds for general L-space satellite operators $P$ when the companion $K$ satisfies $\veps(K)=1$. When $\veps(K)$ is $0$ or $-1$, we state a formula which requires some additional assumptions on $P$ or $n$. Our main tool is our algorithm which computes the knot Floer complex of satellite knots constructed using L-space satellite operators, which we developed in a previous paper \cite{CZZ}. Our formula for $\tau$ recovers many existing formulas for the behavior of $\tau$ under satellite operators, including for cables. We apply our formula to questions about the slice genus of satellite knots, showing, e.g., that if $K$ is a knot with $\tau(K)=g_4(K)>0$, then satellites of $K$ by L-space satellite operators have the same property. Another application is a proof that L-space satellite operators satisfy a conjecture of Hedden and Pinz\'on-Caicedo: If $P$ is an L-space satellite operator which acts as a group homomorphism on the smooth concordance group, then $P$ is either the zero operator, the identity operator, or the orientation reversing operator. 
	\end{abstract}
	
	\tableofcontents
	
	\section{Introduction}

	Satellite operations are one of the most basic constructions in knot theory. Suppose that $K\subset S^3$ is a knot, $n\in \Z$ is an integer, and $P\subset S^1\times D^2$ is a knot in a solid torus. Then we can form a new knot, $P(K,n)\subset S^3$ as follows. We identify $S^1\times D^2$ with a tubular neighborhood of $K$ in such a way that  $S^1\times \{\theta\}$ (for $\theta\in \d D^2$) is sent to the $n$-framed longitude of $K$. Then $P(K,n)$ is the image of $P$ under the inclusion $S^1\times D^2\subset S^3$.
	
In this paper, we study satellite operators through the lens of Heegaard Floer theory. \emph{Knot Floer homology} is an invariant of knots due independently to Ozsv\'{a}th and Szab\'{o} \cite{OSKnots} and Rasmussen \cite{RasmussenKnots}. One of the most basic concordance invariants that can be extracted from knot Floer homology is the $\tau$-invariant of Ozsv\'{a}th and Szab\'{o} \cite{OS4ballgenus}. If $K\subset S^3$ is a knot, then
\[
|\tau(K)|\le g_4(K),
\]
where $g_4(K)$ is the slice genus of $K$.

	Many authors have studied how $\tau(K)$ behaves under satellite operators. The behavior of $\tau$ under cabling was determined by Hom \cite{HomTauCables}. For Whitehead doubles, this was understood by Hedden \cite{HeddenWhiteheadDoubles}. Many other authors have studied this question for specific satellite operators. See  \cite{LevineDoubling,HWCabling,Bodish-Satellites,PX-twobridge} for several examples. 
	
	In the present paper, we study the behavior of $\tau$ under a family of satellite operators called \emph{L-space satellite operators}. Our main technical result is a description of the behavior of $\tau$ under L-space satellite operators when Hom's $\veps$-invariant \cite{HomEpsilon} takes value 1 on the companion knot $K$. The cases when $\veps(K)=0$ or $-1$ are more complicated, and we obtain  partial results. This is stated in Theorem~\ref{thm:main-intro}.
	
	Using the aforementioned result about $\tau$, we prove a number of topological applications. In Theorem~\ref{thm:concordance-homomorphism-intro} that we prove that L-space satellite operators only act by homomorphisms on the concordance group when the underlying operator is one of three basic operators: the zero operator, the identity operator, or the orientation reversing operator. We are also able to compute the slice genus of knots of the form $P(K,0)$ when $P$ is an L-space satellite operator and $K$ is a knot with $\tau(K)=g_4(K)>0$. 
	
	\subsection{L-space satellite operators}
	
		If $P\subset S^1\times D^2$ is a satellite operator, we write $L_P=\mu\cup P$ for the 2-component link obtained by viewing the solid torus containing $P$ as the neighborhood of an unknot in $S^3$. We assume that $S^1\times \{\theta\}$ is sent to the 0-framed longitude of the unknot. The component $\mu$ is a meridian of this solid torus.
		
		If $P\subset S^1\times D^2$ is a satellite operator, one defines the \emph{winding number} of $P$ to be the class of the knot $P$ in $H_1(S^1\times D^2)\iso \Z$. For our purposes, it is more convenient to identify this with the linking number $\ell=\lk(P,\mu)$ of the link $L_P$.
		
		In \cite{CZZ}, we studied the following family of satellite operators:
		\begin{define} We say that a satellite operator $P$  is an \emph{L-space satellite operator} if the 2-component link $L_P$ is an L-space link, i.e., $S^3_{\Lambda}(L_P)$ is an L-space for all integral framings $\Lambda\gg (0,0)$. 
				\end{define}

			In this paper, we study the full knot Floer complex of Ozsv\'{a}th and Szab\'{o} \cite{OSKnots}. This can be packaged in several ways, but we study its formulation as a free, finitely generated chain complex $\cCFK(K)$ over the ring $\bF[W,Z]$.

				In \cite{CZZ}, we described a procedure for computing the knot Floer complex of $P(K,n)$ in terms of the knot Floer complex of $K$ when $P$ is an L-space satellite operator.	 Our description is phrased in terms of the bordered theory from \cite{ZemBordered,ZemExact}, which is based on the link surgery formula \cite{MOIntegerSurgery} of Manolescu and Ozsv\'{a}th. Therein, an associative algebra $\cK$ is described which encodes the Heegaard Floer surgery formulas. If $K\subset S^3$ is a knot, a module $\cX_n(K)^{\cK}$ over the algebra $\cK$ is described. To the negative Hopf link, there is a $DA$-bimodule ${}_{\cK} \cH_-^{\cK}$ over the algebra $\cK$. Finally, if $L_P$ is the 2-component link associated to a satellite operator $P$, then there is a $DA$-bimodule ${}_{\cK} \cX(L_P)^{\bF[W,Z]}$.
		 The gluing theorems from \cite{ZemBordered} imply that there is a homotopy equivalence
		\begin{equation}
		\cCFK(P(K,n))^{\bF[W,Z]} \simeq \cX_n(K)^{\cK} \boxtimes {}_{\cK} \cH_-^{\cK} \boxtimes {}_{\cK} \cX(L_P)^{\bF[W,Z]}.
		\label{eq:intro-description-satellite}
		\end{equation}
		
		In \cite{CZZ}, we described a procedure for computing the bimodule ${}_{\cK} \cX(L_P)^{\bF[W,Z]}$ when $P$ is an L-space operator. In this case, the bimodule ${}_{\cK} \cX(L_P)^{\bF[W,Z]}$ is computable from the multi-variable Alexander polynomials of $L_P$, and its sublinks. We implemented an algorithm to compute $\cCFK(P(K,n))^{\bF[W,Z]}$ based on Equation~\eqref{eq:intro-description-satellite} into Python code, which is avaiable here \cite{CZZCode}.

	\subsection{Formulas for $\tau(P(K,n))$}
	
	The main technical result of this paper is a formula for the behavior of $\tau$ under L-space satellite operators, in many cases. Our result is phrased in terms of several familiar Heegaard Floer invariants. The first is Hom's invariant $\veps(K)\in \{-1,0,1\}$ \cite{HomEpsilon}. Additionally, we consider the $H$-function of $L_P$, which we recall presently.	If $L\subset S^3$ is a 2-component link with components $L_1,L_2$, we set
	\[
	\bH(L):= \left(\frac{\lk(L_1,L_2)}{2}+\Z\right)\times \left(\frac{\lk(L_1,L_2)}{2}+\Z\right).
	\]
	The $H$-function of $L$ is a map
	\[
	H_L\colon \bH(L)\to \Z^{\ge 0}.
	\]
	It is defined as
	\[
	H_L(\ve{s})=-\frac{1}{2} \max \{\gr_{\ws}(\xs)| \xs\in \cHFL(L,\ve{s}), \bF[U]\text{-non-torsion}\}.
	\]
	In the above $\cHFL(L)$ denotes the homology of a version of link Floer homology over the four variable polynomial ring $\bF[W_1,Z_1,W_2,Z_2]$. The variable $U$ acts by $U_1=W_1Z_1$ or $U_2=W_2Z_2$. (Both variables $U_1$ and $U_2$ have the same action on homology). 
	
	It is helpful to define another numerical invariant $R_t\in \lk(L_1,L_2)/2+\Z$, ranging over $t\in \lk(L_1,L_2)/2+\Z$, via the formula
	\[
	R_t=\max\left\{r\in \frac{\lk(L_1,L_2)}{2}+\Z \middle|  \begin{array}{l}  H_{L}(t,r+1)=H_L(t,r)\text{ and}\\   H_L(t,r-1)=H_L(t,r)+1\end{array} \right\}. 
	\]

	Our main technical result is the following:
	 \begin{thm}
	 \label{thm:main-intro}  Let $P$ be an L-space satellite operator, and let $L_P=\mu \cup P$ be the associated 2-component L-space link. Orient the component $P$ of $L_P$ so that the linking number $\ell:=\lk(\mu,P)$ is non-negative. Write $g_3(P)$ for the Seifert genus of the pattern knot $P$ (viewed as a knot in $S^3$). We have the following formula for $\tau(P(K,n))$:
	 	
	 	\begin{enumerate}
	 		\item Suppose $\veps(K)=1$. Then
	 			\[\tau(P(K,n)) =
	 		\begin{cases}
	 			R_{\sfrac{\ell}{2}}-\frac{\ell}{2} +\frac{(\ell-1)\ell }{2}n + \ell\tau(K), & \text{if $ n<2\tau(K),$}\\
	 			g_3(P)+\frac{(\ell-1)\ell}{2}n + \ell\tau(K), & \text{if $n\ge 2\tau(K)$. }
	 		\end{cases}       \]	
	 	\item Suppose $\veps(K)=0. $
	 	\begin{enumerate}
	 	\item If $n\ge 0$, then 
	 	\[\tau(P(K,n)) = g_3(P) + \frac{(\ell-1)\ell}{2}n.\]
	 	\item If $n< 0$ and the following condition holds for the $H$-function of $L_P$:
	 	\begin{equation}
	 		R_{\sfrac{\ell}{2}-1}\ge g_3(P)+\frac{\ell}{2}-1,
	 		\label{eq:extra condition when epsilon=0}
	 	\end{equation}
	 	 then
	 		\[\tau(P(K,n)) =\max\left\{R_{\sfrac{\ell}{2}-1}+\tfrac{\ell}{2},R_{\sfrac{\ell}{2}}-\tfrac{\ell}{2}\right\} +\tfrac{(\ell-1)\ell}{2}n\]
	 		\end{enumerate}
	 	\item If $\veps(K)=-1$, and Equation~\eqref{eq:extra condition when epsilon=0} holds,
	 
 	then
	 	\[\tau(P(K,n)) =
	 	\begin{cases}
	 		\max\{R_{\sfrac{\ell}{2}-1} + \frac{\ell}{2},R_{\sfrac{\ell}{2}} - \frac{\ell}{2}\}+\frac{(\ell-1)\ell}{2}n+\ell\tau(K), & \text{if $ n<2\tau(K),$}\\
	 		\max\{R_{\sfrac{\ell}{2}-1} + \frac{\ell}{2}, R_{\sfrac{\ell}{2}+1}- \frac{\ell}{2}\} +\frac{(\ell-1)\ell}{2}n+\ell\tau(K), & \text{if $n=2\tau(K)$. }\\
	 		\min\{R_{\sfrac{\ell}{2}-1} + \frac{\ell}{2}, R_{\sfrac{\ell}{2}+1} + \frac{\ell}{2}\} +\frac{(\ell-1)\ell}{2}n+\ell\tau(K), & \text{if $n=2\tau(K)+1$. }\\
	 		\min\{R_{\sfrac{\ell}{2}-1} +\frac{\ell}{2},g_3(P)+\ell\}+\frac{(\ell-1)\ell}{2}n+\ell\tau(K), & \text{if $n>2\tau(K)+1$. }
	 	\end{cases}       \]	
	 	\end{enumerate}
	 \label{thm: tau}
	 \end{thm}
	
	Note that the condition that $\ell \ge 0$ is very mild, since it can be obtained by reversing the orientation of $P\subset L_P$, and this operation preserves $\tau(P(K,n))$ and does not affect $P$ being an L-space satellite operator.
	 
	 We describe a large number of examples in Section~\ref{sec:examples}, illustrating how Theorem~\ref{thm:main-intro} recovers many known formulas for the behavior of $\tau$ under satellite operations. These include the following:
	 \begin{enumerate}\item  Hom's formula for cables \cite{HomTauCables};
	 \item  Chen and Hanselman  formula for satellites by 1-bridge braids \cite{ChenHanselmanSatellites};
	 \item Hedden's formula for the Whitehead pattern \cite{HeddenWhiteheadDoubles};
	 \item Levine's formula for the Mazur pattern \cite{LevineMazur};
	 \item Bodish's formula for a family of generalized Mazur patterns \cite{Bodish-Satellites} (we recover only a subfamily of the patterns that Bodish considers);
	 \item Patwardhan and Xiao's formula for a family of satellite patterns $P$, which have the property that the 2-component link $L_P$ is a 2-bridge L-space link \cite{PX-twobridge}. 
	 \end{enumerate} 
	  The examples (3)--(6) fit into the family of patterns studied by Patwardhan and Xiao. These consist of certain patterns $P$ where $L_P$ is a two-bridge link of the form $L(rq-1,q)$, for $r,q$ positive, odd integers. These are L-space links by \cite{LiuLSpaceLinks}*{Theorem~3.8}, and furthermore we show that they satisfy the extra condition in Equation~\eqref{eq:extra condition when epsilon=0}. In particular, the Whitehead pattern corresponds to the case $q=r=3$, and the Mazur pattern corresponds to $q=3, r=5$.  Patwardhan and Xiao \cite{PX-twobridge} study this family of operators with the extra assumption that the framing $n$ is $0$.

	 The condition in Equation~\eqref{eq:extra condition when epsilon=0} seems somewhat restrictive. Nonetheless, we show that it is satisfied when $P$ has linking number $\ell\in \{0,1\}$ or if the pattern knot $P$ (viewed as a knot in $S^3$) is unknotted. See Remark~\ref{rem:extra assumption}. The condition is satisfied for the $(p,p+1)$ cabling operator (therefore recovering Hom's formula all $(p,pn+ 1)$-cables), but it is not satisfied for general cabling operators of the form $(p,pn+r)$ for $r\neq \pm 1$. Therefore, the statement of Theorem~\ref{thm:main-intro} does not immediately recover Hom's formula for the $\tau$ invariant of the $(p,q)$-cable of a knot $K$ with $\veps(K)< 0$. Nonetheless, we recall in Section~\ref{sec:examples} a simple argument using duality which recovers the formula when $\veps(K)<0$ from the formula when $\veps(K)>0$.  Similar arguments can be used in related settings to extend Theorem~\ref{thm:main-intro} in particular examples.

	 We will also give a topological interpretation of the quantity $R_{\sfrac{\ell}{2}}$ in terms of the Thurston norm of the link complement of $L_P$:
     
	 \begin{prop}\label{prop:R_ell/2-intro} If $P$ is an L-space satellite operator, then the quantity $R_{\sfrac{\ell}{2}}-|\ell|/2$ is equal to the minimum genus of a surface in $S^1\times D^2$ which has boundary equal to $P$ and a disjoint union of $|\ell|$ parallel copies of the longitude $S^1\times \{\theta\}$, for $\theta\in \d D^2$. Throughout the paper, we denote this quantity by $g_3^{\rel}(P)$.  
	 \end{prop}

	 This is proven in Proposition~\ref{prop:g3-P=Rll/2}, below. We are unaware of a simple interpretation of $R_t$ for $t\neq \frac{\ell}{2}$ in terms of the Thurston norm of $L_P$.

	 Note that using Proposition~\ref{prop:R_ell/2-intro}, the first part of Theorem~\ref{thm: tau} can be rephrased as claiming that if $\veps(K)=1$ and $P$ is an L-space satellite operator with $\ell \ge0$, then
	 			\[
	 			\tau(P(K,n)) =
	 		\begin{cases}
	 			g_3^{\rel}(P) +\frac{(\ell-1)\ell }{2}n + \ell\tau(K), & \text{if $ n<2\tau(K),$}\\
	 			g_3(P)+\frac{(\ell-1)\ell}{2}n + \ell\tau(K), & \text{if $n\ge 2\tau(K)$. }
	 		\end{cases}       
	 		\]
	 
	 We also study the effect of L-space satellite operators on Hom's $\veps$-invariant. We describe conditions under which we can guarantee that $\veps(P(K,n))>-1$. Note that $\veps(P(K,n))=0$ implies $\tau(P(K,n))=0$, so if $\tau(P(K,n))\neq 0$, and the criteria below are satisfied, then we can conclude $\veps(P(K,n))=1$. When $\tau(P(K,n))=0$, a more detailed cases-by-case analysis is needed. 
    \begin{thm}
    	 \label{thm:epsilon of satellite}
    	 Let $P$ be an L-space satellite operator, and $g_3(P)$ denote the Seifert genus of the pattern knot $P$. Orient the corresponding two-component link $L_P$ such that the linking number $\ell \ge 0$. Then, we have the following inequality of $\veps(P(K,n))$:
    	 \begin{enumerate}
    	 	\item If $\veps(K)=1$, or $\veps(K)=0$ and  $n\ge0$, then 
    	 	\[\veps(P(K,n)) \neq -1.\]
    	 	\item Suppose further that the H-function of the two-component link $L_P$ satisfies 
    	 	\begin{equation}
    	 		R_{\sfrac{\ell}{2}-1} \ge g_3(P) + \frac{\ell}{2}.\label{eq:condition epsilon}
    	 	\end{equation}
    	 	Then, if $\veps(K)=0$ and $ n<0$, or $\veps(K)=-1$, we have
    	 	\[\veps(P(K,n)) \neq -1.\]
    	 \end{enumerate}
    \end{thm} 	 
   In particular, if the condition in Equation~\eqref{eq:condition epsilon} is satisfied, then for any $n\in \mathbb{Z}$, the map on the smooth concordance group 
   \begin{align*}
   	P: \scC &\longrightarrow \scC\\
   	K &\longmapsto P(K,n)
   \end{align*} is not surjective. This extends the non-surjectivity results in \cite{LevineMazur} and \cite{PX-twobridge}.

	\subsection{Slice genus bounds and L-space satellite operators}
	
	We use our bounds on $\tau$ to investigate the slice genus of satellite knots. We recall that if $K\subset S^3$ is a knot, the \emph{slice genus} of $K$, denoted $g_4(K)$, is the minimum genus of a smoothly embedded, oriented surface in $B^4$ which bounds $K$.
	
	If $P\subset S^1\times D^2$ is a satellite knot, then we write $g_3^{\rel}(P)$ for the minimum genus of a surface in $S^1\times D^2$ which has boundary equal to $P$ and $|\ell|$ copies of a 0-framed longitude of $S^1\times D^2$, where $\ell$ is the winding number of $P$.
	
	We prove the following result, which was suggested to us by Jennifer Hom:
	
\begin{thm}
\label{thm:4-ball-genus-intro} Let $K\subset S^3$ be a knot such that $g_4(K)=\tau(K)>0$ and let $P$ be an L-space satellite operator. Write $\ell$ for the linking number of $L_P$ (recall that this is typically called the winding number of $P$). Then
\[
g_4(P(K,0))=\tau(P(K,0))=g_3^{\rel}(P)+|\ell|g_4(K).
\]
\end{thm}

As an example, note that if $K$ is itself an L-space knot, then $\tau(K)=g_4(K)$.

Our proof of Theorem~\ref{thm:4-ball-genus-intro} uses our formula for $\tau(P(K,0))$ from Theorem~\ref{thm:main-intro}, together with the description of the Thurston norm of the link $L_P$ in terms of the $H$-function of $L_P$, due to Liu \cite{BLiuCFK-LSpace}.

We are also able to prove an additional result for L-space satellites with winding number 0:

\begin{prop}
\label{prop:slice-genus-bounds-winding-number-0-intro}
 Suppose that $K\subset S^3$ is a knot with $g_4(K)=\tau(K)>0$, and $P$ is an L-space satellite operator with winding number 0. Then
\[
g_4(P(K,n))=g_3^{\rel}(P)
\]
for all $n<2\tau(K)$. 
\end{prop}

To illustrate the above result, we consider the positively clasped Whitehead pattern  $P=\Wh_+$. By \cite[Theorem 1.4]{HeddenWhiteheadDoubles}, if $n<2\tau(K)$ then $1=|\tau(\Wh_+(K,n))| \leq g_4(\Wh_+(K,n)) \leq g_3(\Wh_+(K,n))=1,$ so we obtain $g_4(\Wh_+(K,n))=1=g_3^{\rel}(\Wh_+)$ for any $K\subset S^3$.  Therefore Proposition~\ref{prop:slice-genus-bounds-winding-number-0-intro} is a natural generalization of Hedden's result.

In Section \ref{ss:2bridgelinks} we introduce a family of L-space satellite operators $P_{r,q}$ for odd integers $r$ and $q$ such that $r,q>1$ with winding number $\frac{r-q}{2}$ (see Figure \ref{fig:twobridgefamily} for an example). When $r=q$, this gives an infinite family of L-space satellite operators with winding number $0$, where $r=q=3$ recovers the Whitehead double.

\subsection{Patterns with minimal wrapping number}
\label{sec:minimal-wrapping-intro}

Recall that the \emph{wrapping number} of a pattern $P$ is the minimal geometric intersection number between $P$ and a meridianal disk in $S^1\times D^2$. By definition, the wrapping number is at least $|\ell|=|\lk(P,\mu)|$. Our formulas simplify somewhat when the wrapping number is equal to $|\ell|$. In this case, we will say that $P$ has \emph{minimal wrapping number}. 

If $P$ is braided about $\mu$ (for example if $P$ is a cabling operator) then $P$ has minimal wrapping number.

 We prove the following:

\begin{prop}
\label{prop:minimally-wrapped-seifert-genus-intro}
 Suppose that $P$ is an L-space pattern with minimal wrapping number. Then 
 \[
 R_{\sfrac{\ell}{2}}-|\ell|/2=g_3^{\rel}(P)=g_3(P).
 \]
\end{prop}
Note that Proposition~\ref{prop:minimally-wrapped-seifert-genus-intro} can be used to simplify some of the computations in Theorem~\ref{thm:main-intro}. In particular, if $\veps(K)>0$ and $P$ is a minimally wrapped L-space satellite operator, then the first part of Theorem~\ref{thm:main-intro} now reads
\[
\tau(P(K,n))=g_3(P)+\frac{(\ell-1)\ell}{2} n+\ell \tau(K),
\]
whenever $\ell\ge 0$. This quickly recovers Hom's formula \cite{HomTauCables} for the $\tau$ invariant of a cable of a knot $K$ when $\veps(K)>0$. Using duality, one may derive the formula when $\veps(K)<0$ from the formula for $\veps(K)>0$. See Section~\ref{sec:examples} for more details.

When $P$ has minimal wrapping number, we can also extend Theorem~\ref{thm:4-ball-genus-intro} to handle the case that $n\ge 0$. We will prove the following:

\begin{prop}
\label{prop:minimal-wrapping-number-slice-genus-general-n} Suppose that $P$ is an L-space satellite operator which has minimal wrapping number and let $n\ge 0$. Assume that $\ell=\lk(P,\mu)\ge 0$. If $\tau(K)=g_4(K)>0$, then
\[
g_4(P(K,n))=\tau(P(K,n))=g_3(P)+\frac{\ell(\ell-1)}{2} n+\ell \tau(K).
\]
\end{prop}
\begin{proof} The statement is an immediate consequence of Theorems~\ref{thm:4-ball-genus-intro} and Proposition~\ref{prop:minimally-wrapped-seifert-genus-intro}.
\end{proof}

	\subsection{L-space satellite operators and homomorphism obstructions}
	
	Using our computation of $\tau$ from Theorem~\ref{thm: tau}, as well as our result from \cite{CZZ} that the $H$-function of $L_P$ determines $\cCFL(L_P)$ when $L_P$ is an L-space link, we show that L-space satellite operators satisfy a conjecture by Hedden and Pinz\'on-Caicedo \cite{Hedden-PinzonCaicedo} about when satellite operators act by group homomorphisms on the concordance group. We let $\scC$ denote the smooth concordance group of knots in $S^3$. This group is generated by oriented knots, modulo smooth concordance. The group operation is connected sum.

 Hedden and Pinz\'on-Caicedo conjecture that the only homomorphisms induced by satellite operations are the following:
	\begin{enumerate}
	\item The zero map.
	\item The identity map.
	\item The orientation reversal map (i.e. the map which sends $K$ to itself with orientation reversed).
	\end{enumerate}

	 	In \cite{Miller_Homomorphism_Obstruction}, \cite{LMPCsatellitehomomorphisms} and \cite{JKMhomomorphisms}, the conjecture is proven for several families of patterns. In \cite{Miller_Homomorphism_Obstruction} and \cite{LMPCsatellitehomomorphisms} the conjecture is proven for patterns satisfying certain technical conditions involving branched covers of the pattern, while in \cite{JKMhomomorphisms} it is proven for all patterns where the winding number is even, but not divisible by 8.
	 	
	 	The signature function $\sigma_K(\omega)$ of Levine-Tristram \cite{Levine_Signature} \cite{Tristram_Signature} also provides a simple homomorphism obstruction. We recall the formula of Levine-Tristram signature of satellite knots in  \cite{LitherlandSignature} that 
	 	\[
	 	\sigma_{P(K,0)}(\omega)=\sigma_{P(U,0)}(\omega)+\sigma_{K}(\omega^{\ell}).
	 	\]
	 	Therefore, if $\sigma_{P(U,0)}(\omega)$ is non-zero on a set of positive measure, then $P$ cannot act by a homomorphism on the concordance group. 
	 
	 Using our results about $\tau(P(K,n))$ for L-space satellite operators, we prove the following:
	 
	 \begin{thm}
	 \label{thm:concordance-homomorphism-intro}
	 	 If $L_P = \mu \cup P$ is an $L$-space link, and the map 
	 	\begin{align*}
	 		P: \scC &\longrightarrow \scC\\
	 		K &\longmapsto P(K,n)
	 	\end{align*}
	 	is a group homomorphism for some $n\ge 0$, then 
	 	$L_P$ is the $2$-component unlink, the positive Hopf link, or the negative Hopf link. In particular, the satellite operator $P$ induces either the trivial map, the identity map, or the orientation-reversing map.
	 	\label{thm:non homomorphism}
	 \end{thm}

\subsection{Further applications}

In Section~\ref{sec:inequality} we prove an additional result, which is independent of our computation of $\tau$. We prove the following general inequality about the concordance invariants of L-space satellites:

\begin{prop}\label{prop:inequality-intro} Suppose that $P$ is an L-space satellite with linking number $\ell\ge 0$ and $n\in \Z$. For any knot $K\subset S^3$, there are inequalities
\begin{enumerate}
\item $\tau(P(K,n))\ge \tau(K_{\ell,\ell n+1})$.
\item $V_k(P(K,n))\ge V_k(K_{\ell,\ell n+1})$ for all $k\in \Z$.
\item $\Upsilon_{P(K,n)}(t)\le \Upsilon_{K_{\ell,\ell n+1}}(t)$ for all $t\in [0,2]$. 
\end{enumerate}
\end{prop}

In the above, $V_k$ denotes Rasmussen's local $h$-invariant \cite{RasmussenGodaTeragaito}, and $\Upsilon_K(t)$ denotes the concordance invariant of Ozsv\'{a}th, Stipsicz and Szab\'{o} \cite{OSSUpsilon}. 

When $\ell=0$, we observe that the cabled knot $K_{\ell,\ell n+1}$ is always the unknot, so we immediately obtain:

\begin{cor}
\label{cor:linking-number-0-intro}
 If $P$ is a winding number 0 L-space satellite operator, then
\[
\tau(P(K,n))\ge 0,\quad \text{and}\quad 
 \Upsilon_{P(K,n)}(t)\le 0. 
\]
\end{cor}

We show Proposition~\ref{prop:inequality-intro}, by showing that if $P$ is an L-space satellite operator of winding number $\ell\ge 0$, then there exists a natural map of bimodules
\[
{}_{\cK} \cX(L_P)^{\bF[W,Z]}\to {}_{\cK} \cX(T(2,2\ell))^{\bF[W,Z]}.
\]
We study the above map to extract the inequalities in Proposition~\ref{prop:inequality-intro}.

	\subsection{Organization}
	In Section \ref{sec:H function}, we recall some background on knot and link Floer homology, in particular the notion of the $H$-function. In Section \ref{sec:surgery}, we review the notion of link surgery modules. In Section \ref{sec: DA bimodule}, we study in detail the link surgery bimodules of the link $L_P$ associated with an L-space satellite operator $P$. In Section \ref{sec: tau computation}, we compute $\tau(P(K,n))$. In Section \ref{sec:epsilon}, we discuss the inequality of the $\veps$-invariant for $P(K,n)$. In Section \ref{sec: Thurston norm}, we examine the relation between our formula with the Thurston norm and slice genus. In Section \ref{sec:examples}, we perform some example computations of $\tau(P(K,n))$ using our formula. In Section \ref{sec: non homomorphism}, we prove the homomorphism obstruction result for L-space satellite operators.  In Section \ref{sec:inequality}, we describe a local inequality of $P(K,n)$. 
	\subsection{Acknowledgements} 
	
	We thank Jennifer Hom and Yi Ni for helpful discussions. DC was partially supported by the Simons Collaboration Grant on New Structures in Low Dimensional Topology. IZ was partially supported by NSF grants 2204375, 2513241 and a Sloan fellowship. HZ was supported by an AMS-Simons travel grant.

\section{Background on knot and link Floer homology}
\label{sec:H function}

\subsection{Knot Floer homology}

Knot Floer homology is a refinement of Heegaard Floer homology for knots, due independently to Ozsv\'{a}th and Szab\'{o} \cite{OSKnots}, and Rasmussen \cite{RasmussenKnots}. Link Floer homology is a generalization of this theory to links, due to Ozsv\'{a}th and Szab\'{o} \cite{OSLinks}.

Given a knot in $S^3$, we will consider a finitely generated, free chain complex $\cCFK(K)$ over the two-variable polynomial ring 
\[
R=\bF[W,Z].
\]
The chain complex $\cCFK(K)$ has a $\Z\times \Z$ Maslov bigrading, denoted $(\gr_{\ws}, \gr_{\zs})$. Here,
\[
(\gr_{\ws},\gr_{\zs})(W)=(-2,0)\quad \text{and} \quad (\gr_{\ws}, \gr_{\zs})(Z)=(0,-2).
\]
The \emph{Alexander grading} is defined to be
\[
A:=\frac{ \gr_{\ws}-\gr_{\zs}}{2}.
\]

If $L\subset S^3$ is an $n$-component link, we consider a version of link Floer homology, denoted $\cCFL(L)$, which takes the form of a finitely generated, free chain complex $\cCFL(L)$ over the ring 
\[
\bF[W_1,Z_1,\dots, W_n,Z_n].
\]
This chain complex has a $\Z\times \Z$-valued Maslov bigrading $(\gr_{\ws}, \gr_{\zs})$ and an $n$-component Alexander grading $A=(A_1,\dots, A_n)$. The Alexander grading takes values in the set
\begin{equation}
	\bH(L)=\prod_{i=1}^n\left( \Z+\frac{\lk(L_i,L\setminus L_i)}{2}\right),
	\label{eq:Alexander-grading-def}
\end{equation}
where $L_1,\dots, L_n$ denote the components of $L$. The Alexander grading satisfies
\[
\sum_{i=1}^{n}A_i = \frac{\gr_{\ws}-\gr_{\zs}}{2}.
\]

The differential has the following grading:
\[
(\gr_{\ws},\gr_{\zs})(\d)=(-1,-1)\quad \text{and} \quad A(\d)=(0,\dots, 0). 
\]

\subsection{Concordance invariants}
\label{sec:concordance invariants}

We now recall the definitions of several knot Floer  concordance invariants which appear in our paper. Firstly, we recall Ozsv\'{a}th and Szab\'{o}'s $\tau(K)$-invariant \cite{OS4ballgenus}. Write $\cCFK_{\bF[Z]}(K)$ for $\cCFK(K)/(W)$, and write $\cHFK_{\bF[Z]}(K)$ for $H_*\left( \cCFK_{\bF[Z]}(K)\right)$. In the literature, this is frequently denoted $\HFK^-(K)$. The localization $Z^{-1} \cHFK_{\bF[Z]}(K)$ can be identified with $\widehat{\HF}(S^3)\otimes \bF[Z,Z^{-1}]$, and therefore we have
\[
\cHFK_{\bF[Z]}(K)\iso \bF[Z]\oplus \Tors,
\] 
where $\Tors$ is an $\bF[Z]$-module which satisfies $Z^{N}\cdot \Tors=0$ for sufficiently large $N$.  The $\tau(K)$ invariant is defined as the Alexander grading of the generator of the tower $\bF[Z]$, or equivalently
\[
\tau(K)=\min \{A(\xs): \xs\in \cHFK_{\bF[Z]}(K), \, \xs \text{ is $Z$-non-torsion}\}. 
\]

We now recall Hom's $\veps$ invariant \cite{HomEpsilon}. Write $\hat{R}=\bF[W,Z]/(U)$, where $U=WZ$. We consider knot Floer complexes over $\hat{R}$, defined by
\[
\cCFK_{\hat{R}}(K):=\cCFK(K)\otimes {}_{\bF[W,Z]} \hat{R}.
\]

\begin{define} 
\label{def:local map}	
We say a chain map $F\colon \cCFK_{\hat{R}}(K_1)\to \cCFK_{\hat{R}}(K_2)$ is an \emph{$\hat{R}$-local map} if
\begin{enumerate}
\item The map $F$ preserves the $\gr_{\ws}$-grading and is homogeneous with respect to the $\gr_{\zs}$-grading.
\item The map $F$ map is an isomorphism on homology after quotienting by $\bF[Z]$-torsion.
\end{enumerate}
\end{define}
(This is called a \emph{local map} in \cite{DHSTmore}*{Definition~3.3}). 

Viewing $\hat{R}$ as $\cCFK_{\hat{R}}(U)$, where $U$ is the unknot, we say that $\veps(K)\in \{0,1\}$ if there is an $\hat{R}$-local map from $\hat{R}$ to $\cCFK_{\hat{R}}(K)$. We say that $\veps(K)\in \{ 0,-1\}$ if there is an $\hat{R}$-local map from  $\cCFK_{\hat{R}}(K)$ to $\hat{R}$. We say that $\veps(K)=0$ if and only if there are $\hat{R}$-local maps in both directions, and we set $\veps(K)$ to be $+1$ or $-1$ otherwise. For the equivalence with Hom's original definition, see \cite{DHSTmore}*{Remark~2.4}. 

\subsection{The $H$-function}

The $H$-function of a link $L\subset S^3$ takes the form of a map
\[
H_L\colon \bH(L)\to \Z^{\ge 0}.
\]
To define $H_L$, we consider a subcomplex
\[
\cCFL(L,\ve{s})\subset \cCFL(L),
\]
for each $\ve{s}\in \bH(L)$. The complex $\cCFL(L,\ve{s})$ is the subspace in Alexander grading $\ve{s}$. We write 
\[
U_i=W_iZ_i, \text{ for } i=1,\dots,n,
\] 
and note that the actions of $U_1,\dots, U_n$ preserve the $n$-component Alexander grading. Since the differential also preserves the Alexander grading, the complex $\cCFL(L,\ve{s})$ is a free, finitely generated chain complex over $\bF[U_1,\dots, U_{n}]$. 

It is not hard to see that for all $i,j$, the variables $U_i$ and $U_j$ induce chain homotopic endomorphisms of $\cCFL(L,\ve{s})$. See, e.g.,  \cite{BLZNonCuspidal}*{Equation~2.8}. 
Therefore, we view the homology group $\cHFL(L,\ve{s})$ as a module over $\bF[U]$, where $U$ acts by any of the $U_i$. By the link surgery formula by Manolescu and Ozsv\'{a}th \cite{MOIntegerSurgery}, we have 
\[\cHFL(L,\ve{s})\cong\bF[U]\oplus \Tors\] as $\bF[U]$-modules, where $\Tors$ is an $\bF[U]$-module which satisfies $U^{N}\cdot \Tors=0$ for sufficiently large $N$. (i.e., the $U$-torsion part.)

We make the following definition (which first appeared in \cite{GorskyNemethiAlgebraicLinks}):
\begin{define} 
	The map $H_L\colon \bH(L)\to \Z$ is defined via the equation
	\[H_L(\ve{s})=-\frac{1}{2} \max \{\gr_{\ws}(\xs): \xs\in \cHFL(L,\ve{s}), \xs \text{ is $U$-non-torsion}\}.\]
	\label{def:H function}
\end{define}
We now recall the definition of an $L$-space link:
\begin{define}
	If $L\subset S^3$ is a link, we say that $L$ is an \emph{L-space link} if $S^3_{\Lambda}(L)$ is an L-space for all sufficiently positive framings $\Lambda$.
\end{define}
The following is a helpful reformulation of the above definition. See, e.g., \cite{LiuLSpaceLinks}*{Proposition~1.11}:
\begin{lem} A link $L\subset S^3$ is an L-space link if and only if $\cHFL(L,\ve{s})\iso \bF[U]$ for all $\ve{s}\in \bH(L)$. Equivalently, $L$ is an L-space link if and only $\cHFL(L)$ is free as an $\bF[U]$-module. 
\end{lem}

The following result, due to Gorsky and N\'{e}methi, describes how to compute the $H$-function of an L-space link in terms of the multivariable Alexander polynomials of $L$ and its sublinks:

\begin{prop}[\cite{GorskyNemethiAlgebraicLinks}*{Theorem~2.10}] \label{prop:GNH-function}Let $L$ be an oriented L-space link in $S^3$. Then
	\[
	H_L(\ve{s})=\sum_{L'\subset L} (-1)^{|L'|-1} \sum_{\substack{\ve{s}'\in \bH(L') \\ \ve{s}'\ge \pi_{L,L'}(\ve{s}+\ve{1})}} \chi(\HFL^-(L',\ve{s}')).
	\]
\end{prop}

In the above, $\HFL^-(L')$ denotes the homology of $\cCFL(L')/(Z_i: L_i\subset L')$. Furthermore, $\pi_{L,L'}\colon \bH(L)\to \bH(L')$ is the function which omits each component associated to link components of $L\setminus L'$, and which modifies the $i$-th component (whenever $L_i\subset L'$) via the transformation
\[
s_i\mapsto s_i-\frac{\lk(L_i, L\setminus L')}{2}.
\]
Also, $\ve{1}$ denotes $(1,\dots, 1)$.

By \cite{OSLinks}*{Section~1.1}, $\chi(\HFL^-(L,\ve{s}))$ is the coefficient of the monomial $x_1^{s_1}\cdots x_n^{s_n} $ in a suitable normalization of the multivariable Alexander polynomial. More explicitly, Let ${\Delta}_L(x_1,\dots,x_n)$ be the symmetric multivariable Alexander polynomial of $L$. Let 

\begin{equation}
	\tilde{\Delta}_L(x_1,\dots,x_n) :=
	\begin{cases}
		(	x_1\cdots x_n)^{\frac{1}{2}}\Delta_L(x_1,\dots,x_n) & \text{if $n>1$}\\
		\Delta_L(x)/(1-x^{-1})& \text{if $n=1$}\\
	\end{cases}       
\label{eq:normalized Alex poly}
\end{equation}
denote the normalized multivariable Alexander polynomial, which satisfies the relation 
\begin{equation*}
	\tilde{\Delta}_{L}(x_1,\dots,x_n) = \sum_{\mathbf{s}\in\mathbb{H}(L)}\chi(\HFL^{-}(L,\mathbf{s}))x_1^{s_1}\cdots x_n^{s_n},
	\label{eq:euler_char_Alex}
\end{equation*}
where $\mathbf{s} = (s_1,\dots,s_n)$. For $n=1$, we expand $1/(1-x^{-1})$ as a power series in $x^{-1}$. There is an overall sign ambiguity of $\tilde{\Delta}_L$. For an L-space link $L$, this indeterminacy can be eliminated by assuming the $H$-function takes non-negative value, and differs by at most $1$ when moving by the unit vector in each direction.

In this paper, we focus mostly on 2-component L-space links $L = L_1 \cup L_2$. We will typically write $(t,r) \in \bH(L)$ for the Alexander gradings, where $t$ is the one associated to the component $L_1$, and $r$ is the one associated to the component $L_2$. Additionally, we will be most interested in  two-component $L$-space links of the form $L_P=\mu \cup P$, for a 2-component L-space satellite operator $P$.  We use the convention that $L_1=\mu$ and $L_2=P$.

We now recall several basic properties of the $H$-function $H_L(t,r)$ for a two-component link $L = L_1\cup L_2$.

\begin{lem}
\label{lem:properties-H-function}
Suppose that $L=L_1\cup L_2$ is a two-component link. Write $\ell=\lk(L_1,L_2)$ for the linking number of $L$, and write $H_{L_1}$ and $H_{L_2}$ for the $H$-functions of $L_1$ and $L_2$ respectively. Then the $H$-function $H_L(t,r)$ satisfies the following properties:
	\begin{enumerate}
		\item (Non-negativity) \[H_L(t,r)\ge 0.\]
		\item (Monotonicity) \[H_L(t,r)\ge H_L(t+1,r) ,\,\,\,H_L(t,r)\ge H_L(t,r+1). \]
		\item (Bounded gap) \[H_L(t,r)\le H_L(t+1,r)+1,\,\,\,H_L(t,r)\le H_L(t,r+1)+1. \]
		\item (Stabilization) \[\begin{split} H_{L_1}(t-\ell/2)&=H_L(t,\infty):=\lim_{r\to \infty}H_L(t,r) \\ H_{L_2}(r-\ell/2)&=H_L(\infty,r):=\lim_{t\to \infty}H_L(t,r). \end{split} \]
		\item (Symmetry)
		\[H_L(t,r)+t+r = H_L(-t,-r).\]
	\end{enumerate}
	
	\label{lem: properties of H function}
\end{lem}

Most of the statements in Lemma~\ref{lem:properties-H-function} follow directly from the definition of the $H$-function. Proofs of the above results can be found in, e.g., \cite{BorodzikGorskyImmersed}. Similar properties also hold for links with more components with appropriate changes.

\begin{rem}
	The ``bounded gap'' property is equivalent to the monotonicity of a version of the $H$-function which is defined using the grading $\gr_{\zs}$ instead of $\gr_{\ws}$. We denote this function
	\[
	\bar{H}_L(t,r):=	-\frac{1}{2} \max \{\gr_{\zs}(\xs): \xs\in \cHFL(L,t,r), \xs \text{ is $U$-non-torsion}\}
	\]
	and note that the symmetry of knot Floer homology implies that
	\[
	\bar{H}_L(t,r) = H_L(t,r)+t+r.
	\]
	Using this symmetry, we also obtain a stabilization result of the $H$-function as $t,r\to -\infty$ in the following sense:
	\[\bar{H}_{L_1}(t+\ell/2)=\lim_{r\to -\infty}\bar{H}_L(t,r),\,\,\bar{H}_{L_2}(r+\ell/2) = \lim_{t\to -\infty}\bar{H}_L(t,r). \]

	\label{rem:symmetry}
\end{rem}

\begin{example} In Figure~\ref{fig:h-function-examples}, we show the $H$-function of the two component unlink $U_2$, the positive Hopf link $\mathcal{H}_+$ (with linking number $1$), and the negative Hopf link $\mathcal{H}_-$ (with linking number $-1$). These will be important in Section~\ref{sec: non homomorphism} when we prove Theorem~\ref{thm:non homomorphism}. They have the normalized multivariable Alexander polynomials, normalized as in Equation (\ref{eq:normalized Alex poly}).
\[
	\tilde{\Delta}_{U_2}(x_1,x_2)=0,\quad 	  	\tilde{\Delta}_{\mathcal{H}_+}(x_1,x_2)=x_1^{\sfrac{1}{2}}x_2^{\sfrac{1}{2}}, \quad \tilde{\Delta}_{\mathcal{H}_-}(x_1,x_2)=-x_1^{\sfrac{1}{2}}x_2^{\sfrac{1}{2}}.
\]
\end{example}

\begin{figure}[h]
\[ 
\begin{array}{c}
	\begin{array}{|c|cccccccc|}
		\hline
		\hbox{\diagbox{$r$}{$t$}}&\cdots&-2&-1&0&1&2&\cdots&\\
		\hline	\vdots&\ddots&2&1&0&0&0&0&\\
		1&\cdots&2&1 &0&0&0 & 0& \\
		0&\cdots&2&1 &0&0&0 &0&\\
		-1&\cdots&3&2 &1&1&1&1&\\
		-2&\cdots&4&3 &2&2&2& 2&\\
		\vdots&\reflectbox{$\ddots$}&\vdots&\vdots &\vdots&\vdots&\vdots&\ddots &\,\\
		\hline
	\end{array}
	\\
	\overset{\phantom{A}}{\displaystyle{H_{U_2}(t,r)}}
	\end{array}
	\qquad 
	\begin{array}{c}
	\begin{array}{|c|cccccccc|}
			\hline
			\hbox{\diagbox{$r$}{$t$}}&\cdots&-\sfrac{3}{2}&-\sfrac{1}{2}&\sfrac{1}{2}&\sfrac{3}{2}&\sfrac{5}{2}&\cdots&\\
			\hline	\vdots&\ddots&2&1&0&0&0&0&\\
		\sfrac{3}{2} &\cdots&2&1 &0&0&0 & 0& \\
			\sfrac{1}{2} &\cdots&2&1 &0&0&0 &0&\\
			-\sfrac{1}{2} &\cdots&2&1 &1&1&1& 1&\\
			-\sfrac{3}{2} &\cdots &3&2 &2&2&2&2 & \\
			\vdots&\reflectbox{$\ddots$}&\vdots&\vdots&\vdots&\vdots&\vdots&\ddots &\,\\
			\hline
		\end{array}
		\\
		\overset{\phantom{A}}{\displaystyle{H_{\mathcal{H}_+}(t,r)}}
		\end{array}
\]
\vspace{1mm}
\[
\begin{array}{c}
\begin{array}{|c|cccccccc|}
			\hline
			\hbox{\diagbox{$r$}{$t$}}&\cdots&-\sfrac{5}{2}&-\sfrac{3}{2}&-\sfrac{1}{2}&\sfrac{1}{2}&\sfrac{3}{2}&\cdots&\\
			\hline	\vdots&\ddots&2&1&0&0&0&0&\\
			\sfrac{1}{2} &\cdots&2&1 &0&0&0 &0&\\
			-\sfrac{1}{2} &\cdots&3&2 &1&0&0& 0&\\
			-\sfrac{3}{2} & \cdots&4&3&2&1&1&1&\\
			-\sfrac{5}{2} &\cdots &5&4&3&2&2&2 & \\
			\vdots&\reflectbox{$\ddots$}&\vdots&\vdots&\vdots&\vdots&\vdots&\ddots &\,\\
			\hline
		\end{array}
		\\
			\overset{\phantom{A}}{\displaystyle{H_{\mathcal{H}_-}(t,r)}}
		\end{array}
\]
\caption{The $H$-functions of the 2-component unlink $U_2$, the positive Hopf link $\mathcal{H}_+$ and the negative Hopf link $\mathcal{H}_-$.}
\label{fig:h-function-examples}
\end{figure}

\section{Background on link surgery modules}
\label{sec:surgery}

In this section, we recall some background on the knot and link surgery modules described in \cite{ZemBordered,ZemExact}.

\subsection{The surgery algebra}

We now recall the surgery algebra from \cite{ZemBordered}.

\begin{define}
The \textit{surgery algebra} $\cK$ is an associative algebra over the ring of two idempotents $\ve{I}=\ve{I}_0\oplus \ve{I}_1$, where each $\ve{I}_i=\bF=\Z/2$. Furthermore
\[
\ve{I}_0\cdot \cK\cdot \ve{I}_0=\bF[W,Z],\quad \ve{I}_0 \cdot \cK \cdot \ve{I}_1=0
\]
\[
\ve{I}_1\cdot \cK\cdot \ve{I}_0=\bF[U,T,T^{-1}] \otimes \langle \sigma, \tau\rangle \quad \text{and}, \quad \ve{I}_1\cdot \cK\cdot \ve{I}_1=\bF[U,T,T^{-1}].
\]
The algebra is subject to the relations that
\[
\sigma W=UT^{-1} \sigma, \quad \sigma Z=T \sigma,\quad \tau W=T^{-1} \tau,\quad \tau  Z=UT  \tau.
\]
\end{define}

\subsection{Type-$D$, $A$, $DA$ and $AA$-modules}

In this section, we briefly recall the formalism of type-$D$, $A$ and $DA$ modules of Lipshitz, Ozsv\'{a}th and Thurston \cite{LOTBordered}.

Let $\cA$ be an associative algebra over a ring $\ve{k}$. A \emph{right type-$D$ module}, denoted $\cX^{\cA}=(\cX, \delta^1)$, consists of a right $\ve{k}$-module, equipped with a $\ve{k}$-linear map
\[
\delta^1\colon \cX\to \cX\otimes_{\ve{k}} \cA
\]
which satisfies the equation
\[
(\bI_{\cX}\otimes \mu_2) \circ (\delta^1\otimes \bI_{\cA})\circ \delta^1=0,
\]
where $\mu_2$ denotes the multiplication action on $\cA$.

A \emph{left type-$A$ module} over $\cA$, denoted ${}_{\cA} \cX=(\cX, m_{j+1})$, is the same as a left $A_\infty$-module over $\cA$. That is, $\cX$ is a left $\ve{k}$-module and
\[
m_{j+1}\colon \underbrace{\cA\otimes \cdots \otimes \cA}_{j}\otimes \cX\to \cX
\]
is a collection of $\ve{k}$-linear maps, ranging over $j\ge 0$, such that
\[
\sum_{i=0}^n m_{n-i}(a_n,\dots a_{i+1}, m_{i+1}(a_{i},\dots  , a_1, \xs))+\sum_{i=1}^{n-1}m(a_n,\dots,  a_{i+1}a_i,\dots, a_1,\xs)=0.
\]

Finally, if $\cA$ and $\cB$ are associative algebras over $\ve{k}$ and $\ve{j}$, respectively, then a \emph{$DA$-bimodule} ${}_{\cA}\cX^{\cB}$ consists of a $(\ve{k}, \ve{j})$-bimodule $\cX$, equipped with $(\ve{k},\ve{j})$-linear maps
\[
\delta_{j+1}^1\colon \underbrace{\cA\otimes \cdots \otimes \cA}_{j} \otimes \cX\to \cX\otimes \cB.
\]
These maps are required to satisfy the following associativity relation:
\[
\begin{split}
0=&\sum_{i=0}^n \bigg((\bI_{\cX}\otimes \mu_2)\circ (\delta^1_{n-i}\otimes \bI_{\cB})(a_n,\dots a_{i+1},-)\bigg)\bigg(  \delta_{i+1}^1(a_{i},\cdots   a_1, \xs)\bigg)\\+
&\sum_{i=1}^{n-1}\delta_{n}^1(a_n,\dots,  a_{i+1}a_i,\dots, a_1,\xs).
\end{split}
\]

We refer the reader to \cite{LOTBordered} and \cite{LOTBimodules} for more background on the categories of these modules and bimodules. 

If $\cX^\cA$ and ${}_{\cA} \cY$ are type-$D$ and type-$A$ modules, then Lipshitz, Ozsv\'ath and Thurston \cite{LOTBordered}*{Section~2.4} define a model for the derived tensor product of $\cX$ and $\cY$, called the \emph{box tensor product}, denoted
\[
\cX^{\cA}\boxtimes {}_{\cA} \cY.
\]
If $\cA$ is an algebra over $\ve{i}$, then the underlying group of the box tensor product is $\cX\otimes_{\ve{i}} \cY$. The differential is given by
\[
\d^{\boxtimes}(\xs\otimes \ys)=\sum_{k=0}^\infty(\bI_{\cX}\otimes m_{k+1})(\delta^k(\xs)\otimes \ys)
\]
where $\delta^k\colon \cX\to \cX\otimes \cA^{\otimes k}$ is obtained by iterating $\delta^1$ $k$-times.

Finally we discuss $AA$-bimodules. If $\cA$ and $\cB$ are associative algebras over $\ve{i}$ and $\ve{j}$, an \emph{$AA$-bimodule} ${}_{\cA} M_{\cB}$ consists of a $(\ve{i},\ve{j})$-module $M$, equipped with structure maps 
\[
m_{i|1|j}\colon \cA^{\otimes i}\otimes M\otimes \cB^{\otimes j}\to M
\]
which satisfy the following $A_\infty$-associativity relation:
\[
\begin{split}
0=&\sum_{\substack{0\le i\le n\\ 0\le j\le m}} m_{n-i|1|m-j}(a_n,\dots, m_{i|1|j}(a_i,\dots, a_1, \xs,b_1,\dots, b_j),\dots, b_m)\\
+&\sum_{i=1}^{n-1} m_{n-1|1|m}(a_n,\dots, a_{i+1}a_i,\dots, a_1, \xs,b_1,\dots, b_m)\\
+&\sum_{i=1}^{m-1} m_{n|1|m-1}(a_n,\dots, a_1, \xs, b_1,\dots, b_ib_{i+1},\dots, b_m) 
\end{split}.
\]
See \cite{LOTBimodules}*{Defintion~2.2.38} for more on type-$AA$ bimodules.

We will also need a version of the homological perturbation lemma for $AA$-bimodules. We first recall that if $Z$ and $M$ are chain complexes, then we say that a diagram of maps
\[
\begin{tikzcd}
 M \ar[loop left, "H"] \ar[r, "\Pi", shift left] & Z \ar[l, "I",shift left]
\end{tikzcd}
\]
is a \emph{strong deformation retraction} of chain complexes from $M$ onto $Z$ if \[
\d(\Pi)=0, \quad\d(I)=0,\quad \Pi \circ I=\bI_Z,\quad I\circ \Pi=\bI_M+\d(H),
\]
\[
\Pi \circ H=0,\quad H\circ I=0,\quad \text{and} \quad H\circ H=0.
\]
Here, we are writing $\d(f)$ for the morphism differential, i.e. $\d(f)=f\circ \d+\d\circ f$. In particular, $\d(\Pi)=0$ is equivalent to $\Pi$ being a chain map.

We now state the homological perturbation lemma for $AA$-bimodules. 
\begin{lem}\label{lem:HPL-AA} Suppose that ${}_{\cA} M_{\cB}$ is a $AA$-bimodule and that $(Z,d)$ is a chain complex. Suppose that the diagram of maps
\[
\begin{tikzcd}
 M \ar[loop left, "H"] \ar[r, "\Pi", shift left] & Z \ar[l, "I",shift left]
\end{tikzcd}
\]
giving a strong deformation retraction of chain complexes from $M$ onto $Z$. Then $Z$ admits the structure of an $AA$-bimodule, denoted ${}_{\cA} Z_{\cB}$, so that $m_{0|1|0}=d$, and so that ${}_{\cA} Z_{\cB}\simeq {}_{\cA} M_{\cB}$. 
\end{lem}

The proof of the above is standard, see, e.g., \cite{OSBorderedKauffmanStates}*{Lemma~2.11}. The structure maps $m^Z_{i|1|j}$ on ${}_{\cA} Z_{\cB}$ have a convenient description. If $i+j>0$, $\ve{a}\in \cA^{\otimes i}$, $\ve{b}\in \cB^{\otimes j}$ and $\xs\in M$, then the structure map is given by a diagram of the following form:
\begin{equation}
m_{i|1|j}^Z(\ve{a},\ve{x},\ve{b})= \begin{tikzcd}[column sep=.3cm, row sep=.2cm]
 \ve{a}\ar[ddr,Rightarrow]& \ve{x} \ar[d]& \ve{b} \ar[ddl,Rightarrow] \\
& I\ar[d] \\
&m\ar[d] \\
&\Pi \ar[dd]\\
&\\
&\,
\end{tikzcd}
+
\begin{tikzcd}[column sep=.3cm, row sep=.2cm] \ve{a} \ar[d, Rightarrow] & \ve{x} \ar[d]& \ve{b}\ar[d, Rightarrow] \\
\Delta\ar[dr,Rightarrow] \ar[dddr,Rightarrow]& I \ar[d]&\Delta \ar[dl,Rightarrow] \ar[dddl,Rightarrow] \\
&m \ar[d] \\
&H\ar[d]\\
&m\ar[d]\\
&\Pi\ar[dd]\\
\\
&\,
\end{tikzcd}+
\begin{tikzcd}[column sep=.3cm, row sep=.2cm] \ve{a}\ar[d,Rightarrow] & \ve{x}\ar[d]& \ve{b} \ar[d,Rightarrow] \\
\Delta \ar[dr,Rightarrow] \ar[dddr,Rightarrow] \ar[dddddr,Rightarrow,bend right=15]& I \ar[d]&\Delta \ar[dl,Rightarrow] \ar[dddl,Rightarrow] \ar[dddddl,Rightarrow,bend left=15] \\
&m\ar[d] \\
&H\ar[d]\\
&m\ar[d]\\
&H\ar[d]\\
&m\ar[d]\\
&\Pi\ar[dd]\\
\\
&\,
\end{tikzcd}
+\cdots
\label{eq:HPL_structure_maps}
\end{equation}
In the above, $\Delta$ is the map which splits $\cA^{\otimes i}$ and $\cB^{\otimes j}$ into its tensor factors. The above sum ranges over configurations where each $m$-labeled vertex has at least two inputs (i.e. $m_{0|1|0}$ is never applied).

\subsection{Modules over the surgery algebra}

We now recall some basic constructions from \cite{ZemBordered}. If $L\subset Y$ is an $n$-component link with Morse framing $\Lambda$, then there is a type-$D$ module
\[
\cX_\Lambda(L)^{\otimes^n \cK}. 
\]
This type-$D$ module encodes the data of link surgery complex for $L$. It is sometimes helpful also to turn some of the type-$D$ actions into type-$A$ actions. This is achieved by taking a suitable tensor product with an algebraically defined module ${}_{\cK\otimes \cK} [\bI^{\Supset}]$, defined in \cite{ZemBordered}*{Section~8.4}.  

Of particular importance for our purpose are the cases that $L$ is a knot or a two-component link in $S^3$. When $K\subset S^3$ is a knot and $n\in \Z$, there is a type-$D$ module
\[
\cX_n(K)^{\cK}.
\]
When $L$ is a two-component link and $\Lambda \in \bZ^2$, there is a type-$DA$ bimodule
\[
{}_{\cK} \cX_{\Lambda}(L)^{\cK}.
\]

In \cite{ZemBordered}*{Theorem~12.1}, the second author proves that the box tensor product of two surgery modules is homotopy equivalent to the surgery complex of the connected sum of the corresponding links. As a special case, if $L=L_1\cup L_2\subset S^3$ is a two-component link with integral framing $\Lambda =(\lambda_1,\lambda_2)$ and $K\subset S^3$ is a knot with framing $n\in \Z$, then
\[
\cX_n(K)^{\cK}\boxtimes {}_{\cK} \cX_{(\lambda_1,\lambda_2)}(L)^{\cK}\simeq \cX_{(\lambda_1+n, \lambda_2)}(K\# L)^\cK. 
\]
Here the connected sum is formed by joining $K$ with $L_1$. 

\subsection{Hat flavored modules}
\label{sec:hat-modules}

In this section, we describe the modules over the hat flavor of the surgery algebra $\hat{\cK}:=\cK/(U)$, focusing on the case of knots in $S^3$.

There is an algebra morphism $q_U\colon \cK\to \cK/(U)$ induced by setting $U=0$, and therefore an induced $DA$-bimodule ${}_{\cK} [q_U]^{\hat{\cK}}$. See \cite{LOTBimodules}*{Definition~2.2.48} for the $DA$-bimodule induced by an algebra morphism. If $K\subset Y$ is a null-homologous knot with integral framing $n$, the hat flavored modules are defined via a box tensor product
\[
\cX_n(Y,K)^{\hat{\cK}}:=\cX_{n}(Y,K)^{\cK}\boxtimes {}_{\cK}[q_U]^{\hat{\cK}}. 
\]

It will be helpful to have some strong results about the structure of $\cX_n(K)^{\hat{\cK}}$ when $K$ is a knot in $S^3$. As a first step, we note the entire complex $\cX_n(K)^{\hat{\cK}}$ is determined by the knot Floer complex $\cCFK(K)^{\hat{R}}$ over $\hat{R}=\bF[W,Z]/(U)$. To see this, we first observe that idempotent $0$ of $\cX_n(K)^{\hat{\cK}}$ is just $\cCFK(K)^{\hat{R}}$ and idempotent $1$ of the module consists of a single generator $\ys'$ with vanishing $\delta^1$. The $\sigma$ weighted terms of the differential $\delta^1$ from idempotent $0$ to idempotent $1$ are induced by the $v$-map in the surgery formula. This map is induced by a homotopy equivalence from $Z^{-1} \cCFK(K)$ to $\bF[T,T^{-1}] \otimes \widehat{\CF}(S^3)\simeq \bF[T,T^{-1}]$. This homotopy can be normalized to preserve the Alexander grading, if we equip $T^i\in \bF[T,T^{-1}]$ with Alexander grading $i$. This homotopy equivalence intertwines the action of $Z$  with $T$ and intertwines $W$ with $0$, by \cite{ZemBordered}*{Lemma~6.7}. Since $\bF[T,T^{-1}]$ admits no grading preserving automorphisms other than the identity map, the map $v$ is therefore uniquely determined up to chain homotopy. Similarly, the $\tau$ weighted terms of $\delta^1$ from idempotent $0$ to idempotent $1$ are determined by a homotopy equivalence $h$ from $W^{-1} \cCFK(K)$ to $\bF[T,T^{-1}]$ which intertwines  $W$ with $T^{-1}$ and intertwines the $Z$ with $0$. The map $h$ shifts the Alexander grading upward by $n$. Such a map is uniquely determined up to chain homotopy. In particular, when $K\subset S^3$, the type-$D$ module $\cX_n(K)^{\hat{\cK}}$ is determined up to homotopy equivalence by $\cCFK(K)^{\hat{R}}$. See \cite{HMSZ_Naturality}*{Section~5.3} for a similar statement.

We now recall the structure theorem of Dai, Hom, Stoffregen and Truong \cite{DHSTmore} for $\cCFK(K)^{\hat{R}}$, when $K\subset S^3$. We recall first their notion of a \emph{standard complex}.  If $m$ is even and $b_1,\dots, b_m$ are non-zero integers, then they define a free chain complex $C(b_1,\dots, b_m)$ over $\hat{R}$. We view this complex as a type-$D$ module over $\hat{R}$. This complex has $m+1$ generators, which we denote 
\[
\ys_0,\dots, \ys_{m}. 
\]
The differential is as follows.
\begin{enumerate}
\item If $i$ is odd and $b_i<0$, then $\d$ contains a component from $\ys_{i-1}$ to $W^{|b_i|} \ys_i$. 
\item If $i$ is odd and $b_i>0$, then $\d$ contains a component from $\ys_i$ to $W^{b_i} \ys_{i-1}$.
\item If $i$ is even and $b_i<0$, then $\d$ contains a component from $\ys_{i-1}$ to $Z^{|b_i|} \ys_{i}$.
\item If $i$ is even and $b_i>0$, then $\d$ contains a component from $\ys_i$ to $Z^{b_i} \ys_{i-1}$.
\end{enumerate}
Schematically, we can depict the standard complex $C(b_1,\dots, b_m)$ as follows:
\[
\begin{tikzcd}[column sep=1.5cm]
\ys_0 \ar[r, leftrightarrow, "W^{|b_1|}"] 
& \ys_1 \ar[r, leftrightarrow, "Z^{|b_2|}"] & \ys_2 \ar[r, leftrightarrow]&\cdots \ar[r, leftrightarrow, "Z^{|b_m|}"]& \ys_m.
\end{tikzcd}
\]
The direction of each arrow is determined by the sign of $b_i$. If $b_i<0$, then the arrow points to the right. If $b_i>0$, then the arrow points to the left. 

The gradings are normalized so that $\gr_{\ws}(\ys_0)=0$ and $\gr_{\zs}(\ys_m)=0$. If $C(b_1,\dots, b_m)$ is a direct summand of $\cCFK(K)^{\hat{R}}$ for $K\subset S^3$, then $\gr_{\zs}(\ys_0)=-2\tau(K)$, and  $\gr_{\ws}(\ys_m)=-2\tau(K)$. Here, $\tau(K)$ denotes the invariant of Ozsv\'{a}th and Szab\'{o} \cite{OS4ballgenus}, as recalled in Section \ref{sec:concordance invariants}.

Dai, Hom, Stoffregen and Truong prove the following structural result for $\cCFK(K)^{\hat{R}}$:

\begin{thm}[\cite{DHSTmore}*{Theorem~6.1, Corollary~6.2}]\label{thm:standard-complex} If $K$ is a knot in $S^3$, then there are integers $b_1,\dots, b_m$ so that
\[
\cCFK(K)^{\hat{R}}\simeq C(b_1,\dots, b_m)\oplus A,
\] 
where $A$ is a free complex over $\hat{R}$ which becomes acyclic after inverting $W$, and after inverting $Z$.
\end{thm}

\begin{rem}
	\label{rem:epsilon}
	By \cite{DHSTmore}*{Remark~3.16, Proposition~4.10}, if $\cCFK(K)^{\hat{R}}\simeq C(b_0,...,b_m)\oplus A$, then the $\veps$-invariant of $K$, recalled in Section \ref{sec:concordance invariants}, is given by the following:
	\[\veps(K) = \begin{cases}
		0, &\text{ if } m=0,\\
		\sign(b_0) = -\sign(b_m), &\text{ if } m > 0.\\
	\end{cases}\]
\end{rem}

One easy corollary of Theorem~\ref{thm:standard-complex} is a structural result for the modules $\cX_n(K)^{\hat{\cK}}$ when $K\subset S^3$. Given non-zero integers $b_1,\dots, b_m$, we can define a type-$D$ module $\cX_n(b_1,\dots, b_m)^{\hat{\cK}}$ as follows. In idempotent 0, we use the complex $C(b_1,\dots, b_m)$. In idempotent 1, we use a complex with one generator, $\ys'$, which has Maslov and Alexander grading $(\gr,A)(\ys')=(0,0)$. We add two components to $\delta^1$: a component of $ \ys'\otimes T^{\tau(K)}\sigma$ to $\delta^1(\ys_0)$, and a component of $\ys'\otimes T^{n-\tau(K)} \tau$ to $\delta^1(\ys_m)$. Schematically, the complex $\cX_n(b_1,\dots ,b_m)^{\hat{\cK}}$ has the following form:
\[
\begin{tikzcd}[column sep=1.5cm]
\ys_0 \ar[drr, "T^{\tau(K)} \sigma",swap] \ar[r, leftrightarrow, "W^{|b_1|}"] 
& \ys_1 \ar[r, leftrightarrow, "Z^{|b_2|}"] & \ys_2 \ar[r, leftrightarrow]&\cdots \ar[r, leftrightarrow, "Z^{|b_m|}"]& \ys_m
\ar[dll, "T^{n-\tau(K)} \tau"]
\\
&& \ys'
\end{tikzcd}
\] There are no components of $\delta^1$ from $A$ to $\ys'$, as it becomes acyclic after inverting $W$ or $Z$.

As an immediate consequence of Theorem~\ref{thm:standard-complex}, and the above description of how to compute $\cX_n(K)^{\hat{\cK}}$ in terms of $\cCFK(K)^{\hat{R}}$, we obtain the following description of $\cX_n(K)^{\hat{\cK}}$ for knots in $S^3$:

\begin{thm}\label{thm:standard-cx-over-K} Suppose that $K\subset S^3$ and that $\cCFK(K)^{\hat{R}}$ decomposes as 
\[
\cCFK(K)^{\hat{R}}\simeq C(b_1,\dots, b_m)\oplus A,
\]
as in the statement of Theorem~\ref{thm:standard-complex}. Then
\[
\cX_n(K)^{\hat{\cK}}\simeq \cX_n(b_1,\dots, b_m)^{\hat{\cK}}\oplus A,
\]
where we view $A$ as being supported in idempotent 0.
\end{thm}

\begin{rem} Sometimes in computations, we will find it easier to replace $\ys'$ with $\ys'\otimes T^{\tau(K)}$, so that $\cX_n(K)^{\hat{\cK}}$ has the following form:
\[
\begin{tikzcd}[column sep=1.5cm]
\ys_0 \ar[drr, "\sigma",swap] \ar[r, leftrightarrow, "W^{|b_1|}"] 
& \ys_1 \ar[r, leftrightarrow, "Z^{|b_2|}"] & \ys_2 \ar[r, leftrightarrow]&\cdots \ar[r, leftrightarrow, "Z^{|b_m|}"]& \ys_m
\ar[dll, "T^{n-2\tau(K)} \tau"]
\\
&& \ys'
\end{tikzcd}
\]
\end{rem}

\section{$DA$ bimodules for L-space satellite operators}
\label{sec: DA bimodule}

If  $P$ is a satellite operator, we write $L_P := \mu \cup P$ for the corresponding two-component link in $S^3$. We typically abuse notation and view $P$ as both a knot in the solid torus, as well as the knot in $S^3$ obtained by embedding the solid torus $S^1\times D^2$ into $S^3$  as a neighborhood of the $0$-framed unknot. We write $\mu$ for the meridian $\{pt\}\times \partial D^2$. 

\begin{define}
We say that $P$ is an \textit{L-space satellite operator} if $L_P\subset S^3$ is a two-component L-space link, i.e., $S^3_{\Lambda}(L_P)$ is an L-space for all integral framings $\Lambda \gg 0$. 
\end{define}

In this section, we review the computation of the bimodules ${}_{\cK} \cX(L_P)^{\bF[W,Z]}$ from \cite{CZZ}. 
 To study the $\tau$ invariant of satellites, we consider a version of this bimodule obtained by setting $W=0$ and changing the type-$D$ action to a type-$A$ action. We denote this type-$AA$ bimodule by
\[
{}_{\cK} \cX(L_P)_{\bF[Z]}.
\]
This module admits a model which has only $m_{1|1|0}$ and $m_{0|1|1}$ non-zero, though is typically not free as an $\bF[Z]$-module. We will write
\[
{}_{\cK} \cX^{\diamond}(L_P)_{\bF[Z]}
\]
for this small model of ${}_{\cK} \cX(L_P)_{\bF[Z]}$. 

Of particular importance will be a quotient $AA$-bimodule of ${}_{\cK} \cX^{\diamond}(L_P)_{\bF[Z]}$, which we denote by
\[
{}_{\cK} \cX^{\free}(L_P)_{\bF[Z]}.
\]
This module is simpler than ${}_{\cK} \cX^{\diamond}(L_P)_{\bF[Z]}$, and is always free as an $\bF[Z]$-module.
In many cases it is sufficient for the purposes of computing $\tau$ of a satellite knot.

\subsection{The bimodule ${}_{\cK} \cX(L_P)^{\bF[W,Z]}$}
\label{sec:bimodule-X-L_P}

In \cite{CZZ}, we described a method for computing the bimodule ${}_{\cK} \cX(L)^{\bF[W,Z]}$ when $L$ is a two-component L-space link. Together with the gluing formulas from \cite{ZemBordered}, this gave a technique for computing the effect of L-space satellite operators on knot Floer homology.

 First, we recall the definition of staircase complexes.

\begin{define}
	A \textit{staircase} is a finitely generated free chain complex over $R=\bF[W,Z]$ with a free $R$-basis $\xs_0,\xs_1,\xs_2,\dots, \xs_{2m-1},\xs_{2m}$, with respect to which the differential takes the following form:
	\[
	\d(\xs_{2i})=0,\quad \text{and} \quad \d(\xs_{2i+1})=\xs_{2i}\otimes W^{\a_{i}}+\xs_{2i+2}\otimes Z^{\b_{i}},
	\]
	for some $\a_{i},\b_{i}>0$. 
\end{define}
The $WZ=0$ reduction of a staircase complex is a special case of the standard complex described in Section~\ref{sec:hat-modules}, which takes the form $C(\alpha_0,-\beta_0,\alpha_1,-\beta_1,\dots ,\alpha_{m-1},-\beta_{m-1})$ in the notation used there.

	We refer to the generator $\xs_{0}$ in a staircase complex as the \textit{top generator}, as it has the largest Alexander grading.
	
\begin{define}\label{def:standard-gradings}	We say that a staircase complex $\cS$ has \emph{standard gradings} if  $\gr_{\ws}(\xs_{0})=0$ and $\gr_{\zs}(\xs_{2m})=0$. This is the case whenever $\cS$ is the knot Floer complex $\cCFK(K)$ of an L-space knot $K$ in $S^3$. 
	\end{define}
	
	There is also another grading on a staircase complex, denoted $\gr_{\alg}$, which has the property that 
	\[
	\gr_{\alg}(\xs_{2i})=0,\qquad \gr_{\alg}(\xs_{2i+1})=1,
	\]
	and $\gr_{\alg}(W)=\gr_{\alg}(Z)=0$. The grading $\gr_{\alg}$ is obtained by viewing the staircase complex as a two-step free resolution of its homology.

When $L_P$ is an L-space link, we can encode the data of ${}_{\cK} \cX(L_P)^{\bF[W,Z]}$ in terms of the diagram of the form shown in Figure~\ref{fig:DA-bimodule-schematic}. Therein, $t\in\Z+\lk(\mu,P)/2$. Each of the complexes $\cC_t$ and $\cS$  are staircase complexes. Each $T^{t} \cS$ denotes a copy of $\cS$. The staircase complexes $\cC_t$ and $\cS$ are determined by the $H$-function of $L$. 

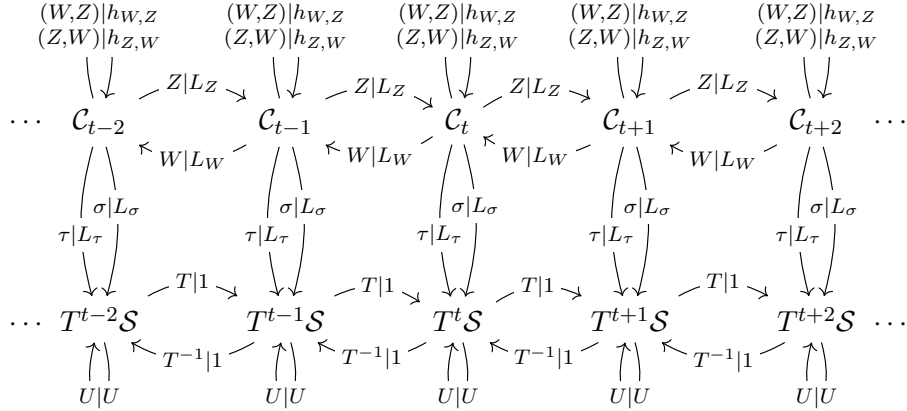
\begin{figure}[h]
\begin{equation*}
	\begin{tikzcd}[labels=description, column sep=1.1cm, row sep=0cm]
		\cdots
		&[-1.2cm]\cC_{t-2}
		\ar[r, bend left, "Z|L_Z"]
		\ar[d,"\sigma|L_\sigma", bend left=15, pos=.4]
		\ar[d, "\tau|L_\tau", bend right=15, pos=.6]
		\ar[loop above, looseness=20, "\substack{(W,Z)|h_{W,Z}\\(Z,W)|h_{Z,W}}"]
		& \cC_{t-1}
		\ar[r, bend left, "Z|L_Z"]
		\ar[l, bend left, "W|L_W"]
		\ar[d,"\sigma|L_\sigma", bend left=15, pos=.4]
		\ar[d, "\tau|L_\tau", bend right=15, pos=.6]
		\ar[loop above, looseness=20, "\substack{(W,Z)|h_{W,Z}\\(Z,W)|h_{Z,W}}"]
		&\cC_{t}
		\ar[r, bend left, "Z|L_Z"]
		\ar[l, bend left, "W|L_W"]
		\ar[d,"\sigma|L_\sigma", bend left=15, pos=.4]
		\ar[d, "\tau|L_\tau", bend right=15, pos=.6]
		\ar[loop above, looseness=20, "\substack{(W,Z)|h_{W,Z}\\(Z,W)|h_{Z,W}}"]
		& \cC_{t+1}
		\ar[r, bend left, "Z|L_Z"]
		\ar[l, bend left, "W|L_W"]
		\ar[d,"\sigma|L_\sigma", bend left=15, pos=.4]
		\ar[d, "\tau|L_\tau", bend right=15, pos=.6]
		\ar[loop above, looseness=20, "\substack{(W,Z)|h_{W,Z}\\(Z,W)|h_{Z,W}}"]
		& \cC_{t+2}
		\ar[l, bend left, "W|L_W"]
		\ar[d,"\sigma|L_\sigma", bend left=15, pos=.4]
		\ar[d, "\tau|L_\tau", bend right=15, pos=.6]
		\ar[loop above, looseness=20, "\substack{(W,Z)|h_{W,Z}\\(Z,W)|h_{Z,W}}"]
		&[-1.2cm] \cdots
		\\[2cm]
		\cdots
		& T^{t-2}\cS
		\ar[r, bend left, "T|1"]
		\ar[loop below,looseness=20, "U|U"]
		& T^{t-1}\cS
		\ar[r, bend left, "T|1"]
		\ar[l, bend left, "T^{-1}|1"]
		\ar[loop below,looseness=20, "U|U"]
		&T^{t}\cS
		\ar[r, bend left, "T|1"]
		\ar[l, bend left, "T^{-1}|1"]
		\ar[loop below,looseness=20, "U|U"]
		&T^{t+1}\cS
		\ar[r, bend left, "T|1"]
		\ar[l, bend left, "T^{-1}|1"]
		\ar[loop below,looseness=20, "U|U"]
		&T^{t+2}\cS
		\ar[l, bend left, "T^{-1}|1"]
		\ar[loop below,looseness=20, "U|U"]
		&\cdots 
	\end{tikzcd}
\end{equation*}
\caption{A schematic encoding the $DA$-bimodule of an L-space link $L$. Additional data (not shown) necessary to encode the $DA$ module are the actions $\delta_3^1(\sigma,W,-)$, $\delta_3^1(\sigma,Z,-)$, $\delta_3^1(\tau,W,-)$ and $\delta_3^1(\tau,Z,-)$.}
\label{fig:DA-bimodule-schematic}
\end{figure}

The diagram in Figure~\ref{fig:DA-bimodule-schematic} displays some of the $\delta_2^1$ and $\delta_3^1$ actions on this bimodule. For example, the arrows labeled $Z|L_Z$ indicate the action of $\delta_2^1(Z,-)$. We will write $L_Z$ for the induced map from $\cC_{t}$ to $\cC_{t+1}$. The arrows labeled $(W,Z)|h_{W,Z}$ indicate the actions $\delta_3^1(W,Z,-)$. Additionally, there will typically be several more $\delta_3^1$ actions which are not shown in the above diagram. These would take the form of extra arrows labeled as $(\sigma,W)|h_{\sigma,W}$, $(\sigma, Z)|h_{\sigma,Z}$, $(\tau,W)| h_{\tau,W}$ and $(\tau,Z)|h_{\tau,Z}$. 

The above data does not explicitly include a description of algebra elements which are not degree 1 monomials, e.g., $\delta_2^1(W Z^2,-)$. In the case of L-space links, the remaining actions on the bimodule ${}_{\cK} \cX(L_P)^R$ are determined (up to overall homotopy equivalence), by the above choices of maps $L_W$, $L_Z$, $L_\sigma$, $L_\tau$, $h_{W,Z}$, $h_{Z,W}$, $h_{\sigma,W}$,  $h_{\sigma, Z}$, $h_{\tau, W},$ and $h_{\tau,Z}$. Furthermore, it is possible to choose the resulting bimodule ${}_{\cK} \cX(L_P)^R$ so that the actions are $U$-equivariant. See \cite{CZZ}*{Section~5.5} for further details. 

The staircase $\cS$ can also be identified with the knot Floer complex $\cCFK(P)$ of the pattern $P$ viewed as a knot in $S^3$. Here, $\cS$ is equipped with standard $(\gr_{\ws},\gr_{\zs})$ bigradings, in the sense of Definition~\ref{def:standard-gradings}.  We denote the generators of $\cS$ by $\xs'_{0},\xs'_1,\dots,\xs'_{2m'}$.

The staircase $\cC_{t}$ is determined by the $H$-function of $L$, in the following way. We consider the restriction of the $H$-function, $H_L(t,-)$. We think of this as a single column of the $H$-function, if we plot the function $H_L(t,r)$ as in Figure~\ref{fig:Whitehead link}. For each $t$, there are finitely many  $r$ such that \[
H_{L}(t,r+1) = H_{L}(t,r),\,\, \text{ and }\,\,H_{L}(t,r-1) = H_{L}(t,r)+1. 
\]
Label these values of $r$ by $r^t_0,r_1^{t},\dots,r^t_{m_t}$. For each $i=0,\dots, m_t$, we add a generator $\xs_{2i}^t$ to $\cC_t$ such that 
\[\gr_{\ws}(\xs^t_{2i}) = -2 H_L(t,r^t_i).\]
The $\gr_{\zs}$-grading on link Floer homology satisfies $\gr_{\zs}=\gr_{\ws}-2A_1-2A_2$, so we define 
\[\gr_{\zs}(\xs^t_{2i}) = \gr_{\ws}(\xs^t_{2i})-2(A_1+A_2)=-2H_L(t,r^t_i)-2t-2r^t_i.
\]
For each $0\le i<m$, we add a generator $\xs_{2i+1}^t$, which satisfies 
\[
\gr_{\ws}(\xs^t_{2i+1}) = \gr_{\ws}(\xs^t_{2i+2})+1,\quad \text{and} \quad  \gr_{\zs}(\xs^t_{2i+1}) = \gr_{\zs}(\xs^t_{2i})+1.
\]

The grading shifts of the actions $L_{a}$ for $a=W,Z,\sigma,\tau$ are given as follows. When viewed as maps with domain $\cC_t$, they have the following gradings:
\begin{equation}
	 \begin{split}
	 	(\gr_{\ws},\gr_{\zs})(L_W) &= (-2,0)\\
	 	(\gr_{\ws},\gr_{\zs})(L_Z)&=(0,-2)\\
	 	(\gr_{\ws},\gr_{\zs})(L_\sigma)&=(0, 2t+\ell)\\
		(\gr_{\ws},\gr_{\zs})(L_\tau)&=(-2t+\ell,0).
\end{split}
 \label{eq:grading shifts of left action}
\end{equation}
A helpful consequence of the above formulas is that
\[
A(L_\sigma(\xs))=A_2(\xs)-\ell/2\quad \text{and} \quad A(L_\tau(\xs))=A_2(\xs)+\ell/2.
\]
In the above, $\ell = \lk(\mu,P)$, and we are writing $A_2(\xs)$ for second component in the $\bH(L_P)$-valued Alexander grading. Here, we give $T^t \cS$ the $(\gr_{\ws},\gr_{\zs})$ grading obtained by identifying it with $\cS\iso \cCFK(P)$. Also $A$ denotes the Alexander grading $\tfrac{1}{2}(\gr_{\ws}-\gr_{\zs})$. 

The maps $L_W$, $L_Z$, $L_\sigma$ and $L_\tau$ are determined up to chain homotopy by the fact that they have the above grading, they are chain maps with respect to $\delta_1^1$, and they are non-zero after taking homology with respect to $\delta_1^1$.

\begin{example}[The positively clasped Whitehead double]	

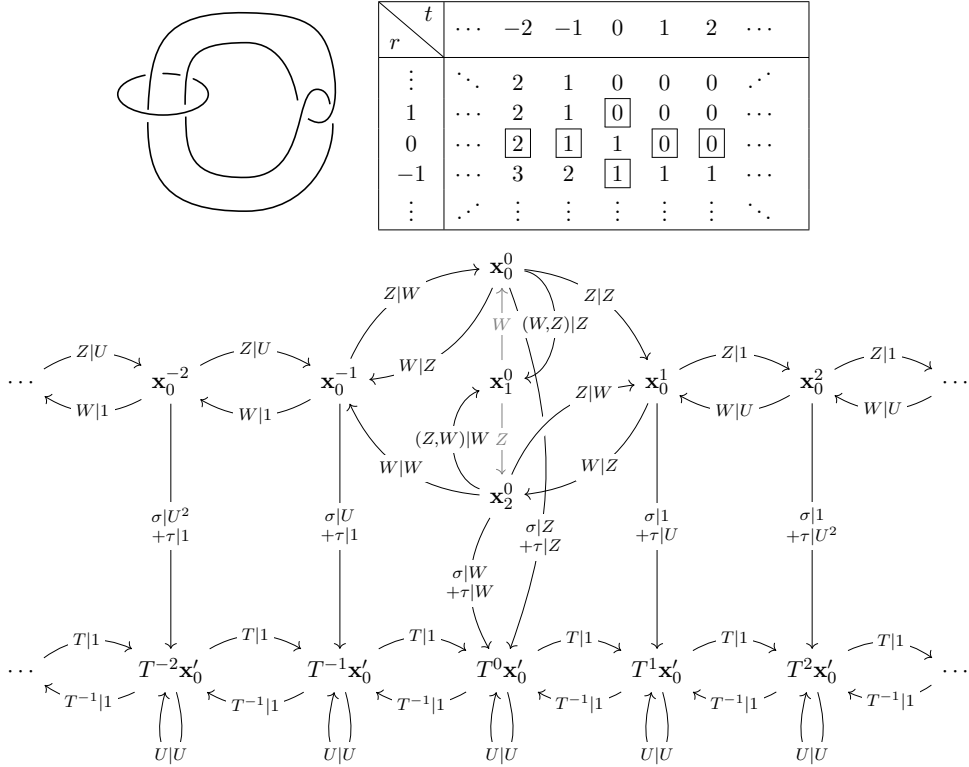
\begin{figure}[h!]
	\adjustbox{scale=0.8}{
		\begin{tabular}{m{4cm}p{8cm}}
			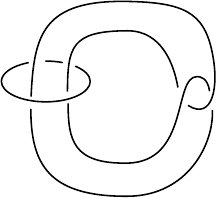
			&$ \begin{array}{|c|cccccccc|}
				\hline
				\hbox{\diagbox{$r$}{$t$}}&\cdots&-2&-1&0&1&2&\cdots&\\
				\hline	\vdots&\ddots&2&1&0&0&0&\reflectbox{$\ddots$}&\\
				1&\cdots&2&1 &\boxed{0}&0&0 & \cdots& \\
				0&\cdots&\boxed{2}&\boxed{1} &1&\boxed{0}&\boxed{0} &\cdots&\\
				-1&\cdots&3&2 &\boxed{1}&1&1& \cdots&\\
				\vdots&\reflectbox{$\ddots$}&\vdots&\vdots &\vdots&\vdots&\vdots&\ddots &\,\\
				\hline
			\end{array}
		$
		\end{tabular}
	}
	\[	
	\adjustbox{scale=0.8}{
	\begin{tikzcd}[labels=description, column sep=1.4cm]
			&&& \xs^0_0
			\ar[dl, bend left, "W|Z",pos=0.7]
			\ar[dr, "Z|Z", bend left]
			\ar[ddd,bend left=20, "\substack{\sigma|Z\\ +\tau|Z}",pos=.7]
			\ar[d,bend left =80, "{(W,Z)|Z}"]
			&&
			\\[.5cm]
			\cdots
			\ar[r,bend left, "Z|U"]
			& \xs^{-2}_0
			\ar[r, bend left, "Z|U"]
			\ar[l, bend left, "W|1"]
			\ar[dd,"\substack{\sigma|U^2\\ +\tau|1}"]
			&
			\xs^{-1}_0
			\ar[ur, bend left, "Z|W"]
			\ar[l, "W|1", bend left]
			\ar[dd, "\substack{\sigma|U \\ +\tau|1 }"]
			&
			\xs^0_1
			\ar[u,gray, "W"]
			\ar[d,gray, "Z"]
			&
			\xs^1_0
			\ar[r, bend left, "Z|1"]
			\ar[dl, "W|Z", bend left]
			\ar[dd, "\substack{\sigma|1 \\ +\tau|U}"]
			&
			\xs^2_0
			\ar[l, bend left, "W|U"]
			\ar[r,bend left, "Z|1"]
			\ar[dd, "\substack{\sigma|1\\+\tau|U^2}"]
			&
			\cdots
			\ar[l,bend left, "W|U"]
			\\[.5cm]
			&&&
			\xs^0_2
			\ar[ul, "W|W", bend left]
			\ar[ur, "Z|W", bend left,crossing over,pos = 0.75]
			\ar[d,bend right, "\substack{\sigma|W\\ +\tau|W}"]
			\ar[u,bend left =70, "{(Z,W)|W}"]
			&&
			\\[1.5cm]
			\cdots
			\ar[r, bend left, "T|1"]
			&T^{-2}\xs'_0
			\ar[r, bend left, "T|1"]
			\ar[l, bend left, "T^{-1}|1"]
			\ar[loop below,looseness=20, "U|U"]
			& T^{-1}\xs'_0
			\ar[r, bend left, "T|1"]
			\ar[l, bend left, "T^{-1}|1"]
			\ar[loop below,looseness=20, "U|U"]
			&T^{0}\xs'_0
			\ar[r, "T|1", bend left]
			\ar[l, "T^{-1}|1", bend left]
			\ar[loop below,looseness=20, "U|U"]
			&T^{1}\xs'_0
			\ar[r, bend left, "T|1"]
			\ar[l, bend left,"T^{-1}|1"]
			\ar[loop below,looseness=20, "U|U"]
			&T^{2}\xs'_0
			\ar[r, bend left, "T|1"]
			\ar[l, bend left,"T^{-1}|1"]
			\ar[loop below,looseness=20, "U|U"]
			&\cdots 
			\ar[l, bend left,"T^{-1}|1"]
	\end{tikzcd}
	}
	\]
	\caption{The $H$-function and $DA$-bimodule of the positively clasped Whitehead link $\Wh_+$. }
	\label{fig:Whitehead link}
\end{figure}
	
The normalized Alexander polynomial $\tilde{\Delta}_{\Wh}(x_1,x_2)$ is 
\[\tilde{\Delta}_{\Wh}(x_1,x_2) = -x_1x_2+x_1+x_2-1, \text{ with linking number  }\ell=0.\]See Figure \ref{fig:Whitehead link} for the $H$-function and the $DA$-bimodule of the positively clasped Whitehead link. Each box represents a generator $\xs^t_{2i}$. 

Each staircase $\cC_t$ for $t\neq 0$ consists of a single generator. For the staircase $\cC_0$, we have 
\begin{align*}
	\gr_{\ws}(\xs^0_0) = -2H(0,1)=0, \,\,\,&\gr_{\zs}(\xs^0_0) = -2H(0,1)-0-2 = -2,\\
	\gr_{\ws}(\xs^0_2) = -2H(0,-1) = -2,
	\,\,\,&\gr_{\zs}(\xs^0_2) = -2H(0,-1) -0 -(-2)=0,\\
	\gr_{\ws}(\xs^0_1) = \gr_{\ws}(\xs^0_2)+1 = -1,\,\,\, &\gr_{\zs}(\xs^0_1) = \gr_{\zs}(\xs^0_0)+1 =-1.
\end{align*}
The pattern $P$ is an unknot, so $\cS$ consists of a single generator $\xs'_0$, with 
\[
\gr_{\ws}(\xs'_0) = \gr_{\zs}(\xs'_0) =0.
\]

The actions $L_a$ for $a= W,Z,\sigma,\tau$ are computed using the fact that they are non-zero on homology, and using the grading shifts formulas in Equation~\eqref{eq:grading shifts of left action}. The arrow labeled $(W,Z)|Z$ from $\xs_{0}^0$ to $\xs_1^0$ represents a non-trivial $\delta^1_3$-action $\delta_{3}^1(W,Z,\xs_{0}^0) = \xs_1^0\otimes Z$, which is a chain homotopy between
\[L_W(L_Z(\xs_{0}^0))=\delta_2^1(W,\delta_2^1(Z,\xs_{0}^0)) = \xs^0_2\otimes Z^2\text{ and }\delta_2^1(WZ, \xs_{0}^0)  = \delta_2^1(U,\xs_{0}^0) = \xs_{0}^0\otimes U,\]
and similarly for the arrow labeled $(Z,W)|W$ from $\xs^0_2$ to $\xs^0_1$. The actions of the form $\delta_2^1(a,-)$ or $\delta_3^1(a,b,-)$ are determined by the above maps. Furthermore, the resulting maps can be chosen to be $U$-equivariant. See  \cite{CZZ}*{Section~5.5} for the details. 
\end{example}

\subsection{On the $H$-function of two-component L-space links}

In this section, we study the $H$-function of two-component L-space links. In particular, we prove some stabilization properties of the $H$-function in the horizontal and vertical directions, refining the corresponding statement in Lemma \ref{lem: properties of H function}. The main result is Lemma \ref{lem:shape of top generators}. We mostly focus on the case where $L$ is of the form $L=L_P$ for an L-space satellite operator $P$. Equivalently, we focus on L-space links $L=L_1\cup L_2$ when $L_1$ is an unknot.

In the horizontal direction, recall the following definition in \cite{CZZ}*{Section~10}:

\begin{define}
	If $L$ is a two-component L-space link with linking number $\ell$, we define the \textit{width} of $L$, denoted by $N_L$,  to be 
	 \[
	N_L := \min\left\{t_0\in \Z+\ell/2 \middle|  \begin{array}{l}  L_{\sigma}:\cC_{t} \to \cS, \text{ and } L_{\tau}:\cC_{-t} \to \cS \text{ are } \\    \text{homotopy equivalence for all } t\ge t_0\end{array} \right\}. 
	\]
	Sometimes we write $N$ instead of $N_L$ if there is no danger of confusion. For most of the paper, we will focus on the two-component link $L= L_P$ associated to the satellite operation with pattern $P$, and we will call $N_{L_P}$ the \textit{width} of the pattern $P$. 	
	\label{def:width N}
\end{define}

We recall that the bimodule ${}_{\cK} \cX(L)^{\bF[W,Z]}$ is constructed using the $H$-function of $L$. Using the symmetry of the $H$-function in Lemma \ref{lem:properties-H-function} and the stabilization behavior of the $H$-function as $t\to \infty$, we observe that the quantity $N_L$ satisfies the following equations:
\begin{equation}
	\label{eq:width}
	\begin{split}
		 N_L =& \min\left\{t_0 \in \mathbb{Z}+\ell/2 \middle| \begin{array}{l}  H_L(t,r)=H_L(t+1,r), \text{ for}  \\   \text{all } t\ge t_0 \textrm{ and } r\in \Z+\ell/2\end{array}\right\} \\
		 =-& \max\left\{t_0 \in \mathbb{Z}+\ell/2 \middle| \begin{array}{l}  H_L(t-1,r)=H_L(t,r)+1,  \\   \text{for all } t\le t_0 \textrm{ and } r\in \Z+\ell/2\end{array}\right\}.
	\end{split}
\end{equation}

 The following lemma gives a bound on $N_L$.
 \begin{lem} Suppose $L=L_1\cup L_2$ is a two-component L-space link with linking number $\ell$, and width $N_L$. Then,
 	\[
 -N_L+g_3(L_1)\le \frac{\ell}{2}\le N_L-g_3(L_1).
 	\] 
 	In particular, if $L_1$ is an unknot, then
 	\[
 	-N_L\le \frac{\ell}{2}\le N_L.
 	\]
 	\label{lem:bound on N}
 \end{lem}
 \begin{proof} 
 	This follows from the stabilization behavior \[H_{L_1}(t-\ell/2)=H_L(t,\infty)\]
 	in Lemma \ref{lem: properties of H function}, together with the following property of the $H$-function of an L-space knot $K$: If $g=g_3(K)$ is the Seifert genus of $K$, then 
 	\[H_K(g-1)=  H_K(g)+1=1,\quad H_K(-g+1) = H_K(-g)=g.\] 
 It follows that 
 	\[
 	\begin{split}
 		 H_L(g+\ell/2-1,\infty) &= H_L(g+\ell/2,\infty)+1,\\
 		 H_L(-g+\ell/2+1,\infty) &= H_L(-g+\ell/2,\infty).
 	\end{split}
 	\]
 	Therefore, from  Equation~\eqref{eq:width}, we have
 	\[N_L \ge g+\ell/2, \text{  and  }  -N_L \le -g+\ell/2.\]
 \end{proof}

\begin{rem}
When $L= L_1\cup L_2$ is a two-component L-space link such that $L_1$ is an unknot, which is the case for $L_P = \mu \cup P$, it follows from the relation between the $H$-function and the normalized Alexander polynomial $\tilde{\Delta}_L$ in Proposition \ref{prop:GNH-function} that:
\begin{enumerate}
	\item If $\tilde{\Delta}_{L}(x_1,x_2)=0,$ then $N_L=0$.
	\item If $\tilde{\Delta}_L (x_1,x_2)\neq 0$, then $N_L$ is the highest power of $x_1$ appearing in $\tilde{\Delta}_L$.
\end{enumerate}
	 \label{rem:highest-power}
\end{rem}

In the vertical direction, it is helpful to make the following definition involving the $H$-function:

\begin{define}
	If $L$ is a two-component L-space link with linking number $\ell$ and $H$-function $H_L(t,r)$, we set
	\[
	R_t := \max\left\{r\in \Z+\ell/2 \middle|  \begin{array}{l}  H_{L}(t,r+1)=H_L(t,r)\text{ and}\\   H_L(t,r-1)=H_L(t,r)+1\end{array} \right\}. 
	\]
\end{define}

Note that $R_t$ is equal to the $A_2$-grading of the top generator $\xs^t_0$ in the staircase $\cC_t$ that we used in our construction of ${}_{\cK} \cX(L)^{\bF[W,Z]}$ in Section~\ref{sec:bimodule-X-L_P}. We denoted this earlier by $ r^t_0 $. It follows easily from the definition of $r^t_0 $ that 
\[H_L(t,r) =H_L(t,\infty)  = H_{L_1}\left(t-\ell/2\right), \text{  for all $t\in \Z +\ell/2$ and $r\ge R_t$.}\]

We now study how the numbers $R_t$ vary as $t$ changes, when $L_P=\mu\cup P$ is a two-component L-space link such that $\mu$ is an unknot:

\begin{lem}Let $P$ be an L-space pattern and let $L_P=\mu\cup P$ be the associated two-component link in $S^3$. Write $\ell = \lk(\mu,P)$ for the linking number, $g_3(P)$ for the Seifert genus of $P$ viewed as a knot in $S^3$, and $N = N_{L_P}$ for the width of $L_P$. Then the numbers $R_t$ satisfy the following:
	\begin{enumerate}
		\item When $t\le-N $, $R_t = g_3(P)-\frac{\ell}{2}$.
		\item\label{claim-2-monotonicity} When $t\le \frac{\ell}{2}, $ $R_{t-1} \le R_t$.
		\item\label{claim-3-monotonicity} When $t\ge \frac{\ell}{2}$, $R_t \ge R_{t+1}$.
		\item When $t\ge N$, $R_t = g_3(P)+\frac{\ell}{2}$.
	\end{enumerate}
	In words, $R_t$ is constant for $t\le-N$, then non-decreasing until $t=\frac{\ell}{2}$, then non-increasing, and finally becomes constant for $t\ge N$. 
	
	In particular, the maximum of $R_t$ is achieved when $t=\frac{\ell}{2}$, and 
	\[
	R_t \ge g_3(P)-\frac{\ell}{2} \text{ for } t\le \frac{\ell}{2},
	\qquad 
	 R_t \ge g_3(P)+\frac{\ell}{2} \text{ for }t\ge \frac{\ell}{2}.\]
	\label{lem:shape of top generators}
\end{lem}

See Figure \ref{fig:example of shape of R_t} for an illustration of the shape of $R_t$ in the case of the Mazur pattern, where $\ell =1$ and $g_3(P)=0$. We add a box to the position corresponding to the top generator $\xs^t_0$ in each staircase $\cC_t$.

\begin{figure}[h!]
	\begin{tabular}{m{4cm}p{8cm}}
		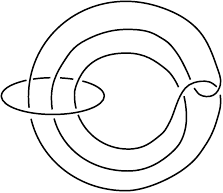
		&$ \begin{array}{|c|ccccccccc|}
			\hline
			\hbox{\diagbox{$r$}{$t$}}&\cdots&\sfrac{-5}{2}&\sfrac{-3}{2}&\sfrac{-1}{2}&\sfrac{1}{2}&\sfrac{3}{2}&\sfrac{5}{2}&\cdots&\\
			\hline	\vdots&\ddots&3&2&1&0&0&0&\reflectbox{$\ddots$}&\\
			\sfrac{3}{2}&\cdots&3&2&1 &\boxed{0}&0&0 & \cdots& \\
			\sfrac{1}{2}&\cdots&3&2&\boxed{1} &1&\boxed{0}&\boxed{0} &\cdots&\\
			\sfrac{-1}{2}&\cdots&\boxed{3}&\boxed{2}&2 &1&1&1& \cdots&\\
			\sfrac{-3}{2}&\cdots&4&3&2 &2&2&2& \cdots&\\
			\vdots&\reflectbox{$\ddots$}&5&4&3 &3&3&3&\ddots &\,\\
			\hline
		\end{array}
		$
	\end{tabular}
	\caption{The $H$-function and the Alexander gradings of the top generator $\xs^t_{0}$ in each $\cC_t$ for $t\in \bZ +\frac{1}{2}$ . }
	\label{fig:example of shape of R_t}
\end{figure}

\begin{proof}

	 The statement for $t\le -N$ and $N\le t$ follows from the the description of $N$ in terms of $H$-function in Equation \eqref{eq:width}. Therefore it suffices to prove the second and third claims, which concern monotonicity of $R_t$.

	 We first point out that since $\mu$ is the unknot, we have	
	 \begin{equation}H_L(t,R_t) =H_L(t,\infty)= H_{\mu}(t-\ell/2) = \begin{cases}
	 	\frac{\ell}{2}-t & \text{if $t< \frac{\ell}{2},$}\\
	 	0 & \text{if $t\ge \frac{\ell}{2}$.}
	 \end{cases}
	 \label{eq:stabilization-H-Rt}
	  \end{equation}

	 We now prove Claim~\eqref{claim-2-monotonicity}.   Suppose, for contradiction, that $R_{t_0-1} > R_{t_0}$ for some $t_0\le \frac{\ell}{2}$. Equation~\eqref{eq:stabilization-H-Rt} implies that if $t\le \frac{\ell}{2}$,  then
	 \begin{equation}
	 H_L(t,R_t) = \frac{\ell}{2}-t\quad \text{and} \quad H_L(t-1,R_{t-1})=\frac{\ell}{2}-t+1.\label{eq:stabilization-H-Rt-case-1}
	 \end{equation}
	 In particular, by the monotonicity of $H_L(t,r)$ in $r$ and by Equation~\eqref{eq:stabilization-H-Rt-case-1}, we have 
     \[H_L(t_0-1,R_{t_0}) \ge H_L(t_0-1,R_{t_0-1}-1) = H_{L}(t_0-1,R_{t_0-1})+1.\]
     This implies 
     \[H_L(t_0-1,R_{t_0}) - H_L(t_0,R_{t_0})\ge H_{L}(t_0-1,R_{t_0-1})+1-H_L(t_0,R_{t_0})= 2,\]
     which contradicts the bounded gap property of $H_L$. See the left figure in Figure~\ref{fig:proof of monotonicity of R_t} for an illustration.

	We now prove Claim~\eqref{claim-3-monotonicity}. When $t\ge \frac{\ell}{2}$, we have 
	 \[
	 H_L(t,R_t) =H_L(t,\infty)=0, \text{ and } H_L(t+1,R_{t+1}) = H_L(t,\infty)=0.
	 \] 
	 Suppose, for contradiction, that $R_{t_0} < R_{t_0+1}$ for some $t_0\ge \frac{\ell}{2}$, then we have
	 \[H_{L}(t_0+1,R_{t_0})\ge  H_{L}(t_0+1,R_{t_0+1}-1)=H_{L}(t_0+1,R_{t_0+1})+1= 1, \]
	  by the monotonicity of $H_L(t,r)$ in $r$ and the definition of $R_{t_0+1}$. However, this would imply $H_{L}(t_0+1,R_{t_0})>H_{L}(t_0,R_{t_0})$ since 
	  \[
	  H_{L}(t_0+1,R_{t_0}) \ge H_{L}(t_0+1,R_{t_0+1})+1=1> H_L(t_0,R_{t_0})=0,
	  \] 
	  contradicting to the monotonicity of $H_L(t,r)$ in $t$. See the right figure in Figure \ref{fig:proof of monotonicity of R_t} for an illustration. This completes the proof.	 \end{proof}
	  \begin{figure}[h!]
	 		$	\begin{array}{|c|cc|}
	 			\hline
	 			\hbox{\diagbox{$r$}{$t$}}&t_0-1&t_0 \\
	 			\hline	R_{t_0-1} & k & \vdots\\
	 			\vdots   &\vdots & \vdots \\
	 			R_{t_0} & \ge k+1 & k-1\\
	 			\hline
	 		\end{array} 
	 		\qquad \qquad 
	 		\begin{array}{|c|cc|}
	 			\hline
	 			\hbox{\diagbox{$r$}{$t$}}&t_0&t_0+1 \\
	 			\hline	R_{t_0+1} & \vdots &0\\
	 			\vdots   &\vdots & \vdots \\
	 			R_{t_0} &  0 &  \ge 1\\
	 			\hline
	 		\end{array}$
	 	\caption{The left figure depicts the contradiction when $t\le \frac{\ell}{2}$, where $k = \frac{\ell}{2}-(t-1)$. The right figure depicts the contradiction when $t\ge \frac{\ell}{2}$.}
	 	\label{fig:proof of monotonicity of R_t}
	 \end{figure}

\subsection{The  bimodule ${}_{\cK} \cX^{\diamond}(L_P)_{\bF[Z]}$}

\label{sec:diamond-module}
In this section, we study the type-$AA$ bimodule
\[
{}_{\cK} \cX(L_P)_{\bF[Z]}:={}_{\cK} \cX(L_P)^{\bF[Z]}\boxtimes {}_{\bF[Z]} \bF[Z]_{\bF[Z]}.
\]
 This bimodule will typically have $m_{0|1|0}$, $m_{1|1|0}$, $m_{2|1|0}$ and $m_{0|1|1}$ non-zero. In this section, we prove that this bimodule is homotopy equivalent to a type-$AA$ bimodule ${}_{\cK} \cX^{\diamond}(L_P)_{\bF[Z]}$ which has only $m_{1|1|0}$ and $m_{0|1|1}$ non-trivial, and we describe this bimodule.

Suppose that $\cC$ is a (positive) staircase complex, viewed as a type-$D$ structure over $\bF[W,Z]$. Write $\xs_{0},\dots, \xs_{m}$ for the generators of $\cC$. Write $\delta_1^1(\xs_{2i+1})=\xs_{2i}\otimes W^{\a_i}+\xs_{2i+2}\otimes Z^{\b_i}$. When we set $W=0$, the effect is to replace $\cC$ with the type-$D$ structure shown below:
\[
\begin{tikzcd}
\xs_0 & \xs_1\ar[r, "Z^{\a_1}"] &\xs_2& \cdots & \xs_{2m-1} \ar[r, "Z^{\a_{m}}"] & \xs_{2m}.
\end{tikzcd}
\] 
If we tensor the above complex (viewed as a type-$D$ module over $\bF[Z]$) with ${}_{\bF[Z]} \bF[Z]_{\bF[Z]}$, then $\cC\boxtimes \bF[Z]$ is quasi-isomorphic to the type-$A$ module
\begin{equation}
\cC^{\diamond}:=\bF[Z] \oplus \bF[Z]/Z^{\a_1}\oplus \cdots \oplus \bF[Z]/Z^{\a_{m}}, \label{eq:C_t-min-def}
\end{equation}
 where we view the above as having vanishing differential and the obvious action of $\bF[Z]$. 
 
 We observe that there is a strong deformation retraction of chain complexes over $\bF$
 \begin{equation}
\begin{tikzcd}
 \cC\boxtimes \bF[Z]\ar[r, "\Pi", shift left] \ar[loop left, "H"] & \cC^{\diamond} \ar[l, "I",shift left]
 \end{tikzcd}
 \label{eq:HPL-diamond}
 \end{equation}
 The map $\Pi$ is projection of modules, the map $I$ is the inclusion of $\bF$-vector spaces. The map $H$ vanishes on $\xs_0Z^i$ as well as $\xs_{2j+1} Z^i$ for all $i$ and $j$. Furthermore, 
 \[
 H(\xs_{2i} Z^j)=\begin{cases}\xs_{2i-1} Z^{j-1} & \text{ if } j\ge \a_j\\ 0 & \text{ otherwise}.\end{cases}
 \]
 
 We apply the above construction to the modules ${}_{\cK} \cX(L_P)^{\bF[W,Z]}$. These bimodules are built from staircases $\cC_t$ for $t\in \ell/2+\Z$ (in idempotent 0) and a staircase $\cS$ (in idempotent 1). We define a right $\bF[Z]$ module
 \[
 \cX^\diamond(L_P)_{\bF[Z]}
 \]
 by replacing each copy of $\cC_t$ and $\cS$ with $\cC_t^{\diamond}$ and $\cS^{\diamond}$, respectively. The strong deformation retraction in Equation~\eqref{eq:HPL-diamond} can be applied to each copy of $\cC_t$ and $\cS$ separately, to give a strong deformation retraction of chain complexes over $\ve{I}$ between $\cX(L_P)^{\bF[Z]}\boxtimes {}_{\bF[X]}\bF[Z]$ and $\cX^{\diamond}(L_P)$. Applying the homological perturbation lemma for type-$AA$ bimodules, Lemma~\ref{lem:HPL-AA}, equips $\cX^{\diamond}(L_P)$ with an $AA$-bimodule structure over $(\cK,\bF[Z])$ which is homotopy equivalent to ${}_{\cK} \cX(L_P)_{\bF[Z]}$ and which has vanishing $m_{0|1|0}$.

 \begin{lem}
 \label{lem:X-min-non-trivial-actions} The only non-vanishing actions $m_{i|1|j}$ on ${}_{\cK} \cX^{\diamond}(L_P)_{\bF[Z]}$ are $m_{1|1|0}$ and $m_{0|1|1}$. Furthermore, the right action of $\bF[Z]$ sends $\cC_t^{\diamond}$ to itself, and coincides with the obvious $\bF[Z]$-module structure from its description in Equation~\eqref{eq:C_t-min-def}. Similarly, the right $\bF[Z]$-module structure sends $\cS^{\diamond}\otimes T^i$ to itself, and coincides with the natural analogous action.
 \end{lem}
 \begin{proof} We note that ${}_{\cK}\cX(L_P)^{\bF[Z]}$ admits an algebraic grading supported in $\{0,1\}$. We also view $\bF[Z]$ as being supported in algebraic grading 0. The generators $\xs_{2i}^t$ are in algebraic grading 0, and the generators $\xs_{2i+1}^t$ are in algebraic grading $1$. We view $\cX^{\diamond}(L_P)$ as having an algebraic grading supported in grading 0.  The maps $\Pi$ and $I$ preserve the algebraic grading, while $H$ increases it by 1. The action of $\delta_2^1$ of ${}_{\cK} \cX(L_P)^{\bF[Z]}$ preserves the algebraic grading, whereas the action $\delta_3^1$ increases it by 1. Since $\cX^{\diamond}(L_P)$ is supported in algebraic grading 0, by examining the structure maps constructed by the homological perturbation lemma (see Equation~\eqref{eq:HPL_structure_maps}) it follows immediately that the only actions on ${}_{\cK} \cX^{\diamond}(L_P)_{\bF[Z]}$ which are non-trivial are those which involve no applications of $H$ or $m_{2|1|0}$, i.e. we include via $I$, then apply an action $m_{1|1|0}$, $m_{0|1|1}$ of ${}_{\cK}\cX(L_P)_{\bF[Z]}$, and then we apply $\Pi$, as claimed. The second claim about the action of $\bF[Z]$ is immediate from this description.
 \end{proof}

\begin{rem}
\label{rem:X-min-uniquely-determined} We recall that the bimodules ${}_{\cK} \cX(L_P)^{\bF[W,Z]}$ were only well-defined up to homotopy equivalence in \cite{CZZ}*{Proposition~5.10}. It follows  that the bimodules ${}_{\cK} \cX^{\diamond}(L_P)_{\bF[Z]}$ are well-defined up to homotopy equivalence. Noting, however, that $m_{0|1|0}$ vanishes on this module, and that there is at most one generator in a given Maslov-Alexander grading, it follows that the module structure ${}_{\cK} \cX^{\diamond}(L_P)_{\bF[Z]}$ is completely specified (not just up to homotopy equivalence).
\end{rem}

\subsection{The bimodules ${}_{\cK} \cX^{\free}(L_P)_{\mathbb{F}[Z]}$ and ${}_{\cK} \cX^{\tors}(L_P)_{\mathbb{F}[Z]}$}

We observe that the underlying $\bF[Z]$-module of ${}_{\cK} \cX^{\diamond}(L_P)_{\bF[Z]}$ splits canonically as
\[
\cX^{\diamond}(L_P)_{\bF[Z]}\iso \cX^{\free}(L_P)_{\bF[Z]} \oplus \cX^{\tors}(L_P)_{\bF[Z]}.
\]
Here $\cX^{\free}(L_P)$ is the $\bF$-vector space spanned by $\xs_0^{t}Z^i$ and $\xs_0'Z^i$, for $i\ge 0$. The module $\cX^{\tors}(L_P)$ is the $\bF$-vector space spanned by $\xs_{2i+1}^t Z^j$ and $\xs_{2i+1}' Z^j$ for $j\ge 0$. Note that this splitting is canonical, since $\xs_0^t Z^i$ and $\xs_0' Z^i$ do not share the same Alexander gradings and idempotents as any other generators of $\cX^{\diamond}(L_P)$. 

Of course, we can also view $\cX^{\free}(L_P)$ as the quotient of $\cX^{\diamond}(L_P)$ by the subspace $\cX^{\tors}(L_P)$.

We now consider the interaction of the left $\cK$-action with respect to this splitting:

\begin{lem}
\label{lem:X-tors-submodule}
 The subspace $\cX^{\tors}(L_P)\subset \cX^{\diamond}(L_P)$ is a submodule with respect to the bimodule structure over $(\cK,\bF[Z])$.
\end{lem}
\begin{proof}
 Note that clearly $m_{0|1|1}$ preserves $\cX^{\tors}(L_P)$, so it suffices to check that the left $\cK$-action preserves $\cX^{\tors}(L_P)$. Recall that the only actions on $\cX^{\diamond}(L_P)$ which are non-trivial are $m_{1|1|0}$ and $m_{0|1|1}$ by Lemma~\ref{lem:X-min-non-trivial-actions}. Therefore $m_{1|1|0}(a,-)$ commutes with the action of $Z$ for all $a\in \cK$. Since there are no non-zero maps from a torsion $\bF[Z]$-module to a free $\bF[Z]$-module, it follows that $m_{1|1|0}(a,-)$ must send $\cX^{\tors}(L_P)$ to itself, proving the claim.
\end{proof}

\begin{define}
	 Suppose that $P$ is an L-space satellite operator. Let ${}_{\cK}\cX^{\tors}(L_P)_{\bF[Z]}$ be the submodule of ${}_{\cK} \cX^{\diamond}(L_P)_{\bF[Z]}$ spanned by all of the $\bF[Z]$-torsion elements of $\cX^{\diamond}(L_P)$.  Define the bimodule ${}_{\cK}\cX^{\free}(L_P)_{\bF[Z]}$ as the quotient
	 \[
	  \cX^{\free}(L_P):= \cX^{\diamond}(L_P)/\cX^{\tors}(L_P)
	  \]
	  with its induced quotient module structure over $(\cK,\bF[Z])$.
\end{define}

 We will see that in many cases, to compute $\tau(P(K,n))$, it is sufficient to consider ${}_{\cK} \cX^{\free}(L_P)_{\bF[Z]}$ instead of the full module ${}_{\cK} \cX^{\diamond}(L_P)_{\bF[Z]}$.

\begin{rem}\label{rem:X-free-uniquely-determined} In a similar manner to Remark~\ref{rem:X-min-uniquely-determined}, the structure on ${}_{\cK} \cX^{\free}(L_P)_{\bF[Z]}$ is uniquely determined (not just up to homotopy equivalence).
\end{rem}

In the above, we are viewing $\cX^{\tors}(L_P)$ as a $(\cK,\bF[Z])$-submodule of $\cX^{\diamond}(L_P)$. It is helpful to recall that the ordinary notion of a submodule extends to the setting of type-$AA$ bimodules in a simple manner. We recall the following: 

\begin{define}
    \label{def:subbimodule} Let $\cA$ and $\cB$ be two associative algebras over rings $\ve{k}$ and $\ve{j}$, respectively. 	Let ${}_{\cA}M_{\cB}$ be a $AA$-bimodule over $\cA$ and $\cB$.  A $(\ve{k},\ve{j})$-submodule $N$ of $M$ is a \textit{sub $AA$-bimodule} of ${}_{\cA}M_{\cB}$ if 
	\[m_{i|1|j}\left(\cA^{\otimes i}\otimes N\otimes \cB^{\otimes j}\right) \subseteq N,\]
	for all $i,j\ge 0$.  
	
	If $N$ is a sub $AA$-bimodule of $M$, then the structure maps $m_{i|1|j}$ naturally descend to the $(\ve{k},\ve{j})$-module quotient $M/N$: 
		\[
		m_{i|1|j}\left(\cA^{\otimes i}\otimes (M/N)\otimes \cB^{\otimes j}\right) \to M/N.\]
		Furthermore, these maps give an $AA$-bimodule structure on the quotient $M/N$. We call the induced $AA$-bimodule a \textit{quotient $AA$-bimodule}.
\end{define}

Note that the bimodule ${}_{\cK} \cX^{\tors}(L_P)_{\bF[Z]}\subset {}_{\cK} \cX^{\diamond}(L_P)_{\bF[Z]}$ satisfies the above definition of a sub $AA$-bimodule, since it is a submodule by Lemma~\ref{lem:X-tors-submodule}, and $\cX^{\diamond}(L_P)$ has vanishing higher actions by Lemma~\ref{lem:X-min-non-trivial-actions}.

\begin{rem}
	By the definition of the box tensor product, if ${}_{\cA}N_{\cB}$ is a sub $AA$-bimodule of ${}_{\cA}M_{\cB}$, then for any $DA$-bimodule $ {}_{\cA'}L^{\cA}$, the box tensor product $ {}_{\cA'}L^{\cA}\boxtimes{}_{\cA}N_{\cB}$ is also a sub $AA$-bimodule of $ {}_{\cA'}L^{\cA}\boxtimes{}_{\cA}M_{\cB}$. A similar statement holds for quotient $AA$-bimodules. Also, similar remarks hold for tensoring on the left with a type-$D$ module $J^{\cA}$ instead of a $DA$-bimodule $ {}_{\cA'}L^{\cA}$.
	\label{rem:sub/quotient under box tensor}
\end{rem}

We now describe the bimodule structure on ${}_{\cK} \cX^{\free}(L_P)_{\mathbb{F}[Z]}$ more explicitly:

\begin{lem}
	\label{lem:quotient bimodule in Lp}
	Suppose $P$ is an L-space satellite operator, and let $L_P$ be the associated two-component link. Let $\ell = \lk(\mu,P)$ be the linking number of $L_P$, and $g=g_3(P)$ be the Seifert genus of the pattern $P$ viewed as a knot in $S^3$.  Then the left action $L_a:= m_{1|1|0}(a,-) \colon \cX^{\free}(L_P) \to \cX^{\free}(L_P)$ of $a\in \{W,Z,\sigma,\tau\}$ on the quotient bimodule ${}_{\cK} \cX^{\free}(L_P)_{\mathbb{F}[Z]}$ are given by the following formulas for each $t\in \Z+\frac{\ell}{2}$:	
	\begin{align*}
		 L_Z(\xs^t_{0}) = &
		 	\begin{cases}
		 		0 & \text{if $t< \frac{\ell}{2}$,}\\
		 		\xs^{t+1}_0  Z^{j_t} & \text{if $t\ge \frac{\ell}{2},$ where $j_t=R_t-R_{t+1}$. }
		 	\end{cases}        \\
	 L_W(\xs^t_{0}) =
	  &	\begin{cases}
	 		\xs^{t-1}_0  Z^{k_t} & \text{if $t\leq \frac{\ell}{2},$ where $k_t=R_t-R_{t-1}$, }\\
	 		0 & \text{if $t> \frac{\ell}{2}$.}
	 	\end{cases}       \\
 	L_{\sigma}(\xs^t_{0}) =
 	&\begin{cases}
 		0 & \text{if $t< \frac{\ell}{2}$,}\\
 		\xs'_0  Z^{j'_t} & \text{if $t\ge \frac{\ell}{2},$ where $j'_t=R_t-\frac{\ell}{2}-g_3(P)$. }
 	\end{cases}       \\
L_{\tau }(\xs^t_{0}) =
&	\begin{cases}
\xs'_0  Z^{k'_t} & \text{if $t\leq \frac{\ell}{2},$ where $k'_t=R_t+\frac{\ell}{2}-g_3(P)$, }\\
0 & \text{if $t> \frac{\ell}{2}$.}
\end{cases}  
	\end{align*}
Additionally,
\[
L_{T^{\pm 1}}(T^t\xs'_0) = T^{t\pm 1}\xs'_0 ,\quad \text{and} \quad\L_{U}(\xs'_0)=0.\]
\end{lem}
Schematically,  ${}_{\cK} \cX^{\free}(L_P)_{\mathbb{F}[Z]}$ takes the following form, where we only label the non-zero actions of $L_{a}$ for $a=W,Z,\sigma,\tau, T,T^{-1}$, and each generator denotes a free $\bF[Z]$-module generator:
\begin{equation}  
	\begin{tikzcd}[ column sep=1.1cm, row sep=0cm]
		\cdots
		&[-1.2cm]\xs^{\sfrac{\ell}{2}-2}_0
		\ar[d, "L_{\tau}"]
		& \xs^{\sfrac{\ell}{2}-1}_0
		\ar[l,  "L_W",shorten = -1.5mm,swap]
		\ar[d, "L_\tau"]
		&\xs^{\sfrac{\ell}{2}}_0
		\ar[r, "L_Z"]
		\ar[l,  "L_W",shorten = -1.5mm,swap]
		\ar[d,"L_{\sigma}", bend left]
		\ar[d, "L_\tau", bend right,swap]
		& \xs^{\sfrac{\ell}{2}+1}_0
		\ar[r,  "L_Z"]
		\ar[d,"L_{\sigma}"]
		& \xs^{\sfrac{\ell}{2}+2}_0
		\ar[d,"L_{\sigma}"]
		&[-1.2cm] \cdots
		\\[2cm]
		\cdots
		&T^{\sfrac{\ell}{2}-2}\xs'_0
		\ar[r, bend left, "L_T"]
		&T^{\sfrac{\ell}{2}-1}\xs'_0
		\ar[r, bend left, "L_{T}"]
		\ar[l, bend left, "L_{T^{-1}}"]
		&T^{\sfrac{\ell}{2}}\xs'_0
		\ar[r, bend left, "L_{T}"]
		\ar[l, bend left, "L_{T^{-1}}"]
		&T^{\sfrac{\ell}{2}+1}\xs'_0
		\ar[r, bend left, "L_T"]
		\ar[l, bend left, "L_{T^{-1}}"]
		&T^{\sfrac{\ell}{2}+2}\xs'_0
		\ar[l, bend left, "L_{T^{-1}}"]
		&\cdots 
	\end{tikzcd}
	\label{eq:X(L_P)-complex}
\end{equation}

\begin{proof}
In our construction of the modules ${}_{\cK} \cX(L_P)^{\bF[W,Z]}$, if $\xs\in\cC_t\subset \cX(L_P)$ is a staircase generator in idempotent 0 and in algebraic grading 0, we can construct a model of the bimodule  ${}_{\cK}\cX(L_P)^{\bF[W,Z]}$ where $L_Z(\xs)=\delta_2^1(Z,\xs)$ is equal to $\ys\otimes W^{i} Z^j$ for any choice of staircase generator $\ys\in\cC_{t+1}$ which has 
\[
(\gr_{\ws},\gr_{\zs})(\ys)-(2i,2j)=(\gr_{\ws},\gr_{\zs})(\xs)-(2,0).
\]
 Similarly, we can pick $L_W(\xs)$, $L_\sigma(\xs)$ and $L_\tau(\xs)$ to be any staircase generators which have the appropriate gradings. This follows immediately from the description in \cite{CZZ}*{Sections~5.2, 5.3}.
	 
	 Recall from Remarks~\ref{rem:X-min-uniquely-determined} and \ref{rem:X-free-uniquely-determined} that although the bimodule ${}_{\cK} \cX(L_P)^{\bF[W,Z]}$ is only determined up to homotopy equivalence, the bimodule ${}_{\cK} \cX^{\free}(L_P)_{\bF[Z]}$ is completely determined. Therefore we only have to show the stated formulas for some choice of structure map on ${}_{\cK} \cX(L_P)^{\bF[W,Z]}$. In the rest of this proof, we will write $\delta_2^1(a,-)$ for the structure maps on the $DA$-bimodule ${}_{\cK} \cX(L_P)^{\bF[W,Z]}$, and $L_a$ for the induced maps on ${}_{\cK} \cX^{\free}(L_P)_{\bF[Z]}$. 
	
We consider first the claim about $ L_\sigma(\xs_0^t)$. We break the proof into two cases, depending on whether $t<\ell/2$ or $t\ge \ell/2$.

 Consider first $t<\ell/2$. We write $\delta_2^1(\sigma,\xs_0^t)=\xs_i'\otimes W^{p_t}Z^{q_t}$ for some $i$. If $i>0$, there is nothing to show. If $i=0$, then since $\gr_{\ws}(\delta_2^1(\sigma,-))=0$ by Equation~\eqref{eq:grading shifts of left action}, we must have
\[
\gr_{\ws}(\xs_0^t)=\gr_{\ws}(\xs_0')-2p_t.
\]
We observe $\gr_{\ws}(\xs_0')=0$ since $\cS$ is s staircase complex with standard gradings in the sense of Definition~\ref{def:standard-gradings}. Furthermore, 
\[
\gr_{\ws}(\xs_0^t)=-2H_{L_P}(t,R_t)=-2H_{L_P}(t,\infty)=-2H_\mu(t-\ell/2)
\] by Lemma~\ref{lem: properties of H function} and the definition of $R_t$. We recall that
\[
H_{\mu}(t-\ell/2)=\begin{cases} 0 & \text{ if } t\ge \ell/2\\
\ell/2-t & \text{ if } t<\ell/2.
\end{cases}
\]
In particular, $ H_\mu(t-\ell/2)>0$ if $t<\ell/2$, and therefore $p_t>0$ if $t<\ell/2$, so the left action $ \delta_2^1(\sigma,\xs_0^t)=\xs_i'\otimes W^{p_t}Z^{q_t}$ will make no contribution after setting $W=0$.

Next, we consider $t\ge \ell/2$. As before, we write $ \delta_2^1(\sigma,\xs_0^t)=\xs_i'\otimes W^{p_t}Z^{q_t}$. We claim that we can take $i=0$ and $p_t=0$. To see this, note that it suffices to show that
\[
\gr_{\ws}( \delta_2^1(\sigma,\xs_0^t))= \gr_{\ws}(\xs_0')\quad \text{and} \quad \gr_{\zs}( \delta_2^1(\sigma,\xs_0^t))\le \gr_{\zs}(\xs_0'),
\]
because then a suitable product of $\xs_0'$ by powers of $Z$ will have the correct grading, and can be chosen to be $ \delta_2^1(\sigma,\xs_0^t)$ in our construction of the module ${}_{\cK} \cX(L_P)^{\bF[W,Z]}$. The grading shift of $ \delta_2^1(\sigma,-)$ is recalled in Equation~\eqref{eq:grading shifts of left action}. We have
\[
\gr_{\ws}( \delta_2^1(\sigma,\xs_0^t))=-2H_{L_P}(t,R_t)=-2H_\mu(t-\ell/2)\quad \text{and} \quad \gr_{\ws}(\xs_0')=0.
\]  
Additionally,
\[
\gr_{\zs}( \delta_2^1(\sigma,\xs_0^t))=-2H_{L_P}(t,R_t)-2R_t+\ell=-2H_\mu(t-\ell/2)-2R_t+\ell\quad \text{and }\quad \gr_{\zs}(\xs_0')=-2g.
\]
We note that $\gr_{\ws}( \delta_2^1(\sigma,\xs_0^t))= \gr_{\ws}(\xs_0')$ because $\gr_{\ws}(\xs_0')=0$ and $\gr_{\ws}( \delta_2^1(\sigma,\xs_0^t))=-2H_{\mu}(t-\ell/2)=0$. On the other hand, the inequality $\gr_{\zs}( \delta_2^1(\sigma,\xs_0^t))\le \gr_{\zs}(\xs_0')$ is equivalent to
\begin{equation}
0\le H_{\mu}(t-\ell/2)+R_t-\ell/2-g. \label{eq:Z-power-positive}
\end{equation}
Since $H_\mu(t-\ell/2)=0$ for $t\ge \ell/2$, Equation~\eqref{eq:Z-power-positive} follows from the cases of $\ell/2\le t\le N$ and $N\le t$ in the statement of Lemma~\ref{lem:shape of top generators}. Together, these imply that $R_t\ge g+\ell/2$ for $t\ge \ell/2$, which yields the formula in Equation~\eqref{eq:Z-power-positive}.

 Next, we easily compute using the above grading formulas in Equation~\eqref{eq:grading shifts of left action} that $q_t=R_t+\ell/2-g$ for $t\ge \ell/2$.

Next, we address the claim for $L_\tau$. As before, we break the claim into two cases, depending on whether $t\le \ell/2$ or $t>\ell/2$.

 We consider the case that $t>\ell/2$ first. Suppose $ \delta_2^1(\tau,\xs_0^t)=\xs_i'\otimes W^{p_t}Z^{q_t}$. If $i>0$, then there is nothing to show, so suppose $i=0$. Using Equation~\eqref{eq:grading shifts of left action} we compute that
 \[
-2p_t= \gr_{\ws}(\delta_2^1(\tau,\xs_0^t))-\gr_{\ws}(\xs_0')=-2H_\mu(t-\ell/2)-2t+\ell. 
 \]
 If $t>\ell/2$, then the right hand of the above equation is $-2t+\ell$, which is negative, so $p_t>0$. Therefore $\delta_2^1(\tau,\xs_0^t)$ becomes zero in $\cX^{\free}$ since we have set $W=0$. 
 
 Next, we consider the case that $t\le \ell/2$. In this case, we claim that we can take $ \delta_2^1(\tau,\xs_0^t)=\xs_0'\otimes Z^{q_t}$ for some $q_t\ge 0$. For this to be the case, we need
 \[
 \gr_{\ws}(\delta_2^1(\tau,\xs_0^t))= \gr_{\ws}(\xs_0')\quad \text{and} \quad \gr_{\zs}( \delta_2^1(\tau,\xs_0^t))\le \gr_{\zs}(\xs_0').
 \]
 The equation involving the $\gr_{\ws}$-grading is a straightforward computation using Equation~\eqref{eq:grading shifts of left action}, while the inequality involving the $\gr_{\zs}$-grading is demonstrated as follows.
We compute that 
\[
\gr_{\zs}( \delta_2^1(\tau,\xs_0^t))=\gr_{\zs}(\xs_0^t)=-2H_{L_P}(t,R_t)-2t-2R_t=-2H_{\mu}(t-\ell/2)-2t-2R_t.
\]
Since $\gr_{\zs}(\xs_0')=-2g$, it suffices therefore to show 
\[
-2H_\mu(t-\ell/2)-2t-2R_t\le -2g.
\] 
Using the value $H_{\mu}(t-\ell/2)=-(t-\ell/2)$ for $t\le \ell/2$, the above inequality is equivalent to
\[
0\le R_t+\ell/2-g,
\] 
for $t\le \ell/2$. This follows from the first two cases of Lemma~\ref{lem:shape of top generators}. The formula for $q_t$ is easily derived using the grading formulas from Equation~\eqref{eq:grading shifts of left action}.

The statements about $L_W$ and $L_Z$ follow from a similar line of reasoning, and we leave the details to the reader.
\end{proof}

\section{Computation of $\tau(P(K,n))$}
\label{sec: tau computation}

In this section, we derive our formulas for $\tau(P(K,n))$ as stated in Theorem~\ref{thm:main-intro}. Our main tool is the complex
\[
\bX(P,K,n)_{\bF[Z]}:= \cX_n(S^3,K)^{\cK}\boxtimes {}_{\cK} \cH_-^{\cK}\boxtimes {}_{\cK} \cX(L_P)_{\bF[Z]},
\]
which is homotopy equivalent to $\cCFK_{\bF[Z]}(K)$ as a type-$A$ module over $\bF[Z]$. After recalling some basic properties of this complex from \cite{CZZ}, we study this complex to compute $\tau(P(K,n))$ when $\veps(K)=1$ and $P$ is an L-space pattern. We also compute $\tau(P(K,n))$ when $\veps(K)\in \{-1,0\}$ under certain extra assumptions on the pattern $P$.

Throughout this section, we orient $P$ so that $\ell =\lk(P,\mu)\ge 0$. This does not affect the value of $\tau(P(K,n))$, since this affects neither the value of $\tau(P(K,n))$ nor $P$ being an L-space pattern.

\subsection{The quotient complex $\bX^{\free}(P,K,n)_{\bF[Z]}$}
\label{sec:maps in X top}

By Lemma~\ref{lem:X-tors-submodule}, ${}_{\cK} \cX^{\tors}(L_P)_{\mathbb{F}[Z]}$ is a submodule of ${}_{\cK} \cX^{\diamond}(L_P)_{\bF[Z]}$, and therefore the quotient ${}_{\cK} \cX^{\free}(L_P)_{\bF[Z]}$ is naturally a quotient bimodule. By Remark~\ref{rem:sub/quotient under box tensor}, this implies that 
\[
\bX^{\free}(P,K,n)_{\bF[Z]}:=\cX_n(S^3,K)^{\cK}\boxtimes {}_{\cK} \cH_-^{\cK} \boxtimes {}_{\cK} \cX^{\free}(L_P)_{\bF[Z]}\] is a quotient complex of 
\[
\bX^{\diamond}(P,K,n)_{\bF[Z]}=\cX_n(S^3,K)^{\cK}\boxtimes {}_{\cK} \cH_-^{\cK} \boxtimes {}_{\cK} \cX^{\diamond}(L_P)_{\bF[Z]}
\]
over $\bF[Z]$. In this section, we describe $\bX^{\free}(P,K,n)_{\bF[Z]}$ in more detail.

We begin by recalling from \cite{CZZ}*{Section~9} that the complex $\bX(P,K,n)^{\bF[W,Z]}$ decomposed as a 2-dimensional hypercube of chain complexes, taking the following form:
\[
\bX(P,K,n)^{\bF[W,Z]}=\begin{tikzcd}[column sep=1.5cm, row sep=1.5cm]
 \bE \ar[r, "\Phi^{\mu}+\Phi^{-\mu}"] \ar[d, "\Phi^{K}+\Phi^{-K}",swap] \ar[dr, "\Phi^{\pm K,\pm \mu}", labels=description] & \bF\ar[d, "\Phi^{K}+\Phi^{-K}"]\\
\bJ \ar[r, swap,"\Phi^{\mu}+\Phi^{-\mu}"]& \bM
\end{tikzcd}
\]
The complex $\bE$ decomposes as
\[
\bE=\bigoplus_{(s,t)\in \bH(P)} E_{s,t},
\]
where 
\[
\bH(P)=\left(\bZ+\frac{1}{2}\right)\times \left(\bZ+\frac{\lk(P,\mu)+1}{2}\right).
\]
Also $\bF$, $\bJ$, $\bM$ similarly decompose as direct sums over $\bH(P)$ of $F_{s,t}$, $J_{s,t}$ and $M_{s,t}$. We refer the reader to \cite{CZZ}*{Section~9} for a detailed description of $\bX(P,K,n)^{\bF[W,Z]}$, though we will describe some basic details shortly.

The complex $\bX^{\diamond}(P,K,n)_{\bF[Z]}$ can be described as a 2-dimensional hypercube of the form
\begin{equation}
\bX^{\diamond}(P,K,n)_{\bF[Z]}=\begin{tikzcd}[column sep=1.1cm, row sep=1.1cm]
 \bE^{\diamond} \ar[r, "\Phi^{\mu}+\Phi^{-\mu}"] \ar[d, "\Phi^{K}+\Phi^{-K}",swap] & \bF^{\diamond}\ar[d, "\Phi^{K}+\Phi^{-K}"]\\
\bJ^{\diamond} \ar[r, swap,"\Phi^{\mu}+\Phi^{-\mu}"]& \bM^{\diamond}.
\end{tikzcd}
\label{eq:X-P-K-n-top-cube}
\end{equation}
The complex $\bX^{\free}(P,K,n)$ can similarly be described as a 2-dimensional hypercube of $AA$-bimodules. We will show in Lemma~\ref{lem:no-diagonal-differential} that no diagonal differential appears in either $\bX^{\diamond}$ or $\bX^{\free}$ (i.e. there is no differential from $\bE^{\diamond}$ to $\bM^{\diamond}$).

In \cite{CZZ}*{Proposition~7.9}, we show that the bimodule ${}_{\cK} \cH_-^{\cK}\boxtimes {}_{\cK} \cX(L_P)^{\bF[W,Z]}$ has a similar description as a hypercube of $DA$-bimodules. It takes the form
\begin{equation}
{}_{\cK} \cH_-^{\cK} \boxtimes {}_{\cK} \cX(L_P)^{\bF[W,Z]}=
\begin{tikzcd}[column sep=1.6cm, row sep=1.6cm, labels=description] \scE_{*,*}\ar[d, "f^K+f^{-K}"]\ar[r,"f^\mu+f^{-\mu}"] \ar[dr,dashed, "f^{\pm K,\pm \mu}"] & \scF_{*,*}\ar[d,"f^K+f^{-K}"]\\
\scJ_{*,*}\ar[r, "f^\mu+f^{-\mu}"]& \scM_{*,*}
\label{eq:bimodule-X-P-n-cube}
\end{tikzcd}
\end{equation}
Here, $\scE_{*,*}$, $\scF_{*,*}$, $\scJ_{*,*}$ and $\scM_{*,*}$ are supported in (left) idempotents $0$, $0$, $1$ and $1$, respectively. The bimodules $\scE_{*,*}$, $\scF_{*,*}$, $\scJ_{*,*}$, $\scM_{*,*}$ are built from the complexes and maps appearing in the bimodule ${}_{\cK} \cX(L_P)^{\bF[W,Z]}$. The subscripts denote  the Alexander gradings of the two components of the Hopf link in the connected sum $H\# L_P$. In particular, we view each $\scE_{s,t}$ (and so forth) as being concentrated in Alexander gradings $(s,t)$, where $s$ corresponds to the component of $H$ in $H\# L_P$ which is not linked with $P\subset L_P$, and $t$ corresponds to the component of $H$ which is linked with $P$.

For each $(s,t)\in \bH(P)$, $\scE_{s,t}$, $\scF_{s,t}$, $\scJ_{s,t}$ and $\scM_{s,t}$ are finitely generated type-$D$ modules over $\bF[W,Z]$. Each is equal to one of the complexes $\cC_{t'}$ or $\cS$ (up to grading shifts).

By performing the construction described in Section~\ref{sec:diamond-module} (i.e. setting $W=0$ and then taking homology with respect to $\delta_1^1$), we can define complexes $\scE^{\diamond}_{*,t}, \scF^{\diamond}_{*,t}, \scJ^{\diamond}_{*,t}, \scM^{\diamond}_{*,t}$, as well as quotient complexes  $\scE^{\free}_{*,t}, \scF^{\free}_{*,t}, \scJ^{\free}_{*,t}, \scM^{\free}_{*,t}$. The versions labeled free can be viewed as the span over $\bF[Z]$ of the top generators $\xs^{t\pm 1/2}_0$ in place of $\cC_{t\pm 1/2} $ and the span of $\xs'_0$ in place of $\cS$. We use the same shifts in $\gr_{\ws}$- and $\gr_{\zs}$-gradings as in these complexes.

We can identify $\scE_{s,t}^{\free}$ as the free $\bF[Z]$-module
\begin{equation*}
	\scE^{\free}_{s,t}=\begin{cases} \langle\xs^{t-\frac{1}{2}}_{0}\rangle[Z][2s+1,0] & \text{ if } s<0,\\
		\langle\xs^{t+\frac{1}{2}}_{0}\rangle[Z][0,-2s+1] & \text{ if } s>0.
	\end{cases}
	\label{eq:cEst-defs}
\end{equation*}
Here, $[x,y]$ means an upward shift by $x$ in $\gr_{\ws}$-grading, and an upward shift by $y$ in the $\gr_{\zs}$-grading. Also $\langle \xs \rangle[Z]$ means the free $\bF[Z]$-module with generator $\xs$. Denote the unique generator of $\scE^{\free}_{s,t}$ by $\ve{e}^{\free}_{s,t}$. 

For a fixed $t$, we decompose 
\[
\cCFK(K)^{\bF[W,Z]}\boxtimes {}_{\bF[W,Z]} \scE^{\free}_{*,t} =\bigoplus_{s\in \Z+\frac{1}{2}} E^{\free}_{s,t},
\]
where $E^{\free}_{s,t}$ is the free chain complex over $\bF[Z]$ where the first coordinate of the Alexander grading equals $s$. 

Similarly, we may write
\[
\scF^{\free}_{s,t}=\langle\xs'_{0} \rangle[Z][\min(2s+1,0),\min(0,-2s-1)]
\]
for all $s,t$. Denote the unique generator of $\scF^{\free}_{s,t} $ by $\ve{f}^{\free}_{s,t}$. Then, $F^{\free}_{s,t}$ is the direct summand chain complex with the first Alexander grading equal to $s$ of the tensor product
 \[
\cCFK(K)^{\bF[W,Z]}\boxtimes_{\bF[W,Z]}\scF^{\free}_{*,t}{}.
\]

We also have 
\[
\scJ^{\free}_{s,t} = \langle \xs^{t+\frac{1}{2}}_0\rangle[Z]\quad \text{and} \quad \scM^{\free}_{s,t} = \langle \xs'_0\rangle[Z].\] We denote the unique generators of $\scM^{\free}_{s,t}$ and $\scJ^{\free}_{s,t}$ by $\ve{j}^{\free}_{s,t}$ and $\ve{m}^{\free}_{s,t}$. The complexes $J^{\free}_{s,t} $ and $ M^{\free}_{s,t}$ are formed in the similar manner as $E^{\free}_{s,t}$ and $ F^{\free}_{s,t}$, with $\cCFK(K)^{\bF[W,Z]}$ replaced by the free type-$D$ module over $\bF[U,T,T^{-1}]$ with a single generator $\ys' $. Therefore,  we have 
\[
J^{\free}_{s,t} =\langle\ys'|\ve{j}^{\free}_{s,t}\rangle[Z]= \langle\ys'|\xs^{t+\frac{1}{2}}_{0}\rangle[Z], \quad \text{and} \quad  M^{\free}_{s,t} =\langle\ys'|\ve{m}^{\free}_{s,t}\rangle[Z]= \langle\ys'|\xs'_{0}\rangle[Z].
\]

The structure maps $\Phi^{\pm K}$ and $\Phi^{\pm \mu}$ appearing in Equation~\eqref{eq:X-P-K-n-top-cube} are obtained by tensoring the identity on $\cX_n(K)^{\cK}$ with the maps $f^{\pm K}$ and $f^{\pm \mu}$ (respectively) from Equation~\eqref{eq:bimodule-X-P-n-cube}, i.e.,
\[
\Phi^K=\bI_{\cX_n}\boxtimes f^K,
\]
and similarly for the other maps. The maps $f^{\pm K}$ and $f^{\pm \mu}$ are shown in Figures~\ref{fig:f-pm K} and ~\ref{fig:maps-f-pm-mu}, respectively. These are obtained from the maps in \cite{CZZ}*{Section~9} by setting $W=0$ and then conjugating with the homotopy equivalence ${}_{\cK}\cX(L_P)_{\bF[Z]}\simeq {}_{\cK} \cX^{\diamond}(L_P)_{\bF[Z]}$ from Lemma~\ref{lem:X-min-non-trivial-actions}. In effect, this amounts to taking the description from \cite{CZZ}*{Section~9}, setting $W=0$ on each type-$D$ output, and then deleting any structure maps which increase the algebraic grading.  In Figures~\ref{fig:f-pm K} and ~\ref{fig:maps-f-pm-mu}, we summarize the structure maps on $\scE_{*,t}^{\diamond}$, $\scF_{*,t}^{\diamond},$ $\scJ_{*,t}^{\diamond}$ and $\scM_{*,t}^{\diamond}$, as well as display the morphisms $f^{\pm K}$ and $f^{\pm \mu}$. The description on the free versions ($\scE^{\free}_{*,t}$ and so forth)  is schematically the same.

Explicit descriptions of the structure maps $m_{1|1|0}$ on ${}_{\cK} \cX^{\free}(L_P)_{\mathbb{F}[Z]}$ are described in Lemma~\ref{lem:quotient bimodule in Lp}.

\begin{figure}[h]
\[
	 \begin{tikzcd}[labels=description,row sep=1.4cm] \scE^{\diamond}_{*,t}\ar[d, "f^K"]
	 	\\  \scJ^{\diamond}_{*,t}
	 \end{tikzcd}
	 \hspace{-.2cm}
	 =
	 \hspace{-.3cm}
	 \begin{tikzcd}[ {column sep=2cm,between origins}, row sep=.4cm]
	 	\cdots
	 	&[-1.2cm] \scE^{\diamond}_{-\frac{5}{2},t}
	 	& \scE^{\diamond}_{-\frac{3}{2},t}
	 	\ar[l, "W|1"]
	 	&\scE^{\diamond}_{-\frac{1}{2},t}
	 	\ar[r,bend left,  "Z|L_Z"]
	 	\ar[l,  "W|1"]
	 	\ar[d,"\sigma|L_Z"]
	 	&[.2cm] \scE^{\diamond}_{\frac{1}{2},t}
	 	\ar[r,  "Z|1"]
	 	\ar[l,bend left, "W|L_W", pos=0.55]
	 	\ar[d,"\sigma|1"]
	 	& \scE^{\diamond}_{\frac{3}{2},t}
	 	\ar[r,  "Z|1"]
	 	\ar[d,"\sigma|1"]
	 	& \scE^{\diamond}_{\frac{5}{2},t}
	 	\ar[d,"\sigma|1"]
	 	&[-1.2cm] \cdots
	 	\\[1cm]
	 	\cdots
	 	& \scJ^{\diamond}_{-\frac{5}{2},t}
	 	\ar[r,bend left, "T|1"]
	 	&\scJ^{\diamond}_{-\frac{3}{2},t}
	 	\ar[r,bend left,  "T|1"]
	 	\ar[l,bend left, "T^{-1}|1"]
	 	&\scJ^{\diamond}_{-\frac{1}{2},t}
	 	\ar[r,bend left, "T|1"]
	 	\ar[l,bend left,  "T^{-1}|1"]
	 	&\scJ^{\diamond}_{\frac{1}{2},t}
	 	\ar[r,bend left, "T|1"]
	 	\ar[l,bend left,  "T^{-1}|1"]
	 	&\scJ^{\diamond}_{\frac{3}{2},t}
	 	\ar[r,bend left,  "T|1"]
	 	\ar[l,bend left,  "T^{-1}|1"]
	 	&\scJ^{\diamond}_{\frac{5}{2},t}
	 	\ar[l,bend left, "T^{-1}|1"]
	 	&\cdots 
	 \end{tikzcd}
\]
\[
\begin{tikzcd}[labels=description,row sep=1.4cm] \scE^{\diamond}_{*,t}\ar[d, "f^{-K}"]
	\\  \scJ^{\diamond}_{*,t-1}
\end{tikzcd}
\hspace{-.2cm}
=
\hspace{-.3cm}
\begin{tikzcd}[ {column sep=2cm,between origins}, row sep=.4cm]
	\cdots
	&[-1.2cm] \scE^{\diamond}_{-\frac{5}{2},t}
	\ar[d,"\tau |1"]
	& \scE^{\diamond}_{-\frac{3}{2},t}
	\ar[l, "W|1"]
	\ar[d,"\tau |1"]
	&\scE^{\diamond}_{-\frac{1}{2},t}
	\ar[r,bend left,  "Z|L_Z"]
	\ar[l,  "W|1"]
	\ar[d,"\tau |1"]
	&[.2cm] \scE^{\diamond}_{\frac{1}{2},t}
	\ar[r,  "Z|1"]
	\ar[l,bend left, "W|L_W",pos=0.55]
	\ar[d,"\tau |L_W"]
	& \scE^{\diamond}_{\frac{3}{2},t}
	\ar[r,  "Z|1"]
	& \scE^{\diamond}_{\frac{5}{2},t}
	&[-1.2cm] \cdots
	\\[1cm]
	\cdots
	& \scJ^{\diamond}_{-\frac{5}{2},t}
	\ar[r,bend left, "T|1"]
	&\scJ^{\diamond}_{-\frac{3}{2},t}
	\ar[r,bend left,  "T|1"]
	\ar[l,bend left, "T^{-1}|1"]
	&\scJ^{\diamond}_{-\frac{1}{2},t}
	\ar[r,bend left, "T|1"]
	\ar[l,bend left,  "T^{-1}|1"]
	&\scJ^{\diamond}_{\frac{1}{2},t}
	\ar[r,bend left, "T|1"]
	\ar[l,bend left,  "T^{-1}|1"]
	&\scJ^{\diamond}_{\frac{3}{2},t}
	\ar[r,bend left,  "T|1"]
	\ar[l,bend left,  "T^{-1}|1"]
	&\scJ^{\diamond}_{\frac{5}{2},t}
	\ar[l,bend left, "T^{-1}|1"]
	&\cdots 
\end{tikzcd}
\]
\[
\begin{tikzcd}[ row sep=1.8cm] \scF^{\diamond}_{*,t} \ar[d, "f^{K}"]
	\\ 
	\scM^{\diamond}_{*,t}
\end{tikzcd}
\hspace{-.2cm}
=
\hspace{-.3cm}
\begin{tikzcd}[ column sep=1.2cm, row sep=0cm]
	\cdots
	&[-1.4cm] \scF^{\diamond}_{-\frac{5}{2},t}
	& \scF^{\diamond}_{-\frac{3}{2},t}
	\ar[l, "W|1"]
	&\scF^{\diamond}_{-\frac{1}{2},t}
	\ar[r, "Z|1"]
	\ar[l, "W|1"]
	\ar[d, "\sigma|1"]
	&\scF^{\diamond}_{\frac{1}{2},t}
	\ar[r, "Z|1"]
	\ar[d, "\sigma|1"]
	&\scF^{\diamond}_{\frac{3}{2},t}
	\ar[d, "\sigma|1"]
	&[-1.4cm]\cdots
	\\[1.5cm]
	\cdots&[-1.3cm] \scM^{\diamond}_{-\frac{5}{2},t}
	\ar[r, bend left, "T|1"]
	&\scM^{\diamond}_{-\frac{3}{2},t}
	\ar[r, bend left, "T|1"]
	\ar[l, bend left, "T^{-1}|1"]
	&\scM^{\diamond}_{-\frac{1}{2},t}
	\ar[r, bend left, "T|1"]
	\ar[l, bend left, "T^{-1}|1"]
	&\scM^{\diamond}_{\frac{1}{2},t}
	\ar[r, bend left, "T|1"]
	\ar[l, bend left, "T^{-1}|1"]
	&\scM^{\diamond}_{\frac{3}{2},t}
	\ar[l, bend left, "T^{-1}|1"]
	&[-1.3cm]\cdots 
\end{tikzcd}
\]
\[
\begin{tikzcd}[ row sep=1.5cm] \scF^{\diamond}_{*,t} 
	\ar[d, "f^{-K}"]
	\\ 
	\scM^{\diamond}_{*,t-1}
\end{tikzcd}
\hspace{-.2cm}
=
\hspace{-.3cm}
\begin{tikzcd}[ column sep=1.2cm, row sep=0cm]
	\cdots
	&[-1.4cm] \scF^{\diamond}_{-\frac{5}{2},t}
	\ar[d, "\tau|1"]
	&\scF^{\diamond}_{-\frac{3}{2},t}
	\ar[l,  "W|1"]
	\ar[d, "\tau|1"]
	&\scF^{\diamond}_{-\frac{1}{2},t}
	\ar[r,  "Z|1"]
	\ar[l,  "W|1"]
	\ar[d, "\tau|1"]
	&\scF^{\diamond}_{\frac{1}{2},t}
	\ar[r, "Z|1"]
	&\scF^{\diamond}_{\frac{3}{2},t}
	&[-1.4cm]\cdots
	\\[1.5cm]
	\cdots&[-1.3cm]\scM^{\diamond}_{-\frac{5}{2},t}
	\ar[r, bend left,"T|1"]
	& \scM^{\diamond}_{-\frac{3}{2},t}
	\ar[r, bend left, "T|1"]
	\ar[l, bend left, "T^{-1}|1"]
	&\scM^{\diamond}_{-\frac{1}{2},t}
	\ar[r, bend left,"T|1"]
	\ar[l, bend left, "T^{-1}|1"]
	&\scM^{\diamond}_{\frac{1}{2},t}
	\ar[r, bend left, "T|1"]
	\ar[l, bend left, "T^{-1}|1"]
	&\scM^{\diamond}_{\frac{3}{2},t}
	\ar[l, bend left, "T^{-1}|1"]
	&[-1.3cm]\cdots 
\end{tikzcd}.
\]
\caption{The complexes $\scE_{*,t}^{\diamond}$, $\scF_{*,t}^{\diamond}$, $\scJ_{*,t}^{\diamond}$, $\scM_{*,t}^{\diamond}$ and the maps $f^{\pm K}$.}
\label{fig:f-pm K}
\end{figure}

\begin{figure}[h]
\[
	 \begin{tikzcd}[labels=description, row sep=2cm] \scE^{\diamond}_{*,t} 
	 	\ar[d, "f^{\mu}"]
	 	\\ 
	 	\scF^{\diamond}_{*,t}
	 \end{tikzcd}
	 \hspace{-.3cm}
	 =
	 \hspace{-.3cm}\begin{tikzcd}[column sep=.8cm, row sep=0cm]
	 	\cdots
	 	&[-1cm]\scE^{\diamond}_{-\frac{5}{2},t}
	 	\ar[d,"L_\sigma",  pos=.4]
	 	& \scE^{\diamond}_{-\frac{3}{2},t}
	 	\ar[l ,"W|1"]
	 	\ar[d,"L_\sigma",  pos=.4]
	 	&\scE^{\diamond}_{-\frac{1}{2},t}
	 	\ar[r, bend left,out=35,in=145, "Z|L_Z"]
	 	\ar[l, "W|1"]
	 	\ar[d,"L_\sigma",  pos=.4]
	 	&[.6cm] \scE^{\diamond}_{\frac{1}{2},t}
	 	\ar[r, "Z|1"]
	 	\ar[l, bend left=15,out=35,in=145,pos=.51, "W|L_W"]
	 	\ar[d,"L_\sigma", pos=.4]
	 	& \scE^{\diamond}_{\frac{3}{2},t}
	 	\ar[r, "Z|1"]
	 	\ar[d,"L_\sigma",  pos=.4]
	 	& \scE^{\diamond}_{\frac{5}{2},t}
	 	\ar[d,"L_\sigma", pos=.4]
	 	&[-1cm] \cdots
	 	\\[1.2cm]
	 	\cdots
	 	&[.4cm]
	 	\scF^{\diamond}_{-\frac{5}{2},t}
	 	&\scF^{\diamond}_{-\frac{3}{2},t}
	 	\ar[l,  "W|1"]
	 	&\scF^{\diamond}_{-\frac{1}{2},t}
	 	\ar[r,  "Z|1"]
	 	\ar[l,  "W|1"]
	 	&\scF^{\diamond}_{\frac{1}{2},t}
	 	\ar[r,  "Z|1"]
	 	&\scF^{\diamond}_{\frac{3}{2},t}
	 	\ar[r,  "Z|1"]
	 	&\scF^{\diamond}_{\frac{5}{2},t}
	 	&\cdots 
	 \end{tikzcd}\]
\[
	 \begin{tikzcd}[labels=description, row sep=2.1cm] 
	 	\scE^{\diamond}_{*,t} 
	 	\ar[d, "f^{-\mu}"]
	 	\\ 
	 	\scF^{\diamond}_{*,t}
	 \end{tikzcd}
	 \hspace{-.2cm}
	 =
	 \hspace{-.3cm}\begin{tikzcd}[ column sep={2cm,between origins}, row sep=.3cm]
	 	\cdots
	 	&[-1cm]\scE^{\diamond}_{-\frac{5}{2},t}
	 	& \scE^{\diamond}_{-\frac{3}{2},t}
	 	\ar[l,  "W|1"]
	 	\ar[dl, "L_\tau"]
	 	&\scE^{\diamond}_{-\frac{1}{2},t}
	 	\ar[r, bend left, "Z|L_Z"]
	 	\ar[l,  "W|1"]
	 	\ar[dl, "L_\tau"]
	 	&\scE^{\diamond}_{\frac{1}{2},t}
	 	\ar[r, "Z|1"]
	 	\ar[l, bend left,pos=.6, "W|L_W"]
	 	\ar[dl, "L_\tau"]
	 	&  \scE^{\diamond}_{\frac{3}{2},t}
	 	\ar[r,  "Z|1"]
	 	\ar[dl, "L_\tau"]
	 	&  \scE^{\diamond}_{\frac{5}{2},t}
	 	\ar[dl, "L_\tau"]
	 	&[-1cm] \cdots
	 	\\[1.2cm]
	 	\cdots
	 	&[.4cm] \scF^{\diamond}_{-\frac{5}{2},t}
	 	& \scF^{\diamond}_{-\frac{3}{2},t}
	 	\ar[l,  "W|1"]
	 	&\scF^{\diamond}_{-\frac{1}{2},t}
	 	\ar[r,  "Z|1"]
	 	\ar[l,  "W|1"]
	 	&\scF^{\diamond}_{\frac{1}{2},t}
	 	\ar[r,  "Z|1"]
	 	&\scF^{\diamond}_{\frac{3}{2},t}
	 	\ar[r,  "Z|1"]
	 	&\scF^{\diamond}_{\frac{5}{2},t}
	 	&\cdots 
	 \end{tikzcd}
\]
\[
\begin{tikzcd}[labels=description, row sep=2.1cm] \scJ^{\diamond}_{*,t} 
	\ar[d, "f^{\mu}"]
	\\ 
	\scM^{\diamond}_{*,t}
\end{tikzcd}
\hspace{-.2cm}
=
\hspace{-.3cm}
\begin{tikzcd}[column sep=1cm, row sep=0cm]
	\cdots
	&[-1cm] \scJ^{\diamond}_{ -\frac{5}{2},t}
	\ar[r,bend left, "T|1"]
	\ar[d, "L_\sigma"]
	& \scJ^{\diamond}_{-\frac{3}{2},t}
	\ar[r,bend left, "T|1"]
	\ar[l, bend left, "T^{-1}|1"]
	\ar[d, "L_\sigma"]
	&\scJ^{\diamond}_{-\frac{1}{2},t}
	\ar[r, bend left, "T|1"]
	\ar[l, bend left, "T^{-1}|1"]
	\ar[d, "L_\sigma"]
	&\scJ^{\diamond}_{\frac{1}{2},t}
	\ar[r, bend left, "T|1"]
	\ar[l, bend left, "T^{-1}|1"]
	\ar[d, "L_\sigma"]
	&\scJ^{\diamond}_{\frac{3}{2},t}
	\ar[l, bend left, "T^{-1}|1"]
	\ar[d, "L_\sigma"]
	&[-1cm]\cdots
	\\[1.5cm]
	\cdots&[-1.3cm] \scM^{\diamond}_{-\frac{5}{2},t}
	\ar[r, bend left, "T|1"]
	& \scM^{\diamond}_{-\frac{3}{2},t}
	\ar[r, bend left, "T|1"]
	\ar[l, bend left, "T^{-1}|1"]
	&\scM^{\diamond}_{-\frac{1}{2},t}
	\ar[r, bend left, "T|1"]
	\ar[l, bend left, "T^{-1}|1"]
	&\scM^{\diamond}_{\frac{1}{2},t}
	\ar[r, bend left, "T|1"]
	\ar[l, bend left, "T^{-1}|1"]
	&\scM^{\diamond}_{\frac{3}{2},t}
	\ar[l, bend left, "T^{-1}|1"]
	&[-1.3cm]\cdots 
\end{tikzcd}
\]
\[
\begin{tikzcd}[labels=description, row sep=2.1cm] \scJ^{\diamond}_{*,t} 
	\ar[d, "f^{-\mu}"]
	\\ 
	\scM^{\diamond}_{*,t}
\end{tikzcd}
\hspace{-.2cm}
=
\hspace{-.3cm}
\begin{tikzcd}[ column sep=1cm, row sep=0cm]
	\cdots
	&[-1cm] \scJ^{\diamond}_{ -\frac{5}{2},t}
	\ar[r,bend left, "T|1"]
	& \scJ^{\diamond}_{-\frac{3}{2},t}
	\ar[r,bend left, "T|1"]
	\ar[l, bend left, "T^{-1}|1",pos=0.5]
	\ar[dl, "L_\tau",pos=0.3,start anchor={[yshift=-1ex]},end anchor={[yshift=1ex]}]
	&\scJ^{\diamond}_{-\frac{1}{2},t}
	\ar[r, bend left, "T|1"]
	\ar[l, bend left, "T^{-1}|1"]
	\ar[dl, "L_\tau",pos=0.3,start anchor={[yshift=-1ex]},end anchor={[yshift=1ex]}]
	&\scJ^{\diamond}_{\frac{1}{2},t}
	\ar[r, bend left, "T|1"]
	\ar[l, bend left, "T^{-1}|1"]
	\ar[dl, "L_\tau",pos=0.3,start anchor={[yshift=-1ex]},end anchor={[yshift=1ex]}]
	&\scJ^{\diamond}_{\frac{3}{2},t}
	\ar[l, bend left, "T^{-1}|1"]
	\ar[dl, "L_\tau",pos=0.3,start anchor={[yshift=-1ex]},end anchor={[yshift=1ex]}]
	&[-1cm]\cdots
	\\[1.5cm]
	\cdots&[-1.3cm] \scM^{\diamond}_{-\frac{5}{2},t}
	\ar[r, bend left, "T|1"]
	& \scM^{\diamond}_{-\frac{3}{2},t}
	\ar[r, bend left,"T|1"]
	\ar[l, bend left,  "T^{-1}|1"]
	&\scM^{\diamond}_{-\frac{1}{2},t}
	\ar[r, bend left, "T|1"]
	\ar[l, bend left, "T^{-1}|1"]
	&\scM^{\diamond}_{\frac{1}{2},t}
	\ar[r, bend left, "T|1"]
	\ar[l, bend left, "T^{-1}|1"]
	&\scM^{\diamond}_{\frac{3}{2},t}
	\ar[l, bend left,"T^{-1}|1"]
	&[-1.3cm]\cdots 
\end{tikzcd}
\]
\caption{The complexes $\scE_{*,t}^{\diamond}$, $\scF_{*,t}^{\diamond}$, $\scJ_{*,t}^{\diamond}$, $\scM_{*,t}^{\diamond}$ and the maps $f^{\pm \mu}$.}
\label{fig:maps-f-pm-mu}
\end{figure}

Using the description of $L_{\sigma}$ and $L_{\tau}$ in Lemma~\ref{lem:quotient bimodule in Lp}, we see that  $f^{\mu}: \scE^{\free}_{*,t} \to \scF^{\free}_{*,t}$ is trivial for $t<\frac{\ell}{2}-1$, and $f^{-\mu}: \scE^{\free}_{*,t} \to \scF^{\free}_{*,t}$ is trivial for $t>\frac{\ell}{2}+1$. The behavior is most interesting when $t=(\ell-1)/2$ and $t=(\ell+1)/2$.

\begin{lem}
\item
\begin{enumerate}
\item The module ${}_{\cK} \cH_-^{\cK}\boxtimes {}_{\cK} \cX^{\diamond}(L_P)_{\bF[Z]}$ has only $m_{0|1|0}$, $m_{1|1|0}$ and $m_{0|1|1}$ non-trivial. The same holds for ${}_{\cK} \cH_-^{\cK}\boxtimes {}_{\cK} \cX^{\free}(L_P)_{\bF[Z]}$.
\item The hypercube descriptions ${}_{\cK} \cH_-^{\cK}\boxtimes {}_{\cK} \cX^{\diamond}(L_P)_{\bF[Z]}$ and ${}_{\cK} \cH_-^{\cK}\boxtimes {}_{\cK} \cX^{\free}(L_P)_{\bF[Z]}$ (shown in  Equation~\eqref{eq:bimodule-X-P-n-cube}) have no diagonal differentials. 
\item The hypercube description of $\bX^{\diamond}(P,K,n)_{\bF[Z]}$ and $\bX^{\free}(P,K,n)$ (as in Equation~\eqref{eq:X-P-K-n-top-cube}) have no diagonal differentials.
\end{enumerate}
\label{lem:no-diagonal-differential}
\end{lem}
\begin{proof}The first claim follows from the fact that ${}_{\cK} \cH_-^{\cK}$ has only $\delta_1^1$ and $\delta_2^1$ non-trivial, and ${}_{\cK} \cX^{\diamond}(L_P)_{\bF[Z]}$ has only $m_{1|1|0}$ and $m_{0|1|1}$ non-trivial by Lemma~\ref{lem:X-min-non-trivial-actions}. 

The second claim follows from the first, as follows. Each component of the structure map of the tensor product is obtained by composing the structure maps of ${}_{\cK} \cH_-^{\cK}$. The $\cK$ outputs are then input into the structure map of ${}_{\cK} \cX^{\diamond}(L_P)_{\bF[Z]}$. For a component of the box tensor product differential to contribute to the diagonal map, we would need a $\sigma$ or $\tau$ multiple to be input into the structure maps of $\cH_-$, and then for the composition of structure maps from $\cH_-$ to have a $\sigma$ or $\tau$ output. By inspection, any such composition involves at least the 2-fold composition of the structure maps of $\cH_-$. However this pairs trivially with the structure maps of $\cX^{\diamond}$ and $\cX^{\free}$, since only $m_{1|1|0}$ and $m_{0|1|1}$ are non-trivial on these modules.

The third claim follows from similar logic as the second.
\end{proof}

\subsection{A direct summand of $ \bX^{\diamond}(P,K,n)_{\bF[Z]}$ when $\veps(K)=1$}
\label{sec:direct summand when epsilon =1}

We now consider the complexes $\bX^{\diamond}(P,K,n)_{\bF[Z]}$ when $P$ is an L-space satellite operator and $K$ is a knot in $S^3$ with $\veps(K)=1$. We will describe a direct summand of $\bX^{\diamond}(P,K,n)_{\bF[Z]}$ contained in $\bX^{\free}(P,K,n)_{\bF[Z]}$, which we call $\mathfrak{C}(P,K,n)$, and which supports the infinite $Z$-tower in the homology. Therefore, it is enough to study this direct summand to compute the $\tau$-invariant of $P(K,n)$ in this case. 

To construct and analyze this summand, we will use the hat flavored surgery modules $\cX_n(K)^{\hat{\cK}}$ and the structural result of these modules from Theorem~\ref{thm:standard-cx-over-K}. As a first step, recall that our module ${}_{\cK} \cH_-^{\cK}$ has $U$-equivariant structure maps, and that the action of any element of $\cK$ which is a multiple of $U$ on ${}_{\cK} \cX^{\diamond}(L_P)_{\bF[Z]}$. Since the action of any $U$ multiple vanishes on ${}_{\cK} \cX^{\diamond}(L_P)_{\bF[Z]}$, we may naturally view it as inducing an $AA$-bimodule ${}_{\hat{\cK}} \cX^{\diamond}(L_P)_{\bF[Z]}$, whose actions take the same formula. We observe
\[
\begin{split}
&{}_{\cK} \cH_-^{\cK}\boxtimes {}_{\cK} \cX^{\diamond}(L_P)_{\bF[Z]} \\
\iso&{}_{\cK} \cH_-^{\cK}\boxtimes {}_{\cK}[q_U]^{\hat{\cK}}\boxtimes {}_{\hat{\cK}} \cX^{\diamond}(L_P)_{\bF[Z]}\\
\iso &{}_{\cK}[q_U]^{\hat{\cK}} \boxtimes {}_{\hat{\cK}} \cH_-^{\hat{\cK}}\boxtimes {}_{\hat{\cK}} \cX^{\diamond}(L_P)_{\bF[Z]}
\end{split}
\]
where ${}_{\hat{\cK}} \cH_-^{\hat{\cK}}$ (and so forth) are the modules obtained by setting $U$ to be zero in the inputs and outputs of ${}_{\cK} \cH_-^{\cK}$.  Here, $q_U\colon \cK\to \cK/(U)=\hat{\cK}$ denotes the algebra morphism induced by setting $U=0$, and ${}_{\cK} [q_U]^{\hat{\cK}}$ denotes the corresponding $DA$-bimodule (see \cite{LOTBimodules}*{Definition~2.2.48}). 

In particular, in order to compute $\tau(P(K,n))$, it suffices to work with the hat flavored surgery modules $\cX_n(K)^{\hat{\cK}}$ instead of the full surgery modules $\cX_n(K)^{\cK}$. We will use the description of $\cX_n(K)^{\hat{\cK}}$ in terms of standard complexes, essentially due to Dai, Hom, Stoffregen and Truong, stated above as Theorem~\ref{thm:standard-cx-over-K}. According to this theorem, we write
\begin{equation}
\cX_n(K)^{\hat{\cK}}=\cX_n(b_1,\dots, b_m)^{\hat{\cK}}\oplus A.\label{eq:standard-complex-X_K}
\end{equation}
For the purposes of computing $\tau(P(K,n))$, the summand $A$ plays no role. Therefore, for our purposes it suffices to consider the tensor product
\[
\cX_n(b_1,\dots, b_m)^{\hat{\cK}}\boxtimes {}_{\hat{\cK}} \cH_-^{\hat{\cK}}\boxtimes {}_{\hat{\cK}} \cX^{\diamond}(L_P)_{\bF[Z]}.
\]

We recall from Section~\ref{sec:hat-modules} that $\cX_n(b_1,\dots, b_m)^{\hat{\cK}}$ is the complex
\begin{equation}
	 \begin{tikzcd}[column sep=1.5cm]
	 	\ys_0 \ar[drr, "T^{\tau(K)}\sigma",swap] \ar[r, leftrightarrow, "W^{|b_1|}"] 
	 	& \ys_1 \ar[r, leftrightarrow, "Z^{|b_2|}"] & \ys_2 \ar[r, leftrightarrow]&\cdots \ar[r, leftrightarrow, "Z^{|b_m|}"]& \ys_m
	 	\ar[dll, "T^{n-\tau(K)} \tau"]
	 	\\
	 	&& \ys'
	 \end{tikzcd}.
 \label{eq:standard surgery complex}
\end{equation}

In the above, if $b_i$ is positive, there is a differential from $\ys_{i}$ to $\ys_{i-1}$ (i.e. an arrow to the left). If $b_i$ is negative, there is a differential from $\ys_{i-1}$ to $\ys_i$. Note that Hom's invariant $\veps(K)$ is the sign of $b_1$ if $m>0$. Furthermore, $\veps(K)=0$ if and only if $m=0$.

\begin{define}
	Define \[
	\mathfrak{C}(P,K,n)
	\] to be the free $\bF[Z]$-submodule of $\bX^{\diamond}(P,K,n)_{\bF[Z]}$ spanned by the following generators:
	\begin{enumerate}
		\item  $\ys_0| \xs^{\sfrac{\ell}{2}}_{0}=\ys_0|\ve{e}^{\free}_{s-\tau(K),\frac{\ell-1}{2}}\in \langle \ys_0\rangle\otimes \scE^{\free}_{s-\tau(K),\frac{\ell-1}{2}}\subset E_{s,\frac{\ell-1}{2}}^{\free}$ for $s>\tau(K)$;
		\item $\ys_0| \xs'_0=\ys_0|\ve{f}^{\free}_{s-\tau(K),\frac{\ell-1}{2}}\in \langle \ys_0\rangle\otimes \scF^{\free}_{s-\tau(K),\frac{\ell-1}{2}}\subset F_{s,\frac{\ell-1}{2}}^{\free}$ for $s >-1+\tau(K)$;
		\item $\ys'| \xs_{0}^{\sfrac{\ell}{2}} =\ys'|\ve{j}^{\free}_{s,\frac{\ell-1}{2}}\in \langle \ys'\rangle \otimes \scJ^{\free}_{s,\frac{\ell-1}{2}}\subset J_{s,\frac{\ell-1}{2}}^{\free}$ for all $s\in\Z+1/2$;
		\item  $\ys'| \xs'_0=\ys'|\ve{m}^{\free}_{s,\frac{\ell-1}{2}}\in \langle \ys'\rangle \otimes\scM^{\free}_{s,\frac{\ell-1}{2}} \subset M_{s,\frac{\ell-1}{2}}^{\free}$ for all $s\in\Z+1/2$;
		\item $\ys_m| \xs^{\sfrac{\ell}{2}}_0=\ys_m| \ve{e}^{\free}_{s+\tau(K),\frac{\ell+1}{2}}\in \langle \ys_m\rangle \otimes \scE^{\free}_{s+\tau(K),\frac{\ell+1}{2}} \subset E_{s,\frac{\ell+1}{2}}^{\free} $ for $s< -\tau(K)$;
		\item $\ys_m| \xs'_0 =\ys_m| \ve{f}^{\free}_{s+\tau(K),\frac{\ell+1}{2}}\in\langle \ys_m\rangle \otimes \scF^{\free}_{s+\tau(K),\frac{\ell+1}{2}} \subset F_{s,\frac{\ell+1}{2}}^{\free}$ for $ s<-\tau(K) $.
	\end{enumerate}
\label{def:direct summand when epsilon=1}
\end{define}

Note that $\mathfrak{C}(P,K,n)$ is contained in $\bX^{\free}(P,K,n)$. We view $\mathfrak{C}(P,K,n)$ as a free chain complex over $\bF[Z]$ by using all of the components of the differential on $\bX^{\free}(P,K,n)$ between the generators of $\frC(P,K,n)$. These are described in Section~\ref{sec:maps in X top}.

See Figures \ref{fig:direct summand n less than 2tau} and \ref{fig:direct summand n greater than 2tau} for  schematics  of  $\mathfrak{C}(P,K,n)$ when $n<2\tau(K)$ and $n\ge 2\tau(K)$, respectively. Note that there is a qualitative difference between these two cases, which explains the difference in the formulas of $\tau(P(K,n))$ in Theorem~\ref{thm:main-intro} depending on whether $n<2\tau(K)$ or $n \ge 2 \tau(K)$. 

All the vertical maps are identity maps on $\xs_{0}^{\sfrac{\ell}{2}}$ or $\xs'_0$. All the leftward pointing arrows labeled by $L_{\tau}$ are given by the formula  
\[
L_{\tau}(\xs^{\sfrac{\ell}{2}}_0) = \xs'_0\cdot Z^{ R_{\sfrac{\ell}{2}}+\frac{\ell}{2}-g_3(P)}.
\]
All of the rightward arrows labeled by $L_{\sigma}$ are given by 
\[
L_{\sigma}(\xs^{\sfrac{\ell}{2}}_0) = \xs'_0\cdot Z^{ R_{\sfrac{\ell}{2}}-\frac{\ell}{2}-g_3(P)}.
\]
These formulas are derived in Lemma~\ref{lem:quotient bimodule in Lp}.

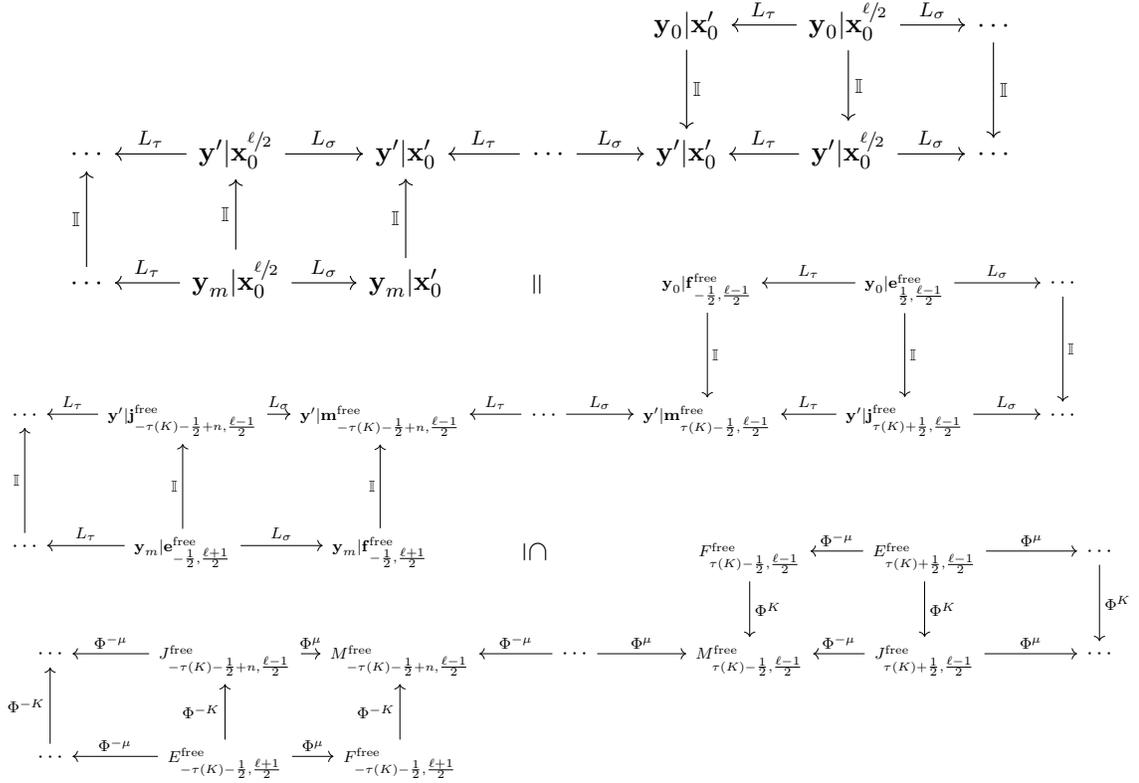
\begin{figure}[!h]
		\adjustbox{scale=0.95}{
			\begin{tikzcd}[row sep=10mm]
				&\phantom{E}& \phantom{F}&\phantom{\cdots}& \ys_0|\xs_0' \arrow[d,"\bI"] & \ys_0| \xs_0^{\sfrac{\ell}{2}}\arrow[l,swap,"L_{\tau}"]\arrow[r,"L_{\sigma}"] \arrow[d,"\bI"] &\cdots\arrow[d,"\bI"]\\
				\cdots &\ys'|\xs_0^{\sfrac{\ell}{2}}\arrow[l,swap,"L_{\tau}"]\arrow[r,"L_{\sigma}"]& \ys'|\xs'_0 & \cdots\arrow[l,swap,"L_{\tau}"] \arrow[r,"L_{\sigma}"]& \ys'|\xs_0' & \ys'|\xs_0^{\sfrac{\ell}{2}}  \arrow[l,swap,"L_{\tau}"] \arrow[r,"L_{\sigma}"] & \cdots\\
				\cdots \arrow[u,"\bI"] & \ys_m|\xs_0^{\sfrac{\ell}{2}} \arrow[l,swap,"L_{\tau}"] \arrow[r,"L_{\sigma}"] \arrow[u,"\bI"]& \ys_m|\xs_0'\arrow[u,"\bI"] & & & & 
			\end{tikzcd}
		}
	\\	
			\vspace{-.5cm}
	\rotatebox{270}{$=$}
			\vspace{-.5cm}
		\\
	\adjustbox{scale=0.75}{
					\begin{tikzcd}[column sep={3.5cm,between origins},row sep=1.5cm]
						&[-.7cm]\phantom{E}& \phantom{F}
						&[-.6cm]\phantom{\cdots}
						&[-.6cm] \Ss{\ys_0|\ve{f}_{-\frac{1}{2},\frac{\ell-1}{2}}^{\free}} \arrow[d,"\bI"] & 
						\Ss{\ys_0| \ve{e}_{\frac{1}{2},\frac{\ell-1}{2}}^{\free}}\arrow[l,swap,"L_{\tau}"]\arrow[r,"L_{\sigma}"] \arrow[d,"\bI"] &[-.7cm]\cdots\arrow[d,"\bI"]\\
						\cdots &
						\Ss{\ys'| \ve{j}_{-\tau(K)-\frac{1}{2}+n,\frac{\ell-1}{2}}^{\free}}
						\arrow[l,swap,"L_{\tau}"]
						\arrow[r,"L_{\sigma}"]
						&
						 \Ss{\ys'|\ve{m}_{-\tau(K)-\frac{1}{2}+n,\frac{\ell-1}{2}}^{\free}} & \cdots\arrow[l,swap,"L_{\tau}"] \arrow[r,"L_{\sigma}"]& \Ss{\ys'|\ve{m}_{\tau(K)-\frac{1}{2},\frac{\ell-1}{2}}^{\free}}
						 &
						  \Ss{\ys'|\ve{j}_{\tau(K)+\frac{1}{2},\frac{\ell-1}{2}}^{\free}}  \arrow[l,swap,"L_{\tau}"] \arrow[r,"L_{\sigma}"] & \cdots
						  \\
						\cdots \arrow[u,"\bI"] & \Ss{\ys_m|\ve{e}_{-\frac{1}{2},\frac{\ell+1}{2}}^{\free}} \arrow[l,swap,"L_{\tau}"] \arrow[r,"L_{\sigma}"] \arrow[u,"\bI"]& \Ss{\ys_m|\ve{f}_{-\frac{1}{2},\frac{\ell+1}{2}}^{\free}}
						\arrow[u,"\bI"] & & & & 
					\end{tikzcd}
				}
		\\
		\vspace{-.5cm}
					 \rotatebox{270}{$\subset$} 
		\vspace{-.5cm}
		\\
		  \adjustbox{scale=0.75}{
				\begin{tikzcd}[column sep={3.1cm,between origins},row sep=10mm]
					&\phantom{E}& \phantom{F}&\phantom{\cdots}&
					 \Ss{F^{\free}_{\tau(K)-\frac{1}{2},\frac{\ell-1}{2}} }
					 	\arrow[d,"\Phi^{K}"] 
					 &
					 \Ss{E^{\free}_{\tau(K)+\frac{1}{2},\frac{\ell-1}{2}}}
					 	\arrow[l,"\Phi^{-\mu}",swap]
					 	\arrow[r,"\Phi^\mu"]
					 	\arrow[d,"\Phi^K"] 
					 	&\cdots\arrow[d,"\Phi^K"]
					 	\\
					\cdots &
					\Ss{J^{\free}_{-\tau(K)-\frac{1}{2}+n,\frac{\ell-1}{2}}}\arrow[l,swap,"\Phi^{-\mu}"]\arrow[r,"\Phi^\mu"]& 
					\Ss{M^{\free}_{-\tau(K)-\frac{1}{2}+n,\frac{\ell-1}{2}}} & \cdots\arrow[l,"\Phi^{-\mu}",swap] \arrow[r,"\Phi^\mu"]& \Ss{M^{\free}_{\tau(K)-\frac{1}{2},\frac{\ell-1}{2}}} & \Ss{J^{\free}_{\tau(K)+\frac{1}{2},\frac{\ell-1}{2}}} \arrow[l,"\Phi^{-\mu}",swap] \arrow[r,"\Phi^\mu"] & \cdots\\
					\cdots \arrow[u,"\Phi^{-K}"] & \Ss{E^{\free}_{-\tau(K)-\frac{1}{2},\frac{\ell+1}{2}}} \arrow[l,"\Phi^{-\mu}",swap] \arrow[r,"\Phi^\mu"] \arrow[u,"\Phi^{-K}"]& \Ss{F^{\free}_{-\tau(K)-\frac{1}{2},\frac{\ell+1}{2}}}\arrow[u,"\Phi^{-K}"] & & & & 
				\end{tikzcd}
			}
		
	\caption{A schematic drawing of $\mathfrak{C}(P,K,n)_{\bF[Z]}$ when $n<2\tau(K)$. On the top diagram, we describe the generators in the definition of $\frC(P,K,n)$. On the bottom diagram, we describe the corresponding complexes $E_{s,t}^{\free}$, $F_{s,t}^{\free}$, $J_{s,t}^{\free}$ and $M_{s,t}^{\free}$ that these generators live in.}
	\label{fig:direct summand n less than 2tau}
\end{figure}

\begin{figure}[!h]
	\adjustbox{scale=0.9}{
		\begin{tikzcd}[
		row sep=10mm]
	&\phantom{E}
	& \ys_0|\xs'_0
		\arrow[d,"\bI"]
	&\cdots
		\arrow[l,swap,"L_{\tau}"]
		\arrow[r,"L_{\sigma}"]
		 \arrow[d,"\bI"]
	& \ys_0|\xs'_0 
		\arrow[d,"\bI"]
	&\ys_0|\xs^{\sfrac{\ell}{2}}_0
		\arrow[l,swap,"L_{\tau}"]
		\arrow[r,"L_{\sigma}"]
		\arrow[d,"\bI"]
	&\cdots
		\arrow[d,"\bI"]\\
	\cdots 
	&
	\ys'|\xs^{\sfrac{\ell}{2}}_0
		\arrow[l,swap,"L_{\tau}"]
		\arrow[r,"L_{\sigma}"]
	& \ys'|\xs'_0 
	& \cdots
		\arrow[l,swap,"L_{\tau}"]
		\arrow[r,"L_{\sigma}"]
	& 
	\ys'|\xs'_0 
	& \ys'|\xs^{\sfrac{\ell}{2}}_0 
		\arrow[l,swap,"L_{\tau}"]
		\arrow[r,"L_{\sigma}"]
		&
	\cdots\\
	\cdots 
		\arrow[u,"\bI"] 
		&
	\ys_m|\xs^{\sfrac{\ell}{2}}_0
		\arrow[l,swap,"L_{\tau}"]
		\arrow[r,"L_{\sigma}"]
		\arrow[u,"\bI"]& 
	 \ys_m|\xs'_0
	  	\arrow[u,"\bI"] 
	  &
	 \cdots
	 	\arrow[l,swap,"L_{\tau}"]
	 	\arrow[r,"L_{\sigma}"]
	 	\arrow[u,"\bI"]&
	 \ys_m|\xs'_0
	 	 \arrow[u,"\bI"] & & 
				\end{tikzcd}
		}
	\\
	\rotatebox{270}{$\subset$}  \adjustbox{scale=0.75}{
		\begin{tikzcd}[column sep={3cm,between origins},row sep=10mm]
			&\phantom{E}& \Ss{F^{\free}_{\tau(K)-\frac{1}{2},\frac{\ell-1}{2}}}\arrow[d,"\Phi^{K}"]&\cdots\arrow[l,swap,"\Phi^{-\mu}"]\arrow[r,"\Phi^{\mu}"] \arrow[d,"\Phi^{K}"]& \Ss{F^{\free}_{-\tau(K)-\frac{1}{2}+n,\frac{\ell-1}{2}}}\arrow[d,"\Phi^{K}"] &\Ss{E^{\free}_{-\tau(K)+\frac{1}{2}+n,\frac{\ell-1}{2}}}\arrow[l,swap,"\Phi^{-\mu}"]\arrow[r,"\Phi^{\mu}"] \arrow[d,"\Phi^{K}"] &\cdots\arrow[d,"\Phi^{K}"]\\
		\cdots &\Ss{J^{\free}_{\tau(K)-\frac{1}{2},\frac{\ell-1}{2}}}\arrow[l,swap,"\Phi^{-\mu}"]\arrow[r,"\Phi^{\mu}"]& \Ss{M^{\free}_{\tau(K)-\frac{1}{2},\frac{\ell-1}{2}} }& \cdots\arrow[l,swap,"\Phi^{-\mu}"] \arrow[r,"\Phi^{\mu}"]&\Ss{M^{\free}_{-\tau(K)-\frac{1}{2}+n,\frac{\ell-1}{2}}}& \Ss{J^{\free}_{-\tau(K)+\frac{1}{2}+n,\frac{\ell-1}{2}}} \arrow[l,swap,"\Phi^{-\mu}"] \arrow[r,"\Phi^{\mu}"] & \cdots\\
		\cdots \arrow[u,"\Phi^{-K}"] & \Ss{E^{\free}_{\tau(K)-\frac{1}{2}-n,\frac{\ell}{2}}} \arrow[l,swap,"\Phi^{-\mu}"] \arrow[r,"\Phi^{\mu}"] \arrow[u,"\Phi^{-K}"]& \Ss{F^{\free}_{\tau(K)-\frac{1}{2}-n,\frac{\ell}{2}}}\arrow[u,"\Phi^{-K}"] & \cdots \arrow[l,swap,"\Phi^{-\mu}"] \arrow[r,"\Phi^{\mu}"] \arrow[u,"\Phi^{-K}"]& \Ss{F^{\free}_{-\tau(K)-\frac{1}{2},\frac{\ell}{2}}} \arrow[u,"\Phi^{-K}"] & & 
		\end{tikzcd}
	}
	\caption{A schematic drawing of $\mathfrak{C}(P,K,n)_{\bF[Z]}$ when $n\ge2\tau(K)$, which is obtained from the previous diagram by sliding the bottom row relatively to the right.}
	\label{fig:direct summand n greater than 2tau}
\end{figure}
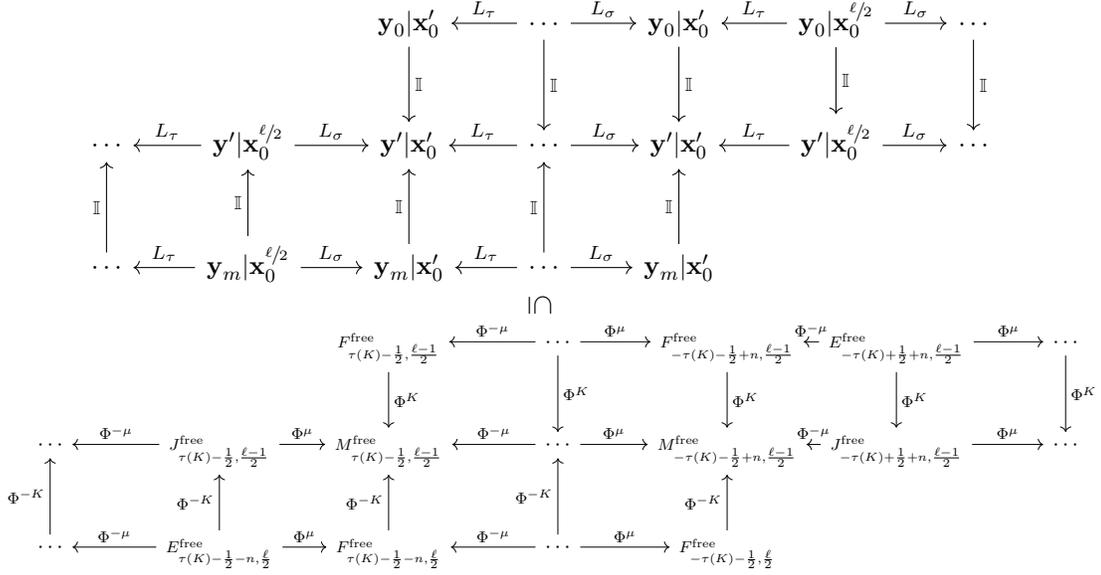

\begin{prop}\label{prop:summand-veps=1} When $\veps(K)=1$, the complex $\mathfrak{C}(P,K,n)_{\bF[Z]}$ is a direct summand of $\bX^{\diamond}(P,K,n)_{\bF[Z]}$.
\end{prop}

We break the proof of Proposition~\ref{prop:summand-veps=1} into two lemmas. In  Lemma~\ref{lem:quotient complex of X(P,K,n)}, we prove that when $\veps(K)=1$, $\mathfrak{C}(P,K,n)$ is a quotient complex. In Lemma~\ref{lem:subcomplex of X(P(K,n))}, we prove that when $\veps(K)=1$, $\mathfrak{C}(P,K,n)$ is a subcomplex. Combining these yields Proposition~\ref{prop:summand-veps=1}.

\begin{lem}
If $K\subset S^3$ is a knot with $\veps(K)=1$, then $\mathfrak{C}(P,K,n)_{\bF[Z]}$ is a quotient complex of $\bX^{\diamond}(P,K,n)_{\bF[Z]}$.
	\label{lem:quotient complex of X(P,K,n)}
\end{lem}

\begin{proof}
	Since $\bX^{\free}(P,K,n)_{\bF[Z]}$ is a quotient complex of  $\bX^{\diamond}(P,K,n)_{\bF[Z]}$ by Lemma~\ref{lem:X-tors-submodule}, and $ \mathfrak{C}(P,K,n)$ is a quotient module of $\bX^{\free}(P,K,n)$, it suffices to show that $\mathfrak{C}(P,K,n)$ is a quotient complex of $\bX^{\free}(P,K,n)_{\bF[Z]}$, i.e., there are no non-trivial differentials $m_{0|1|0}$ pointing from other generators in $\bX^{\free}(P,K,n)_{\bF[Z]}$ to $\mathfrak{C}(P,K,n)$.
	
	Throughout the proof, we will write
		\[
	C(b_1,\dots, b_m)
	\]
for the standard complex summand of $\cCFK(K)^{\hat{R}}\simeq C(b_1,\dots, b_m)\oplus A$ for the companion knot $K$, where $b_1,\dots, b_m$ are as in Equation~\eqref{eq:standard-complex-X_K}. By slight abuse notation, we will assume
\[
\cCFK(K)^{\hat{R}}=C(b_1,\dots, b_m).
\]

	By Lemma~\ref{lem:no-diagonal-differential}, there are  no length $2$ differentials in $\bX^{\free}(P,K,n)_{\bF[Z]}$. Therefore, it is enough to consider length $0$ and length 1 differentials. The length 0 differentials are the internal differentials of each $E^{\free}_{s,t}, F^{\free}_{s,t}, J^{\free}_{s,t}, M^{\free}_{s,t}$. The length $1$ differentials are the maps $\Phi^{\pm K}$ and $\Phi^{\pm \mu}$ discussed in Section \ref{sec:maps in X top}.
	
	We  consider the length $0$ differentials first. The claim is obvious for $J^{\free}_{s,\frac{\ell-1}{2}}$ and $M^{\free}_{s,\frac{\ell-1}{2}}$, since the only generator in $J^{\free}_{s,\frac{\ell-1}{2}}$ is $\ys'|\xs^{\sfrac{\ell}{2}}_0 $, and similarly for $M^{\free}_{s,\frac{\ell-1}{2}}$, and there are no internal differentials to consider.
	
	For $E^{\free}_{s,\frac{\ell-1}{2}}$ with $s>\tau(K)$, by the definition of the standard complex, the only differential in $C(b_1,\dots ,b_m)$  pointing to $\ys_0$ is 
	\[
	\begin{tikzcd} \ys_0 & \ys_1 \ar[l, "W^{b_1}",swap] \end{tikzcd}.
	\] 
	Note $b_1$ is positive when $\veps(K)=1$. Note also that the Alexander grading of $\ys_0$ is $A(\ys_0) = \tau(K)$. From the description of $\scE^{\free}_{*,\frac{\ell-1}{2}}$ in Figure~\ref{fig:f-pm K}, we see that the arrow $\ys_0\xleftarrow{W^{b_1}}\ys_1$ becomes $0$ in $E^{\free}_{s,\frac{\ell-1}{2}}$ for $s>\tau(K)$. 
	 Therefore, there is no arrow pointing to $\ys_0|\xs^{\sfrac{\ell}{2}}_0$ in $E^{\free}_{s,\frac{\ell-1}{2}}$ for $s>\tau(K)$. The same observation applies to $F^{\free}_{s,\frac{\ell-1}{2}}$ for $s>-1+\tau(K)$ by looking at the expression of $\scF^{\free}_{*,\frac{\ell-1}{2}}$ in Figure~\ref{fig:f-pm K}.  
	
	The corresponding statement about $E^{\free}_{s,\frac{\ell+1}{2}} $ and $F^{\free}_{s,\frac{\ell+1}{2}} $ for $s<-\tau(K)$ can be proven by in a similar manner.

	For the length $1$ differentials (i.e. $\Phi^{\pm K}$ and $\Phi^{\pm \mu}$), we begin by considering the subspace of $\mathfrak{C}(P,K,n)_{\bF[Z]}$ in the row consisting of $E^{\free}_{*,\frac{\ell-1}{2}}$ and $F^{\free}_{*,\frac{\ell-1}{2}}$. There are no vertical arrows $\Phi^{\pm K}$ pointing into this row, so we only need to consider the horizontal arrows $\Phi^{\pm \mu}$. The maps $\Phi^{\pm \mu}$ are described in Section~\ref{sec:maps in X top}. More explicitly, the map 
	\[
	\Phi^{\mu}\colon  E^{\free}_{s,\frac{\ell-1}{2}} \to F^{\free}_{s,\frac{\ell-1}{2}}
	\] 
	is given by 
	\[
	\Phi^\mu=\bI \boxtimes f^{\mu}: C(b_1,\dots, b_m) \boxtimes \scE^{\free}_{*,\frac{\ell-1}{2}} \to C(b_1,\dots, b_m) \boxtimes \scF^{\free}_{*,\frac{\ell-1}{2}},
	\]
	where $f^{\mu}$ is described in Figure~\ref{fig:maps-f-pm-mu}. Note that $f^{\mu}$ only has one term, $(f^\mu)_1^1$, and this acts by $L_{\sigma}$. Therefore the only possible rightward horizontal arrow that can map to $\ys_0|\xs'_0 =\ys_0|\ve{f}^{\free}_{s-\tau(K),\frac{\ell-1}{2}}$ in $F^{\free}_{s,\frac{\ell-1}{2}}$ is \[
	 \bI\otimes L_{\sigma}: \ys_0|\ve{e}^{\free}_{s-\tau(K),\frac{\ell-1}{2}} \to \ys_0|\ve{f}^{\free}_{s-\tau(K),\frac{\ell-1}{2}}. \]
	 
	For $s>\tau(K),$ we have already included such arrows in the definition of $\mathfrak{C}(P,K,n)_{\bF[Z]}$. For $s = \tau(K)-\frac{1}{2}$, the corresponding map is 
	 \[\bI\otimes L_{\sigma}: \ys_0|\ve{e}^{\free}_{-\frac{1}{2},\frac{\ell-1}{2}} = \ys_{0}|\xs^{\sfrac{\ell}{2}-1}_0 \to \ys_0|\ve{f}^{\free}_{-\frac{1}{2},\frac{\ell-1}{2}} = \ys_{0}|\xs'_0, \]
	 but $L_{\sigma}(\xs_{0}^{\sfrac{\ell}{2}-1})=0$ in 
	${}_{\hat{\cK}} \cX^{\free}(L_P)_{\bF[Z]}$ by Lemma~\ref{lem:quotient bimodule in Lp}.

	For the leftward horizontal arrows, the argument is similar, and this time we have included all the possible non-trivial maps \[\bI\otimes L_{\tau }: \ys_0| \ve{e}^{\free}_{s-\tau(K),\frac{\ell-1}{2}} \to  \ys_0| \ve{f}^{\free}_{s-\tau(K)-1,\frac{\ell-1}{2}}\] for $s>\tau(K)$ in  $\mathfrak{C}(P,K,n)^{\bF[Z]}$.

	Again, the statement about the part of $\mathfrak{C}(P,K,n)_{\bF[Z]}$ in the row consisting of  $E^{\free}_{*,\frac{\ell+1}{2}}$ and $F^{\free}_{*,\frac{\ell+1}{2}}$  could be obtained similarly.
	
	Finally, let's consider length-$1$ maps pointing to the row consisting of $J^{\free}_{*,\frac{\ell-1}{2}}$ and $M^{\free}_{*,\frac{\ell-1}{2}}$ in $\bX^{\free}(P,K,n)$. There is nothing to prove for the horizontal maps $\Phi^{\pm \mu}$, as the whole row of $\bX^{\free}(P,K,n)^{\bF[Z]}$ consisting of $J^{\free}_{*,\frac{\ell-1}{2}}$ and $M^{\free}_{*,\frac{\ell-1}{2}}$ is contained in $\mathfrak{C}(P,K,n)_{\bF[Z]}$.
	
	The vertical maps $\Phi^K: E^{\free}_{s,\frac{\ell-1}{2}} \to J^{\free}_{s,\frac{\ell-1}{2}}$ are given by $\bI\boxtimes f^K$, where $f^K$ is described in Figure~\ref{fig:f-pm K}. If a generator $ \ys_k| \ve{e}^{\free}_{s-A(\ys_k),\frac{\ell-1}{2}}$ has a non-trivial arrow of $\Phi^K$ pointing to $\ys'| \ve{j}^{\free}_{s,\frac{\ell-1}{2}}$, then $k=0$, since $\ys_0$ is the only generator of $C(b_1,\dots, b_m)\subset \cX_n(b_1,\dots ,b_m)^{\hat{\cK}}$ with a non-trivial $\delta^1$ whose coefficient is a multiple of $\sigma$.

	Therefore, it is enough to consider $\Phi^K(\ys_0 | \ve{e}^{^{\free}}_{s-\tau(K),\frac{\ell-1}{2}})$ for $s\in \bZ+\frac{1}{2}$. 
	
	When $s-\tau(K)>0$, the non-trivial arrows are already contained in $\frC(P,K,n)$.
 
	When $s-\tau(K)<-1$, the map $\Phi^K$ sends $\ys_{0}|\ve{e}^{\free}_{s-\tau(K),\frac{\ell-1}{2}} $ to zero because $f^K$ vanishes on $\scE_{s-\tau(K), \frac{\ell-1}{2}}^{\free}$ by its description in Figure~\ref{fig:f-pm K}.
	
	When $s-\tau(K)=-\tfrac{1}{2}$, 
	we have 
	\[
	\Phi^K(\ys_0 | \ve{e}^{^{\free}}_{-\frac{1}{2},\frac{\ell-1}{2}}) = \Phi^K(\ys_0 | \xs^{\sfrac{\ell}{2}-1}_0)= \ys_0| L_Z(\xs^{\sfrac{\ell}{2}-1}_0)=0,
	\]
	by Lemma~\ref{lem:quotient bimodule in Lp}.

	The proof for $\Phi^{K}:F^{\free}_{s,\frac{\ell-1}{2}} \to M^{\free}_{s,\frac{\ell-1}{2}}$ is similar. The only difference is this time we allow $s$ to be $\tau(K)-\frac{1}{2}$, but that non-trivial arrow is already included in $\mathfrak{C}(P,K,n)_{\bF[Z]}$ as well.
	
	The proof for $\Phi^{-K}: E^{\free}_{s,\frac{\ell+1}{2}} \to J^{\free}_{s+n,\frac{\ell-1}{2}}$ and $\Phi^{-K}: F^{\free}_{s,\frac{\ell+1}{2}} \to M^{\free}_{s+n,\frac{\ell-1}{2}}$ is similar, so we leave the details to the reader.
\end{proof}

\begin{lem}
If $K\subset S^3$ is a knot with $\veps(K)=1$, then $ \mathfrak{C}(P,K,n)_{\bF[Z]}$ is a subcomplex of $\bX^{\diamond}(P,K,n)_{\bF[Z]}$.
	\label{lem:subcomplex of X(P(K,n))}
\end{lem}

\begin{proof}
Unlike in Lemma~\ref{lem:quotient complex of X(P,K,n)}, it is not sufficient to consider only $\bX^{\free}(P,K,n)^{\bF[Z]}$, because $\bX^{\free}(P,K,n)$ is not in general a subcomplex of $\bX^{\diamond}(P,K,n)$.

We begin by considering the subspace of $\mathfrak{C}(P,K,n)$ in the row of $J^{\diamond}_{*,\frac{\ell-1}{2}}$ and $M^{\diamond}_{*,\frac{\ell-1}{2}}$ first. There are no length-$2$ arrows or vertical length-$1$ arrows leaving the $J^{\diamond}_{*,*}$ and $M^{\diamond}_{*,*}$ rows, so we only need to consider the maps $\Phi^{\pm \mu}$ (i.e. the maps in a given row).

For the length-$0$ maps (the internal differentials of $J^{\diamond}_{s,t}$ and $M^{\diamond}_{s,t}$), recall that $J^{\diamond}_{s,\frac{\ell-1}{2}}$ is a copy of $\cC^{\diamond}_{\sfrac{\ell}{2}}$, and $M^{\diamond}_{s,\frac{\ell-1}{2}}$ is a copy of $\cS^{\diamond}$, so there is no differential leaving the top generators $\xs^{\sfrac{\ell}{2}}_0$ and $\xs'_0$ in each of the staircases.

The horizontal length-$1$ maps $\Phi^{\mu}:J_{s,\frac{\ell-1}{2}} \to M_{s,\frac{\ell-1}{2}}$ and $\Phi^{-\mu}:J_{s,\frac{\ell-1}{2}} \to M_{s-1,\frac{\ell-1}{2}}$,  are given by \[
L_{\sigma}: \cC_{\sfrac{\ell}{2}}^{\diamond} \to \cS^{\diamond} \text{ and }L_{\tau}:\cC_{\sfrac{\ell}{2}}^{\diamond} \to \cS^{\diamond}
\]
respectively.
 It follows from Lemma \ref{lem:quotient bimodule in Lp} that the $\xs_0'$ coefficients of
\[
L_\sigma(\xs_0^{\sfrac{\ell}{2}}),L_\tau(\xs_0^{\sfrac{\ell}{2}})\in \cS^{\diamond}
\]
are both non-zero. The map $L_{\sigma}$ intertwines the $\gr_{\ws}$-grading on $\cC_{\sfrac{\ell}{2}}^{\diamond}$ with the $\gr_{\ws}$-grading on $\cS^{\diamond}$. The map $L_\tau$ has a more complicated interaction with the gradings, however it is homogeneously graded with respect to $\gr_{\ws}$ and $\gr_{\zs}$ in each fixed Alexander grading. See Equation~\eqref{eq:grading shifts of left action}, and \cite{CZZ}*{Section~5.4} for more details. In particular, the other components of $L_{\sigma}(\xs_0^{\sfrac{\ell}{2}})$ (resp. $L_{\tau}(\xs_0^{\sfrac{\ell}{2}})$) must have the same $\gr_{\ws}$-grading as the $\xs_0'$ component of $L_{\sigma}(\xs_0^{\sfrac{\ell}{2}})$ (resp. $L_{\tau}(\xs_0^{\sfrac{\ell}{2}})$). However, all the other generators of $\cS$ with algebraic grading 0, i.e., $\xs_{2i}'$ for $i>0$, have $\gr_{\ws}$-grading strictly smaller than $\xs_0'$. Therefore, these generators cannot appear as components of $L_{\sigma}(\xs_0^{\sfrac{\ell}{2}})$ or $L_{\tau}(\xs_0^{\sfrac{\ell}{2}})$.
This shows that the length-$1$ maps in $\bX^{\diamond}(P,K,n)_{\bF[Z]}$ in the $J^{\diamond}_{*,\frac{\ell-1}{2}}$ and $M^{\diamond}_{*,\frac{\ell-1}{2}}$ will send the top generators to the top generators, as we need. Therefore, there are no differentials leaving the part of $\mathfrak{C}(P,K,n)$ in the row of $J^{\diamond}_{*,\frac{\ell-1}{2}}$ and $M^{\diamond}_{*,\frac{\ell-1}{2}}$.

Consider the part of $\mathfrak{C}(P,K,n)$ in the row consisting of $E^{\diamond}_{*,\frac{\ell-1}{2}}$ and $F^{\diamond}_{*,\frac{\ell-1}{2}}$. Note that the corresponding statement about the subspace in the row of $E^{\diamond}_{*,\frac{\ell+1}{2}}$ and $F^{\diamond}_{*,\frac{\ell+1}{2}}$ could be obtained similarly. We consider the length-$0$ maps first, which are maps within each $E_{s,\frac{\ell-1}{2}}$ and $F_{s,\frac{\ell-1}{2}}$. Recall, by the description of the standard complex $C(b_1,\dots ,b_m)$, there is no differential leaving $\ys_0$ in $C(b_1,\dots ,b_m)$, so there will be no length-$0$ maps in $\bX^{\diamond}(P,K,n)_{\bF[Z]} $ starting at $\ys_0 | \ve{e}^{\free}_{s-\tau(K),\frac{\ell-1}{2}}$ in $E^{\diamond}_{s,\frac{\ell-1}{2}}$ or  $\ys_0| \ve{f}^{\free}_{s-\tau(K),\frac{\ell-1}{2}}$ in $F^{\diamond}_{s,\frac{\ell-1}{2}}$.

For length-$1$ maps, we need to consider both the horizontal maps $\Phi^{\pm \mu}$ and the vertical maps $\Phi^{\pm K}$ this time. For the horizontal maps, $\Phi^{\pm \mu}$, these are given as tensor products of the form $\bI\boxtimes f^{\pm \mu}$. The maps $f^{\pm \mu}$ are $AA$-module maps, but are only non-trivial on one input, i.e. $(f^{\pm \mu})_{i|1|j}=0$ unless $i=j=0$. In particular, the map $\Phi^\mu$ is given by $\bI\otimes L_\sigma$, and similarly for $\Phi^{-\mu}$. See Remark 9.2 in \cite{CZZ} for the details.

By the restriction on the $s$-coordinate in $\mathfrak{C}(P,K,n)_{\bF[Z]}$, we are dealing with 
\[L_{\sigma}: \scE_{s,\frac{\ell-1}{2}}^{\diamond} \to \scF_{s,\frac{\ell-1}{2}}^{\diamond} \text{ and  } L_{\tau}: \scE_{s,\frac{\ell-1}{2}}^{\diamond} \to \scF_{s-1,\frac{\ell-1}{2}}^{\diamond}\] for $s>0,$ which corresponds to 
\[L_{\sigma}:\cC_{\sfrac{\ell}{2}}^{\diamond}\to \cS^{\diamond} \text{ and  }L_{\tau}: \cC_{\sfrac{\ell}{2}}^{\diamond} \to \cS^{\diamond}.\]
By the same argument as above, we see that these horizontal length-$1$ maps between  $\scE^{\diamond}_{s,\frac{\ell-1}{2}}$ and $\scF^{\diamond}_{s,\frac{\ell-1}{2}}$ for $s>0$ will send the top generators to the top generators, as desired.

For the vertical length-$1$ maps, note that in the surgery module $\cX_n(b_1,\dots, b_m)^{\hat{\cK}}$, when $b_1$ is positive, we have 
\[\delta_1(\ys_0) = \ys' \otimes T^{\tau(K)}\sigma.\]   Again, by the restriction on the $s$-coordinate and the description of maps $f^{\pm K}$ in Figure \ref{fig:f-pm K}, we have
\begin{align*}
	&\Phi^{K}(\ys_0| \ve{e}^{\free}_{s-\tau(K),\frac{\ell-1}{2}} ) = \ys' | \ve{j}^{\free}_{s-\tau(K),\frac{\ell}{2}}\otimes 1 &\text{and } \Phi^{-K}(\ys_0| \ve{e}^{\free}_{s-\tau(K),\frac{\ell-1}{2}} )=0 \,\,&\text{ for }s>\tau(K),\\
	&\Phi^{K}(\ys_0| \ve{f}^{\free}_{s-\tau(K),\frac{\ell-1}{2}} ) = \ys'| \ve{m}^{\free}_{s-\tau(K),\frac{\ell-1}{2}}\otimes 1
	&\text{and }	\Phi^{-K}(\ys_0| \ve{f}^{\free}_{s-\tau(K),\frac{\ell-1}{2}} )=0 \,\, &\text{ for }s>\tau(K)-1,
\end{align*}
and these non-trivial arrows are already included in $\mathfrak{C}(P,K,n)_{\bF[Z]}$.

 The complex $\bX^{\diamond}(P,K,n)$ has no length 2 maps by Lemma~\ref{lem:no-diagonal-differential}, so the proof is complete.\end{proof}

Combining Lemmas~\ref{lem:quotient complex of X(P,K,n)} and ~\ref{lem:subcomplex of X(P(K,n))}, we conclude that $\mathfrak{C}(P,K,n)_{\bF[Z]}$ is a direct summand of $\bX^{\diamond}(P,K,n)_{\bF[Z]}$ as desired.

\subsection{Formulas for $\tau(P(K,n))$ when $\veps(K)=1$}

In this section, we compute $\tau(P(K,n))$ when $P$ is an L-space satellite operator and $\veps(K)=1$.

\begin{thm}
		Suppose $P$ is an L-space satellite operator, $K$ is a companion knot with $\veps(K)=1$. Let $\ell\ge0$ be the linking number $\ell =\lk(\mu,P)$ of the two components of $L_P$ (which we can assume is nonnegative by choosing the orientation of $P$ appropriately). Then, we have the following formula for $\tau(P(K,n))$.
	\[\tau(P(K,n)) =
	\begin{cases}
		 R_{\sfrac{\ell}{2}}-\frac{\ell}{2} +\frac{(\ell-1)\ell}{2}n + \ell\tau(K), & \text{if $ n<2\tau(K),$}\\
		g_3(P)+\frac{(\ell-1)\ell}{2}n + \ell\tau(K), & \text{if $n\ge 2\tau(K)$. }
	\end{cases}       \]

	In the above, $g_3(P)$ is the Seifert genus of the pattern $P$ viewed as a knot in $S^3$, and $ R_{\sfrac{\ell}{2}}$ is the second Alexander grading in $\bH(L_P)$ of the top generator $\xs^{\sfrac{\ell}{2}}_0$ in the staircase $\cC_{\frac{\ell}{2}}$. 
	\label{them:tau for epsilon=1}
\end{thm}

\begin{proof}
	We divide into two cases depending on whether $n<2\tau(K)$ or not.

	When $n<2\tau(K)$, we can truncate $\mathfrak{C}(P,K,n)_{\bF[Z]}$ as drawn in Figure \ref{fig:direct summand n less than 2tau}, such that it is homotopy equivalent to the following complex:
	\[
 \begin{tikzcd}[row sep=10mm]
 	\phantom{F}&\phantom{\cdots}& & & \ys_0|\xs'_0\arrow[d,"\bI"] \\
 	\ys'|\xs'_0 & \ys'|\xs^{\sfrac{\ell}{2}}_0\arrow[l,swap,"L_{\tau}"] \arrow[r,"L_{\sigma}"]& \cdots& \ys'|\xs^{\sfrac{\ell}{2}}_0\arrow[l,swap,"L_{\tau}"] \arrow[r,"L_{\sigma}"]&  \ys'|\xs'_0 \\
 	\ys_m|\xs'_0\arrow[u,"\bI"] & & & &
 \end{tikzcd},\] where the top right element $ \ys_0|\xs'_0=\ys_0| \ve{f}^{\free}_{-\frac{1}{2},\frac{\ell-1}{2}} $ lies in $F^{\free}_{\tau(K)-\frac{1}{2},\frac{l-1}{2}}$, and the bottom left element $\ys_m|\xs'_0 =\ys_m| \ve{f}^{\free}_{-\frac{1}{2},\frac{\ell+1}{2}}  $ lies in $F^{\free}_{-\tau(K)-\frac{1}{2},\frac{l+1}{2}}$.
	Recall from Lemma~\ref{lem:quotient bimodule in Lp} that the horizontal maps $L_{\sigma}$ and $L_{\tau}$ are 
	\[L_{\sigma}(\ys'|\xs^{\sfrac{\ell}{2}}_0) = \ys'|\xs'_0\otimes Z^{ R_{\sfrac{\ell}{2}}-\frac{\ell}{2}-g_3(P)},\,\,\,L_{\tau}(\ys'|\xs^{\sfrac{\ell}{2}}_0)= \ys'|\xs'_0\otimes Z^{ R_{\sfrac{\ell}{2}}+\frac{\ell}{2}-g_3(P)} \] respectively, so the homology of $\mathfrak{C}(P,K,n)^{\bF[Z]}$ contains an infinite $Z$-tower, which is generated by 
	\[
	\begin{split}
		 &\ys_0| \ve{f}^{\free}_{-\frac{1}{2},\frac{\ell-1}{2}} \otimes Z^{ R_{\sfrac{\ell}{2}}-\frac{\ell}{2}-g_3(P)}\\
		 +& \sum_{j=0}^{2\tau(K)-1-n}\ys'|\ve{j}^{\free}_{\tau(K)-\frac{1}{2}-j,\frac{\ell-1}{2}}\otimes Z^{j\ell}\\
		 +&\ys_m|\ve{f}^{\free}_{-\frac{1}{2},\frac{\ell+1}{2}}\otimes Z^{ R_{\sfrac{\ell}{2}}+\frac{\ell}{2}-g_3(P) + \ell(2\tau(K)-1-n)}.
	\end{split}
	\]

	Recall we have assumed that $\ell \ge0$ by picking the appropriate orientation on the pattern $P$. The $\tau$ invariant of $P(K,n)$ is given by the Alexander grading of the the generator of the $Z$-tower, which can be computed from the grading shift formula Equation~$(9.5)$ in \cite{CZZ}*{Section~9.5}.  The Alexander grading of
	\[\ys_0|\ve{f}^{\free}_{-\frac{1}{2},\frac{\ell-1}{2}} \otimes Z^{ R_{\sfrac{\ell}{2}}-\frac{\ell}{2}-g_3(P)}  \in F^{\free}_{\tau(K)-\frac{1}{2},\frac{\ell-1}{2}} \] is given by 
	\begin{align*}
	\tau(P(K,n))=	&A\big(\ys_0|\ve{f}^{\free}_{-\frac{1}{2},\frac{\ell-1}{2}} \otimes Z^{ R_{\sfrac{\ell}{2}}-\frac{\ell}{2}-g_3(P)}\big)\\
		=& A\big(\ys_0|\ve{f}^{\free}_{-\frac{1}{2},\frac{\ell-1}{2}}\big) + R_{\sfrac{\ell}{2}}-\frac{\ell}{2}-g_3(P) \\
		=&  A^{\sigma}(\ve{f}^{\free}_{-\frac{1}{2},\frac{\ell-1}{2}})+ \left(\tau(K)-\frac{1}{2}+\frac{\ell-1}{2}n\right)\ell  + R_{\sfrac{\ell}{2}}-\frac{\ell}{2}-g_3(P)\\
		=& A(\xs'_0)+ \frac{\ell}{2} +\left(\frac{n(\ell-1)-1}{2} + \tau(K)\right)\ell+ R_{\sfrac{\ell}{2}}-\frac{\ell}{2}-g_3(P)\\
		=&g_3(P)+\frac{(\ell-1)\ell}{2}n + \ell\tau(K)+ R_{\sfrac{\ell}{2}}-\frac{\ell}{2}-g_3(P).\\
		=&  R_{\sfrac{\ell}{2}}-\frac{\ell}{2}+\frac{(\ell-1)\ell}{2}n + \ell\tau(K),
	\end{align*}
when $n<2\tau(K)$, as desired.  Here, $A^{\sigma}$ is the \textit{$\sigma$-normalized Alexander grading} defined in \cite{CZZ}*{Section~4.4}, which satisfies $A^{\sigma}(\ve{f}^{\free}_{-\frac{1}{2},\frac{\ell-1}{2}})=A(\xs_0')+\frac{\ell}{2}$ in this case.

When $n\ge 2\tau(K)$, $\mathfrak{C}(P,K,n)_{\bF[Z]}$ is shown in Figure \ref{fig:direct summand n greater than 2tau}. We can truncate the complex to see that it is homotopy equivalent to the following:
\[\begin{tikzcd}		\ys_0|\xs'_0 & \ys_0|\xs^{\sfrac{\ell}{2}}_0  \arrow[l,swap,"L_{\tau}"] \arrow[r,"L_{\sigma}"]& \cdots   & \ys_0|\xs^{\sfrac{\ell}{2}}_0  \arrow[l,swap, "L_{\tau}"] \arrow[r,"L_{\sigma}"]& \ys_0|\xs'_0
\end{tikzcd},\]
where the leftmost copy $\ys_0|\xs'_0=\ys_0|\ve{f}^{\free}_{-\frac{1}{2},\frac{\ell-1}{2}} $ lies in $F^{\free}_{\tau(K)-\frac{1}{2},\frac{\ell-1}{2}}$.
Again,the maps $L_{\sigma}$ and $L_{\tau}$ are the same as before: 
\[L_{\sigma}(\ys_0|\xs^{\sfrac{\ell}{2}}_0) = \ys_0|\xs'_0\otimes Z^{ R_{\sfrac{\ell}{2}}-\frac{\ell}{2}-g_3(P)},\,\,\,L_{\tau}(\ys_0|\xs^{\sfrac{\ell}{2}}_0)= \ys_0|\xs'_0\otimes Z^{ R_{\sfrac{\ell}{2}}+\frac{\ell}{2}-g_3(P)}. \]
This time, the generator of an infinite $Z$-tower is given by the leftmost copy $\ys_0|\xs'_0=\ys_0|\ve{f}^{\free}_{-\frac{1}{2},\frac{\ell-1}{2}}  \in F^{\free}_{\tau(K)-\frac{1}{2},\frac{l-1}{2}}$ so 
	\begin{align*}
	\tau(P(K,n))=	&A\big(\ys_0|\ve{f}^{\free}_{-\frac{1}{2},\frac{\ell-1}{2}}\big)\\
	=& g_3(P)+\frac{(\ell-1)\ell}{2}n + \ell\tau(K),
\end{align*}
when $n\ge2\tau(K)$, as desired.
\end{proof}

\subsection{Formulas for $\tau(P(K,n))$ when $\veps(K)=0$}

By \cite{DHSTmore}*{Remark~3.16,Theorem~6.1,Corollary~6.2},  when $\veps(K)=0$, we have 
\[
\cCFK(K)^{\hat{R}}\simeq \langle \ys_0 \rangle\oplus A,
\]
where $\langle \ys_0 \rangle$ is the trivial type-$D$ module over $\hat{R}$ with a single generator $\ys_0$. Therefore, $P(K,n)$ is locally equivalent to $P(U,n)$, where $U$ is the unknot. Hence, we have
\[\tau(P(K,n)) = \tau(P(U,n)).\]
Thus, when $\veps(K)=0$, it suffices to consider the case when $K$ is the unknot, and compute $\tau(P(U,n))$. 

The arguments in Lemma \ref{lem:quotient complex of X(P,K,n)} and Lemma \ref{lem:subcomplex of X(P(K,n))} do not apply in this case because $\delta^1(\ys_0)$ in $\cX_n(U)^{\hat{\cK}}$ contains both $\sigma$-weighted and $\tau$-weighted arrows.  However, the complex $\bX(P,U,n)^{\bF[W,Z]}$ is not especially complicated, and we can compute $\tau(P(U,n))$ directly.

\begin{lem}[\cite{CZZ}*{Section~11.1}]
\label{lem:satellites-unknot}
 Let $U$ be the unknot, and $P$ be an L-space satellite operator. Then, $\cCFK(P(U,n))$ is homotopy equivalent to one of the following, depending on the sign of $n$.
\begin{enumerate}
	\item When $n=0$, $P(U,0) = P$, viewed as a knot in $S^3$, so $\cCFK(P(U,0)) = \cCFK(P).$
	\item When $n>0$, $\cCFK(P(U,n))$ is homotopy equivalent to the following complex:
				\[\begin{tikzcd}[labels={description}, column sep={.7cm,between origins}]
					 \, 
					 & \cC_{-N+1}
					 	\arrow[dl,"L_{\tau}"]
					 	\arrow[dr,"L_{\sigma}"]
					 	\ar[rrrr,start anchor=north west,end anchor=north east, decorate,decoration={calligraphic brace,amplitude=7pt}, no head,"n"{yshift=11pt},swap]  
					 &
					 \,&
					  \cdots \arrow[dl] \arrow[dr]
					  &\,&
					   \cC_{-N+1}
					   	\arrow[dl,"L_{\tau}"]
					   	\arrow[dr,"L_{\sigma}"] &\,& 
					   	\cC_{-N+2}
					   	\ar[dl, "L_\tau"]
					   	\ar[dr, "L_\sigma"]
					   	 &\,& 
					   	 \cdots
					   	 \ar[dl]
					   	 \ar[dr]
					   	 &\,&
					   	 \cC_{N-2}
					   	 \ar[dl, "L_\tau"]
					   	 \ar[dr, "L_\sigma"]
					   	 &\,&
					   	\cC_{N-1}\ar[rrrr,start anchor=north west,end anchor=north east, decorate,decoration={calligraphic brace,amplitude=7pt}, no head,"n"{yshift=11pt},swap]  
					   	 \ar[dl, "L_\tau"] \ar[dr, "L_\sigma"]&& \cdots \ar[dl] \ar[dr] && 
					   	\cC_{N-1} \ar[dl, "L_\tau"] \ar[dr, "L_\sigma"]&\,		\\
				\cS&& \cS&\cdots& \cS&& \cS && \cS & \cdots& \cS&& \cS && \cS &\cdots& \cS&& \cS
				\end{tikzcd}\]
	where $\cC_t$ and $\cS$ are the staircases appearing in the complex ${}_{\cK} \cX(L_P)^{\bF[W,Z]}$. There are $n$ copies of each $\cC_t$ for $-N<t<N$, where $N=N_{L_P}$ is the width of the pattern $P$.	
	\item When $n<0$,  $\cCFK(P(U,n))$ is homotopy equivalent to the following complex:
		\[\begin{tikzcd}[labels={description}, column sep={.8cm,between origins}]
			  C_{-N+1}
			 	\arrow[dr,"L_{\tau}"]
				\ar[rrrr,start anchor=north west,end anchor=north east, decorate,decoration={calligraphic brace,amplitude=7pt}, no head,"|n|"{yshift=13pt},swap]   
			 &&
			  \cdots \arrow[dl] \arrow[dr]
			  &&
			   \cC_{-N+1}
			   	\arrow[dl,"L_{\sigma}"]
			   	\arrow[dr,"L_{\tau}"] && 
			   	\cC_{-N+2}
			   	\ar[dl, "L_\sigma"]
			   	\ar[dr, "L_\tau"]
			   	 && \cdots \ar[dl] \ar[dr] &&
			   	\cC_{N-1}\ar[rrrr,start anchor=north west,end anchor=north east, decorate,decoration={calligraphic brace,amplitude=7pt}, no head,"|n|"{yshift=13pt},swap]   \ar[dl, "L_\sigma"] \ar[dr, "L_\tau"]&& \cdots \ar[dl] \ar[dr] && 
			   	\cC_{N-1} \ar[dl, "L_\sigma"]		\\
		& \cS&\cdots& \cS&& \cS && \cS & \cdots& \cS&& \cS &\cdots& \cS &
		\end{tikzcd}\]
		where again we have $|n|$ copies of each $\cC_t$ for $-N<t<N$.	
\end{enumerate}
\end{lem}

For simplicity, when $K=U$ is an unknot, we write $\bX(P,U,n)^{\bF[W,Z]}$ for the complexes described in Lemma~\ref{lem:satellites-unknot}. We can, of course, consider versions $\bX^{\diamond}(P,U,n)_{\bF[Z]}$ and $\bX^{\free}(P,U,n)_{\bF[Z]}$ as before. We warn the reader that the above complexes are not merely a truncation of $\cX_n(U)^{\cK}\boxtimes{}_{\cK} \cH_-^{\cK}\boxtimes {}_{\cK} \cX(L_P)^{\bF[W,Z]}$. Instead they are only homotopy equivalent. (In fact, they are more naturally a truncation of the tensor product $\cD_{\pm 1/n}^{\cK}\boxtimes {}_{\cK} \cX(L_P)^{\bF[W,Z]}$ where $\cD_{\pm 1/n}^{\cK}$ is the $\pm 1/n$ framed solid torus module described in \cite{ZemBordered}*{Section~18.2}).

As before, we will identify a direct summand of $\bX^\diamond(P,U,n)_{\bF[Z]}$ that supports the infinite $Z$-tower in homology. This direct summand again consists of free generators of $\cC^{\diamond}_{t}$ and $\cS^{\diamond}$. When $n>0$, this works for all L-space satellite operators $P$. When $n<0$, we impose additional restrictions on $P$ for the argument to work. This leads to the extra assumption stated in Equation~\eqref{eq:extra condition when epsilon=0} in Theorem \ref{thm: tau}.

\begin{define}
	\label{def:direct summand when epsilon=0}
	\begin{enumerate}
	\item When $n>0$, we define $\mathfrak{C}'(P,U,n)_{\bF[Z]}\subset \bX^{\diamond}(P,U,n)_{\bF[Z]}$ to be the free chain complex over $\bF[Z]$ spanned by $n$-copies of $\xs_0^{\sfrac{\ell}{2}}$ (one in each copy of $\cC^{\diamond}_{\sfrac{\ell}{2}}$), and $n+1$ copies of $\xs_0'$ (one in each adjacent copy of $\cS^{\diamond}$).
	\item When $n<0$, we define 
	\[
	\mathfrak{C}''(P,U,n)_{\bF[Z]}\subset \bX^{\diamond}(P,U,n)_{\bF[Z]}
	\]
	to be the free chain complex over $\bF[Z]$ spanned by $n$ copies of $\xs_{0}^{\sfrac{\ell}{2}}$ (one in each copy of $\cC_{\sfrac{\ell}{2}}^\diamond$), $n+1$ copies of $\xs_0'$ (one in each adjacent copy of $\cS^{\diamond}$), as well as one copy of the generators $\xs_0^{\sfrac{\ell}{2}-1}$ and $\xs_0^{\sfrac{\ell}{2}+1}$ in the neighboring copies of $\cC^{\diamond}_{\sfrac{\ell}{2}-1}$ and $\cC^{\diamond}_{\sfrac{\ell}{2}+1}$. 
	\end{enumerate}
These complexes are shown in Figure~\ref{fig:direct summand when epsilon=0}.
	\end{define}

\begin{rem} In the cases when the width $N$ of $L_P$ is $\ell/2$ or $\ell/2+1$, we need to reinterpret the complexes $\frC'$ and $\frC''$. In order to avoid considering these cases separately, in these cases we replace $N$ with $N+2$ in the complexes in Lemma~\ref{lem:satellites-unknot} (noting that the results are homotopy equivalent). 
\end{rem}

	\begin{figure}[!h]
\[
			\begin{tikzcd}[labels={description},column sep={1cm,between origins}]
&&& \xs^{\sfrac{\ell}{2}}_0 
	\arrow[dl,"L_{\tau}"]
	\arrow[dr,"L_{\sigma}"]
	\ar[rrrr,start anchor=north west,end anchor=north east, decorate,decoration={calligraphic brace,amplitude=7pt}, no head,"n"{yshift=11pt},swap]  
& &
\cdots
	\arrow[dl,"L_{\tau}"]
	\arrow[dr,"L_{\sigma}"]
& &
\xs^{\sfrac{\ell}{2}}_0
	\arrow[dl,"L_{\tau}"]
	\arrow[dr,"L_{\sigma}"]
&&&\,\\
&&\xs'_0& & \xs'_0& & \xs'_0 & & \xs'_0
\\[-1.4cm]
&&&&&\rotatebox{-90}{$\subset$}&&&&&& \boxed{n> 0}
\\[-.8cm]
\cdots&\cC^{\diamond}_{\sfrac{\ell}{2}-1}
	\arrow[dr,"L_{\sigma}"] 
	\ar[dl, "L_\tau"]
	&&
\cC^{\diamond}_{\sfrac{\ell}{2}}
	\arrow[dl,"L_{\tau}"]
	\arrow[dr,"L_{\sigma}"]
	&&
\cdots
	\arrow[dl,"L_{\tau}"]
	\arrow[dr,"L_{\sigma}"]
&&
\cC^{\diamond}_{\sfrac{\ell}{2}}
	\arrow[dl,"L_{\tau}"]
	\arrow[dr,"L_{\sigma}"]
& & \cC^{\diamond}_{\sfrac{\ell}{2}+1} \arrow[dl,"L_{\tau}"] \ar[dr, "L_\sigma"]&\cdots \\
\cdots&&\cS^{\diamond}& & \cS^{\diamond}& & \cS^{\diamond} & & \cS^{\diamond}&&\cdots
			\end{tikzcd}
			\]
			\vspace{.8cm}
			\[
			\begin{tikzcd}[labels={description}, column sep={1cm,between origins}]
	&\xs^{\sfrac{\ell}{2}-1}_0
		\arrow[dr,"L_{\tau}"]
		&&
	\xs^{\sfrac{\ell}{2}}_0
		\arrow[dl,"L_{\sigma}"]
		\arrow[dr,"L_{\tau}"]
		\ar[rrrr,start anchor=north west,end anchor=north east, decorate,decoration={calligraphic brace,amplitude=7pt}, no head,"n"{yshift=11pt},swap]  
		&&
	\cdots
		\arrow[dl,"L_{\sigma}"]
	\arrow[dr,"L_{\tau}"]
	&&
	\xs^{\sfrac{\ell}{2}}_0
		\arrow[dl,"L_{\sigma}"]
		\arrow[dr,"L_{\tau}"]
	&&
	\xs^{\sfrac{\ell}{2}+1}_0
		\arrow[dl, "L_{\sigma}"]
&&
	\\
	&&\xs'_0
	&&
	\xs'_0
	&&
	\xs'_0
	&&
	\xs'_0
	&
\\[-1.4cm]
&&&&&\rotatebox{-90}{$\subset$}&&&&&& \boxed{n< 0}
\\[-.8cm]
\cdots&\cC^{\diamond}_{\sfrac{\ell}{2}-1}
	 \arrow[dl,"L_{\sigma}"]
	 \arrow[dr,"L_{\tau}"]
&&
\cC^{\diamond}_{\sfrac{\ell}{2}}
	\arrow[dl,"L_{\sigma}"]
	\arrow[dr,"L_{\tau}"]
&&\cdots
	\arrow[dl,"L_{\sigma}"]
	\arrow[dr,"L_{\tau}"]
&&
\cC^{\diamond}_{\sfrac{\ell}{2}}
	\arrow[dl,"L_{\sigma}"]
	\arrow[dr,"L_{\tau}"]
&&
\cC^{\diamond}_{\sfrac{\ell}{2}+1}
	\arrow[dl,"L_{\sigma}"]
	\arrow[dr,"L_{\tau}"]
&\cdots
\\
	\cdots& &\cS^{\diamond}& & \cS^{\diamond}& & \cS^{\diamond} & & \cS^{\diamond}& & \cdots&
			\end{tikzcd}
\]
	\caption{The top two rows illustrate $\mathfrak{C}'(P,U,n)_{\bF[Z]}\subset \cCFK(P(U,n))_{\bF[Z]}$ and when $n>0$, and the bottom two rows illustrate $\mathfrak{C}''(P,U,n)_{\bF[Z]}\subset \cCFK(P(U,n))_{\bF[Z]}$ when $n<0$.}
	\label{fig:direct summand when epsilon=0}
	\end{figure}

\begin{lem}
	\label{lem:restriction on P}
	
	Suppose that $P$ is an L-space satellite operator. Orient $P$ so the linking number $\ell$ of the corresponding two-component link $L_P = \mu \cup P$ is nonnegative. Suppose further that
	\[
	 R_{\sfrac{\ell}{2}-1}\ge g_3(P)+\frac{\ell}{2}-1.
	\]
	Then the actions of ${}_{\cK} \cX^{\diamond}(L_P)_{\bF[Z]}$ satisfy 
	\[
	L_{\sigma}(\xs_0^{\ell/2-1})=0,\quad L_{\tau}(\xs_0^{\ell/2+1})=0.
	\]
\end{lem}

\begin{proof}
	 By the computation in the proof of Lemma \ref{lem:quotient bimodule in Lp},   we have 
	 \begin{align*} \gr_{\ws}(L_{\sigma}(\xs_0^{\sfrac{\ell}{2}-1})) = -2, \quad &\gr_{\zs}(L_{\sigma}(\xs_0^{\sfrac{\ell}{2}-1})) = -2 - 2 R_{\sfrac{\ell}{2}-1} + \ell
	 	\intertext{ and} 
	 	\gr_{\ws}(\xs'_0) = 0, \quad  &\gr_{\zs}(\xs'_0)= -2g_3(P).
	 \end{align*}
	 Therefore, when $\gr_{\zs}(\xs'_0) \ge  \gr_{\zs}(L_{\sigma}(\xs_0^{\sfrac{\ell}{2}-1})),$ i.e., when  \[R_{\sfrac{\ell}{2}-1}\ge g_3(P)+\frac{\ell}{2}-1,\]  we may choose the action $L_{\sigma}(\xs_0^{\sfrac{\ell}{2}-1})$ on the full module ${}_{\cK} \cX(L_P)^{\bF[W,Z]}$ so that 
	 \[
	 L_{\sigma}(\xs_0^{\sfrac{\ell}{2}-1}) = \xs'_0 \otimes WZ^{R_{\sfrac{\ell}{2}-1}-\sfrac{\ell}{2}-g_3(P)+1} = 0  \text{ in } \cS\otimes \bF[Z].
	 \]
	 This implies that $L_{\sigma}(\xs_0^{\sfrac{\ell}{2}-1})$ will vanish in $\cX^{\diamond}(L_P)$.
	 For $L_{\tau}(\xs^{\sfrac{\ell}{2}+1}_0)$, by similar computation, we may choose the action $L_{\tau}(\xs^{\sfrac{\ell}{2}+1}_0)$ on the full bimodule ${}_{\cK} \cX(L_P)^{\bF[W,Z]}$ to vanish once we set $W=0$ if 
	 \[R_{\sfrac{\ell}{2}+1}\ge g_3(P)-\frac{\ell}{2}-1.\]
	 However, by Lemma \ref{lem:shape of top generators}, we have 
	 \[R_{\sfrac{\ell}{2}+1}\ge g_3(P)+\frac{\ell}{2}\ge g_3(P)-\frac{\ell}{2}>g_3(P)-\frac{\ell}{2}-1,\]
	 assuming $\ell\ge 0$. 
\end{proof}

\begin{lem}
	Let $U$ be the unknot, and suppose $P$ is an L-space satellite operator. Then:
	\begin{enumerate}
		\item If $n>0$, $\mathfrak{C}'(P,U,n)_{\bF[Z]}$ is a direct summand of $\bX^{\diamond}(P,U,n)_{\bF[Z]}$.
		\item If $n<0$ and additionally 
		\[
		\ell \ge 0 \quad\text{and}\quad	R_{\sfrac{\ell}{2}-1}\ge g_3(P)+\frac{\ell}{2}-1,
		\]
		then, $\mathfrak{C}''(P,U,n)_{\bF[Z]}$ is a direct summand of $\bX^{\diamond}(P,U,n)_{\bF[Z]}$.
		 Here $\ell=\lk(\mu,P)$ is the linking number of $L_P$, and $g_3(P)$ is the Seifert genus of $P$ viewed as a knot in $S^3$.
	\end{enumerate}
\label{lem:direct summand when epsilon=0}
\end{lem}

\begin{proof}
We prove the case $n>0$ first. From the description of $L_{\sigma}(\xs^{\sfrac{\ell}{2}}_0)$, $L_{\tau}(\xs^{\sfrac{\ell}{2}}_0)$ from Lemma \ref{lem:quotient bimodule in Lp} and the shape of $\bX^{\diamond}(P,U,n)$, it is easy to see that $\mathfrak{C}'(P,U,n)_{\bF[Z]}$ is a subcomplex of $\bX^\diamond(P,K,n)_{\bF[Z]}.$ 

To see that it is also a quotient complex, note that by Lemma~\ref{lem:X-tors-submodule}, there is a quotient complex $\bX^{\free}(P,U,n)$ of  $\bX^{\diamond}(P,U,n),$ consisting of the generators $\xs^t_0$ or $\xs'_0$ in each copy of $\cC_t^{\diamond}$ or $\cS^{\diamond}$. The complex $\mathfrak{C}'(P,U,n)_{\bF[Z]}$ is contained in $\bX^{\free}(P,U,n)$, so it suffices to show $\mathfrak{C}'(P,U,n)_{\bF[Z]}$ is a quotient complex of $\bX^{\free}(P,U,n)$. From the shape of the complex, the only generators in $\bX^{\free}(P,U,n)$ for which the differential could have a non-trivial component in $\frC'(P,U,n)$ are the generators $\xs_0^{\sfrac{\ell}{2}-1}$ in the rightmost copy of $\cC_{\sfrac{\ell}{2}-1}^{\diamond}$, and $\xs_0^{\sfrac{\ell}{2}+1}$ in the leftmost copy of $\cC_{\sfrac{\ell}{2}+1}^{\diamond}.$ Therefore it suffices to show the following:
\begin{enumerate}
\item The $\xs_0'$ component of $L_\sigma(\xs_0^{\sfrac{\ell}{2}-1})$ vanishes. 
\item The $\xs_0'$ component of $L_\tau(\xs_0^{\sfrac{\ell}{2}+1})$ vanishes.
\end{enumerate}
These two components vanish by Lemma~\ref{lem:quotient bimodule in Lp}, which finishes the proof when $n>0$.

When $n<0$, it is easy to see $\mathfrak{C}''(P,U,n)_{\bF[Z]}$ is a quotient complex of $\bX^{\free}(P,U,n)$ from the shape of the complex. Therefore it is also a quotient complex of $\bX^{\diamond}(P,U,n).$

To show that $\mathfrak{C}''(P,U,n)_{\bF[Z]}$ is a subcomplex of $\bX^{\diamond}(P,U,n)_{\bF[Z]},$ it suffices to show the following:
\begin{enumerate}
\item $L_\sigma(\xs_0^{\sfrac{\ell}{2}-1})=0$.
\item $L_\tau(\xs_0^{\sfrac{\ell}{2}+1})=0$. 
\end{enumerate}

This holds by Lemma~\ref{lem:restriction on P}.
\end{proof}

\begin{thm}
	 	Suppose $K$ is a knot in $S^3$ with $\veps(K)=0$ and $P$ is an L-space satellite operator. Let $\ell\ge0$ be the non-negative linking number $\ell =\lk(\mu,P)$ of the two components of $L_P$. Then, we have the following formula for $\tau(P(K,n))$:
	 	\begin{enumerate}
	 		\item When $n\ge0$, 
	 		\[\tau(P(K,n)) = g+ \frac{\ell(\ell-1)}{2}n.\]
	 		\item When $n<0$, and if $P$ satisfies \[	R_{\sfrac{\ell}{2}-1}\ge g+\frac{\ell}{2}-1,\]
	 		then, 
	 		\[\tau(P(K,n))=	\max\left\{R_{\sfrac{\ell}{2}-1}+\ell/2 ,R_{\sfrac{\ell}{2}}-\ell/2\right\} +\frac{\ell(\ell-1)}{2}n.\]
	 	\end{enumerate}
	 In the above, $g=g_3(P)$ is the Seifert genus of $P$ viewed as a knot in $S^3$, and $R_t$ is the second Alexander grading in $\bH(L_P)$ of the top generator $\xs^{t}_0$ in the staircase $\cC_{t}$, for $t = \frac{\ell}{2}-1,\frac{\ell}{2}$.
	 \label{thm:tau when epsilon=0}
\end{thm}

\begin{proof}
As mentioned at the beginning of the section, when $\veps(K)=0$, $\cCFK(P(K,n))$ is locally equivalent to $\cCFK(P(U,n))$, so it is enough to compute $\tau(P(U,n))$. 

The statement when $n=0$ is trivial, since
\[ \tau(P(U,0)) = \tau(P) = g_3(P),\]
 where the second quality holds because $P$ is an L-space knot by the assumption that $L_P$ is an L-space link.

When $n>0$, recall from Lemma \ref{lem:quotient bimodule in Lp}, that  
\[
L_{\sigma}(\xs^{\sfrac{\ell}{2}}_0 ) = \xs'_0\cdot Z^{R_{\sfrac{\ell}{2}}-\frac{\ell}{2}-g_3(P)},\qquad L_{\tau}(\xs^{\sfrac{\ell}{2}}_0 )= \xs'_0\cdot Z^{R_{\sfrac{\ell}{2}}+\frac{\ell}{2}-g_3(P)}.\]

It is then easy to see that the homology of $\mathfrak{C}'(P(U,n))$ contains an infinite $Z$-tower, which is an infinite $Z$-tower in $\cHFK_{\bF[Z]}(P(U,n))$. Assuming $\ell\ge0$, the generator of the $Z$-tower is the leftmost $\xs'_0$. 

Recall the truncation and homotopy equivalence in \cite{CZZ}*{Section~11.1} to obtain the model of $\cCFK(P(U,n))$ in Lemma~\ref{lem:satellites-unknot} from the larger complex $\cX_n(U)^{\cK}\boxtimes{}_{\cK} \cH_-^{\cK}\boxtimes {}_{\cK} \cX(L_P)^{\bF[W,Z]}$. The leftmost $\xs'_0$ in $\mathfrak{C}'(P(U,n))$ lies in the copy $\cS = F_{-\frac{1}{2},\frac{\ell-1}{2}} \subset \cX_n(U)^{\cK}\boxtimes{}_{\cK} \cH_-^{\cK}\boxtimes {}_{\cK} \cX(L_P)^{\bF[W,Z]}$ by the truncation procedure.  Therefore, by the same computation as in the case $n\ge2\tau(K)$ in Theorem~\ref{them:tau for epsilon=1} (with $\tau(K)=0$), we get 
\[
\tau(P(U,n)) = g_3(P) + \frac{\ell(\ell-1)}{2}n.
\]

When $n<0$, again from Lemma~\ref{lem:quotient bimodule in Lp}, we have:
\[
L_{\sigma}(\xs^{\sfrac{\ell}{2}+1}_0 ) = \xs'_0\cdot Z^{R_{\sfrac{\ell}{2}+1}-\frac{\ell}{2}-g_3(P)},\qquad L_{\tau}(\xs^{\sfrac{\ell}{2}-1}_0 )= \xs'_0\cdot Z^{R_{\sfrac{\ell}{2}-1}+\frac{\ell}{2}-g_3(P)}.
\]
Therefore, the homology of $\mathfrak{C}''(P,U,n)$ again contains an infinite $Z$-tower of $\cHFK_{\bF[Z]}(P(U,n))$.  Assuming $\ell\ge 0$, the generator of the infinite $Z$-tower depends on the comparison between $R_{\sfrac{\ell}{2}-1}+\frac{\ell}{2}$ and $R_{\sfrac{\ell}{2}}-\frac{\ell}{2}$.

\begin{enumerate}
	\item If $R_{\sfrac{\ell}{2}-1}+\frac{\ell}{2}\ge R_{\sfrac{\ell}{2}}-\frac{\ell}{2}$, the generator of the infinite $Z$-tower is an appropriate linear combination of the generators in the top layer of $\mathfrak{C}''(P,U,n)$, which has the same Alexander grading as the leftmost generator $\xs^{\sfrac{\ell}{2}-1}_0$. This copy of $\xs^{\sfrac{\ell}{2}-1}_0$ lies in the copy of $\cC_{\sfrac{\ell}{2}-1} = J_{n+\sfrac{1}{2},\sfrac{\ell}{2}-\sfrac{3}{2}}\subset \cX_n(U)^{\cK}\boxtimes{}_{\cK} \cH_-^{\cK}\boxtimes {}_{\cK} \cX(L_P)^{\bF[W,Z]}$ by the truncation procedure in \cite{CZZ}*{Section~11.1}. By the grading shift formula in \cite{CZZ}*{Section~9.5},  we have 
	\[
	\tau(P(U,n)) = R_{\sfrac{\ell}{2}-1} + \left(n+\frac{1}{2}+\frac{\ell-3}{2}n\right)\ell =  R_{\sfrac{\ell}{2}-1}+ \frac{\ell}{2}+ \frac{(\ell-1)\ell}{2}n.
	\]
	\item If $ R_{\sfrac{\ell}{2}-1}+\frac{\ell}{2}<  R_{\sfrac{\ell}{2}}-\frac{\ell}{2}$, the generator of the infinite $Z$-tower has the same Alexander grading as $\xs^{\sfrac{\ell}{2}-1}_0\cdot Z^{ R_{\sfrac{\ell}{2}}- R_{\sfrac{\ell}{2}-1}-\ell}$, where $\xs^{\sfrac{\ell}{2}-1}$ is again the leftmost generator lying in the copy $\cC_{\sfrac{\ell}{2}-1} = J_{n+\sfrac{1}{2},\sfrac{\ell}{2}-\sfrac{3}{2}}\subset \cX_n(U)^{\cK}\boxtimes{}_{\cK} \cH_-^{\cK}\boxtimes {}_{\cK} \cX(L_P)^{\bF[W,Z]}$. Therefore,
	\[\tau(P(U,n)) = R_{\sfrac{\ell}{2}-1}+ \frac{\ell}{2}+ \frac{(\ell-1)\ell}{2}n+R_{\sfrac{\ell}{2}}-R_{\sfrac{\ell}{2}-1}-\ell =  R_{\sfrac{\ell}{2}}- \frac{\ell}{2}+\frac{(\ell-1)\ell}{2}n.\]
\end{enumerate}
\end{proof}

\begin{rem}
 If $n<0$, and the condition 
	\[
	R_{\sfrac{\ell}{2}-1}\ge g+\frac{\ell}{2}-1,
	\] 
	does not hold, then the generator of the infinite $Z$-tower of $\bX^\diamond(P,U,n)_{\bF[Z]}$ will typically involve generators of the staircases $\cC_t^{\diamond}$ and $\cS^{\diamond}$ other than the free generators. This appears to require a more careful case-by-case analysis depending on the pattern $P$.
\end{rem}

\subsection{Formulas for $\tau(P(K,n))$ when $\veps(K)=-1$}

Now, suppose $\veps(K)=-1$. Then, by \cite{DHSTmore}, in the decomposition
\[
\cCFK_{\hat{R}}(K)\simeq C(b_1,\dots ,b_m)\oplus A,
\]
we have $b_1=-b_m<0$, so $C(b_1,\dots ,b_m)$ takes the form:
\[
\begin{tikzcd}[column sep=1.5cm]
	\ys_0 \ar[r, rightarrow, "W^{-b_1}"] 
	& \ys_1 \ar[r, leftrightarrow, "Z^{|b_2|}"] & \cdots \ar[r, leftrightarrow,"W^{|b_{m-1}|}"]&\ys_{m-1} \ar[r, leftarrow, "Z^{b_m}"]& \ys_m.
\end{tikzcd}
\]

In particular, there are arrows leaving $\ys_0$ and $\ys_m$ in $\cCFK_{\hat{R}}(K)$, which differs from the situation when $\veps(K)=1$. Under the extra assumption 
\[	 R_{\sfrac{\ell}{2}-1} \ge  g+\frac{\ell}{2}-1\] as stated in Equation (\ref{eq:extra condition when epsilon=0}) in Theorem \ref{thm: tau}, we will identify a direct summand of $\bX^{\diamond}(P,K,n)_{\bF[Z]}$ that supports the unique infinite $Z$-tower in the homology. This direct summand is a slight modification of the one defined in Section \ref{sec:direct summand when epsilon =1}, and we will use it to compute $\tau(P(K,n))$.

\begin{define}
	\label{def:direct summand when epsilon=-1}
	 Define $\mathfrak{C}'''(P,K,n)$ to be the free chain complex over $\bF[Z]$ consisting of generators in $\mathfrak{C}(P,K,n)$, as in Definition \ref{def:direct summand when epsilon=1}, together with the following four extra generators:
	 \begin{enumerate}
	 	\item $ \ys_1 | \xs'_0 =\ys_1|\ve{f}^{\free}_{-\frac{1}{2}+b_1,\frac{\ell-1}{2}} \in\langle \ys_1 \rangle \otimes \scF^{\free}_{-\frac{1}{2}+b_1,\frac{\ell-1}{2}} \subset  F^{\free}_{\tau(K)-\frac{1}{2},\frac{\ell-1}{2}};$
	 	\item $ \ys_1 | \xs^{\sfrac{\ell}{2}-1}_0=\ys_1|\ve{e}^{\free}_{\frac{1}{2}+b_1,\frac{\ell-1}{2}}  \in\langle \ys_1\rangle \otimes \scE^{\free}_{\frac{1}{2}+b_1,\frac{\ell-1}{2}} \subset  E^{\free}_{\tau(K)+\frac{1}{2},\frac{\ell-1}{2}};$
	 	\item $\ys_{m-1}|\xs'_0=\ys_{m-1}|\ve{f}^{\free}_{-\frac{1}{2}+b_m,\frac{\ell+1}{2}} \in\langle \ys_{m-1} \rangle \otimes \scF^{\free }_{-\frac{1}{2}+b_m,\frac{\ell+1}{2}} \subset F^{\free}_{-\tau(K)-\frac{1}{2},\frac{\ell+1}{2}};$
	 	\item $ \ys_{m-1}|\xs^{\sfrac{\ell}{2}+1}_0=\ys_{m-1}|\ve{e}^{\free}_{-\frac{1}{2}+b_m,\frac{\ell+1}{2}}  \in\langle \ys_{m-1} \rangle \otimes \scE^{\free }_{-\frac{1}{2}+b_m,\frac{\ell+1}{2}} \subset E^{\free}_{-\tau(K)-\frac{1}{2},\frac{\ell+1}{2}}.$
	 \end{enumerate}
	 The differentials in $\mathfrak{C}'''(P,K,n)$ are the components of the differential on $\bX^{\free}(P,K,n)$ between the generators of $\mathfrak{C}'''(P,K,n)$.
\end{define}
  See Figure \ref{fig:direct summand when epsilon=-1, n<2tau(K)} and \ref{fig:direct summand when epsilon=-1, n>2tau(K)} for a schematic of the complex $\mathfrak{C}'''(P,K,n)$, and its truncation.
\begin{figure}[!h]	
	\adjustbox{scale=0.8}{
		\begin{tikzcd}[row sep=10mm]
			& & & & \ys_1| \xs'_0 & \ys_1| \xs_0^{\sfrac{\ell}{2}-1} \arrow[l,"L_{\tau}"]& \\
			&\phantom{E}& \phantom{F}&\phantom{\cdots}& \ys_0| \xs'_0 \arrow[d,"\bI"] \arrow[u,"\bI"] &\ys_0| \xs_0^{\sfrac{\ell}{2}}\arrow[l,"L_{\tau}"]\arrow[r,"L_{\sigma}"] \arrow[d,"\bI"] \arrow[u,"L_{W}"] &\cdots\arrow[d,"\bI"]\\
			\cdots &\ys'|\xs_0^{\sfrac{\ell}{2}}\arrow[l,"L_{\tau}"]\arrow[r,"L_{\sigma}"]& \ys'|\xs'_0 & \cdots\arrow[l,"L_{\tau}"] \arrow[r,"L_{\sigma}"]& \ys'|\xs'_0 & \ys'|\xs_0^{\sfrac{\ell}{2}} \arrow[l,"L_{\tau}"] \arrow[r,"L_{\sigma}"] & \cdots\\
			\cdots \arrow[u,"\bI"] & \ys_m|\xs^{\sfrac{\ell}{2}}_0 \arrow[l,"L_{\tau}"] \arrow[r,"L_{\sigma}"] \arrow[u,"\bI"] \arrow[d,"L_{Z}"]& \ys_m|\xs'_0\arrow[u,"\bI"] \arrow[d,"\bI"] & & & & \\
			& \ys_{m-1}|\xs^{\sfrac{\ell}{2}+1} \arrow[r,"L_{\sigma}"]& \ys_{m-1}|\xs'_0& & & &
		\end{tikzcd}
	}
\adjustbox{scale=0.8}{
	\begin{tikzcd}[row sep=8mm,column sep = 8mm]
		& & & &\\[-3cm]
		& & & \ys_1|\xs'_0 & \ys_1| \xs^{\sfrac{\ell}{2}-1}_0\arrow[l,"L_{\tau}"]\\
		& & & \ys_0|\xs'_0 \arrow[u,"\bI"] \arrow[d,"\bI", end anchor={[yshift=-0.5ex]}]& \\
		\simeq& \ys'|\xs'_0
		  & \cdots \arrow[l,"L_{\tau}"] \arrow[r,"L_{\sigma}"] &\ys'|\xs'_0 &\\
		& \ys_m|\xs'_0 \arrow[u,"\bI"] \arrow[d,"\bI"]& & &\\
		\ys_{m-1}|\xs^{\sfrac{\ell}{2}+1}_0\arrow[r,"L_{\sigma}"] & \ys_{m-1}|\xs'_0 & & &
	\end{tikzcd}
}
	 \caption{The schematic drawing of the chain complex $\mathfrak{C}'''(P,K,n)$ and its truncation when $n
	 	\le 2\tau(K)$, with $2\tau(K)-n+1$ copies of $\ys'|\xs'_0$ in the middle row of the truncation.  }
	 \label{fig:direct summand when epsilon=-1, n<2tau(K)}
\end{figure}
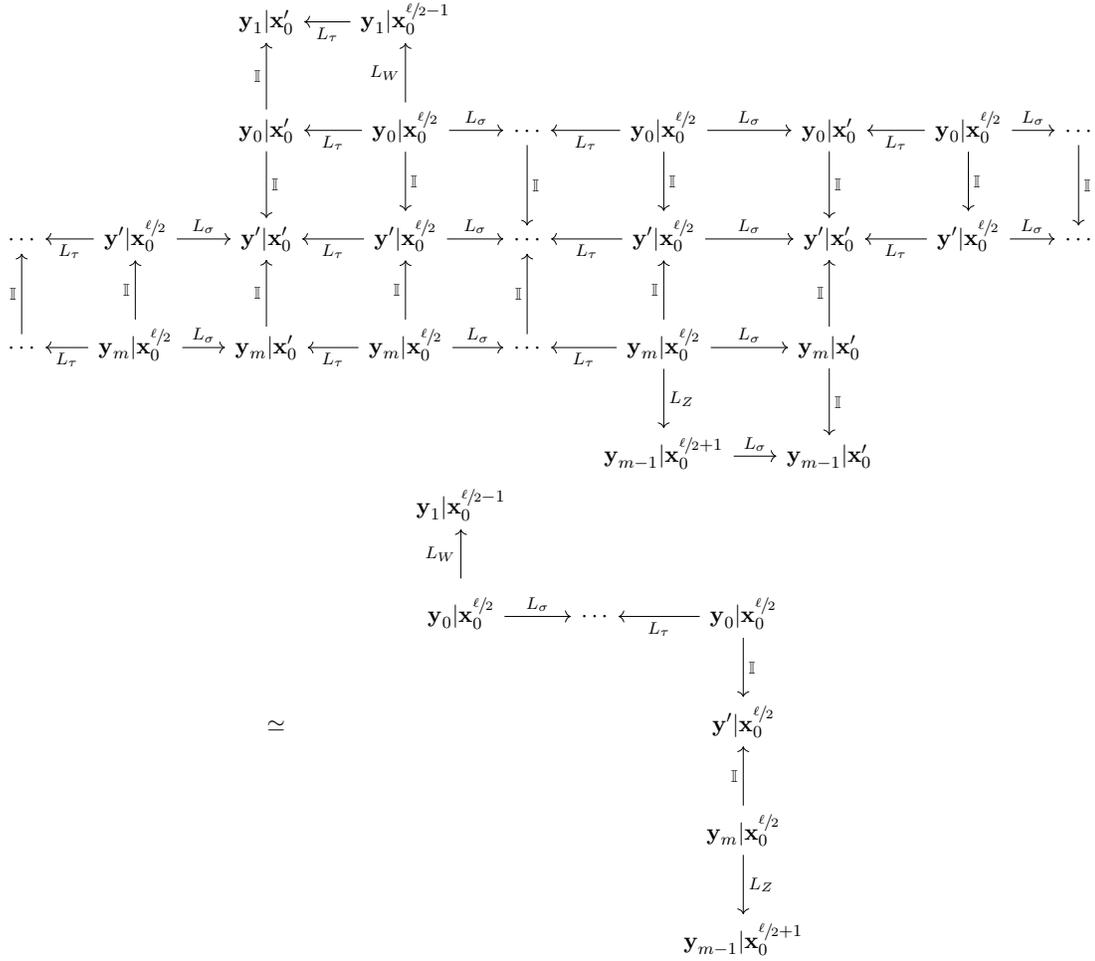
\begin{figure}[h!]
		\adjustbox{scale=0.8}{
		\begin{tikzcd}[row sep=10mm]
			& [-2mm] &[-2mm] \ys_1|\xs'_0 &[-2mm] \ys_1|\xs^{\sfrac{\ell}{2}-1}_0 \arrow[l,"L_{\tau}"] &[-2mm] &[-2mm] & [-2mm]&[-2mm] &[-2mm]\\
			& & \ys_0|\xs'_0 \arrow[u,"\bI"] \arrow[d,"\bI"]& \ys_0|\xs^{\sfrac{\ell}{2}}_0 \arrow[l,"L_{\tau}"] \arrow[u,"L_W"] \arrow[d,"\bI"] \arrow[r,"L_{\sigma}"] &\cdots \arrow[d,"\bI"] & \ys_0|\xs^{\sfrac{\ell}{2}}_0 \arrow[l,"L_{\tau}"] \arrow[r,"L_{\sigma}"] \arrow[d,"\bI"] & \ys_0|\xs'_0 \arrow[d,"\bI"] & \ys_0|\xs^{\sfrac{\ell}{2}}_0 \arrow[l,"L_{\tau}"] \arrow[r,"L_{\sigma}"] \arrow[d, "\bI"] & \cdots \arrow[d,"\bI"] \\
			\cdots & \ys'|\xs^{\sfrac{\ell}{2}}_0 \arrow[l,"L_{\tau}"] \arrow[r,"L_{\sigma}"] & \ys'|\xs'_0 &\ys'|\xs^{\sfrac{\ell}{2}}_0 \arrow[l,"L_{\tau}"] \arrow[r,"L_{\sigma}"] & \cdots & \ys'|\xs^{\sfrac{\ell}{2}}_0 \arrow[l,"L_{\tau}"] \arrow[r,"L_{\sigma}"] & \ys'|\xs'_0 & \ys'|\xs^{\sfrac{\ell}{2}}_0 \arrow[l,"L_{\tau}"] \arrow[r,"L_{\sigma}"] &\cdots\\
			\cdots \arrow[u,"\bI"]& \ys_{m}|\xs^{\sfrac{\ell}{2}}_0 \arrow[l, "L_{\tau}"] \arrow[r, "L_{\sigma}"] \arrow[u, "\bI"] & \ys_{m}|\xs'_0 \arrow[u, "\bI"] &\ys_{m}|\xs^{\sfrac{\ell}{2}}_0 \arrow[l, "L_{\tau}"] \arrow[r, "L_{\sigma}"] \arrow[u, "\bI"]&\cdots \arrow[u,"\bI"]&\ys_{m}|\xs^{\sfrac{\ell}{2}}_0 \arrow[l, "L_{\tau}"] \arrow[r, "L_{\sigma}"] \arrow[u, "\bI"] \arrow[d,"L_Z"]&  \ys_{m}|\xs'_0 \arrow[u,"\bI"] \arrow[d,"\bI"]& &\\
			& & & & & \ys_{m-1}| \xs^{\sfrac{\ell}{2}+1}_0 \arrow[r,"L_{\sigma}"] &\ys_{m-1}| \xs'_0 & & 
		\end{tikzcd}
	}
	\adjustbox{scale=0.8}{
	\begin{tikzcd}[row sep=10mm]
		& & \ys_1|\xs_{0}^{\sfrac{\ell}{2}-1}& &\\
		& & \ys_0|\xs^{\sfrac{\ell}{2}}_0 \arrow[r,"L_{\sigma}"] \arrow[u,"L_W", start anchor={[yshift=1ex]}] 
		 & \cdots & \ys_0|\xs^{\sfrac{\ell}{2}}_0 \arrow[l,"L_{\tau}"] \arrow[d,"\bI"] \\
		\simeq& & & & \ys'|\xs^{\sfrac{\ell}{2}}_0\\
		& & & & \ys_m|\xs^{\sfrac{\ell}{2}}_0\arrow[u,"\bI"]\arrow[d,"L_Z"]\\
		& & & & \ys_{m-1}|\xs^{\sfrac{\ell}{2}+1}_0
	\end{tikzcd}
}
	 \caption{The schematic drawing of the chain complex $\mathfrak{C}'''(P,K,n)$ and its truncation when $n> 2\tau(K)$, with $n-2\tau(K)$ copies of $ \ys_0|\xs^{\sfrac{\ell}{2}}_0$ in the truncation.}
	 \label{fig:direct summand when epsilon=-1, n>2tau(K)}
\end{figure}

\begin{lem}
	Suppose $K$ is a knot in $S^3$ with $\veps(K)=-1$, and $P$ is an L-space satellite operator satisfying 
	\[
		\ell \ge 0, \text{ and }\, R_{\sfrac{\ell}{2}-1} \ge g_3(P)+ \frac{\ell}{2}-1.
	\]
	Then,  $\mathfrak{C}'''(P,K,n)$ is a direct summand of $\bX^{\diamond}(P,K,n)_{\bF[Z]}$.
	\label{lem:direct summand when epsilon=-1}
\end{lem}

\begin{proof}
	We will prove $\mathfrak{C}'''(P,K,n)$ is a quotient complex of $\bX^{\free}(P,K,n)_{\bF[Z]}$ (hence a quotient complex of  $\bX^{\diamond}(P,K,n)_{\bF[Z]}$ as well), and a subcomplex of $\bX^{\diamond}(P,K,n)_{\bF[Z]}$. The proof follows the same approach as in Lemma \ref{lem:quotient complex of X(P,K,n)} and Lemma \ref{lem:subcomplex of X(P(K,n))}, so we will be brief, emphasizing only the parts that differ from the previous ones.
	
	First, we prove $\mathfrak{C}'''(P,K,n)$ is a quotient complex of $\bX^{\free}(P,K,n)_{\bF[Z]}$, that is, there are no differentials from other generators in $\bX^{\free}(P,K,n)_{\bF[Z]}$  to $\mathfrak{C}'''(P,K,n)$. The row consisting of $J^{\free}_{*,\frac{\ell-1}{2}}$ and $M_{*,\frac{\ell-1}{2}}^{\free}$ is the same as the one in $\mathfrak{C}(P,K,n)$, and the same proof as before applies. We will focus on the row consisting of $E_{s,\frac{\ell-1}{2}}^{\free}$ and $F_{s-1,\frac{\ell-1}{2}}^{\free}$ for $s>\tau(K)$, and the result for the other row follows by similar argument.
	
	As there is no length-$2$ differential in $\bX^{\free}(P,K,n)_{\bF[Z]}$, it is enough to consider length-$0$ and length-$1$ differentials. 
	
	For length-$0$ differentials, there are no differentials pointing to $\ys_0|\xs'_0$ or $\ys_0|\xs^{\sfrac{\ell}{2}}_0$, since there are no differentials in $\cCFK_{\hat{R}}(K)$ pointing to $\ys_0$ when $\veps(K)=-1$. 
	
	For $\ys_1 | \xs^{\sfrac{\ell}{2}-1}_0=\ys_1|\ve{e}^{\free}_{\frac{1}{2}+b_1,\frac{\ell-1}{2}}$ in $E_{\tau(K)+\frac{1}{2},\frac{\ell-1}{2}},$ there might be a non-trivial arrow 
	\[\begin{tikzcd}
		 \ys_1 \ar[r, leftarrow, "Z^{b_2}"]& \ys_2
	\end{tikzcd}\]
	pointing to $\ys_1$ in $\cCFK_{\hat{R}}(K)$, if $b_2$ is positive. However, as shown in Figure~\ref{fig:f-pm K}, all the $Z$-arrows pointing to $\scE^{\free}_{s,t}$ for $s<0$ are trivial, and $\ve{e}^{\free}_{\frac{1}{2}+b_1,\frac{\ell-1}{2}}$ lies in $\scE^{\free}_{\frac{1}{2}+b_1,\frac{\ell-1}{2}}$, with $\frac{1}{2}+b_1<0$ by the assumption $\veps(K)=-1$. Therefore, there is no length-$0$ differential pointing to $\ys_1 | \xs^{\sfrac{\ell}{2}-1}_0$ in $E^{\free}_{\tau(K)+\frac{1}{2},\frac{\ell-1}{2}}$ either. The same argument applies to the potential length-$0$ differential pointing to $\ys_1 | \xs'_0$ in $F^{\free}_{\tau(K)-\frac{1}{2},\frac{\ell-1}{2}}$ as well. Hence, there are no length-$0$ differentials from other generators in $\bX^{\free}(P,K,n)_{\bF[Z]}$ pointing to $\mathfrak{C}'''(P,K,n)$.

	For length-$1$ differentials, the only difference from the situation in Lemma \ref{lem:quotient complex of X(P,K,n)} is the map
	\[\Phi_{\mu}:E^{\free}_{\tau(K)-\frac{1}{2},\frac{\ell-1}{2}} \to F^{\free}_{\tau(K)-\frac{1}{2},\frac{\ell-1}{2}}.\] The component of $\Phi_{\mu}$ pointing to $\ys_0|\xs'_0$  (resp. $\ys_1|\xs'_0 $) in $F^{\free}_{\tau(K)-\frac{1}{2},\frac{\ell-1}{2}}$ is given by the component of $\ys_0|\xs_0'$ (resp. $\ys_1|\xs_0'$) in $\ys_0|L_\sigma(\xs_0^{\sfrac{\ell}{2}-1})$ (resp. $\ys_1|L_\sigma(\xs_0^{\sfrac{\ell}{2}-1})$).
	However, by Lemma \ref{lem:quotient bimodule in Lp}, the $\xs_0'$ component of $L_\sigma(\xs^{\sfrac{\ell}{2}-1})$ vanishes.  Therefore, there is no length-$1$ differential in $\bX^{\free}(P,K,n)_{\bF[Z]}$ pointing to $\mathfrak{C}'(P(K,n))$ either. This finishes the proof that $\mathfrak{C}'''(P,K,n)$ is a quotient complex of $\bX^{\diamond}(P,K,n)_{\bF[Z]}$.
	
	Next, we show that $\mathfrak{C}'''(P,K,n)$ is also a subcomplex of $\bX^{\diamond}(P,K,n)_{\bF[Z]}$, that is, there is no differential from $\mathfrak{C}'''(P,K,n)$ to other generators in $\bX^{\diamond}(P,K,n)_{\bF[Z]}$.  
	
	For the row consisting of $J^{\diamond}_{*,\frac{\ell-1}{2}}$ and $M^{\diamond}_{*,\frac{\ell-1}{2}}$, the same proof as in Lemma \ref{lem:subcomplex of X(P(K,n))} applies, since this row is identical in $\mathfrak{C}(P,K,n)$ and $\mathfrak{C}'''(P,K,n).$ 
	
	We will prove the statement for the row consisting of $E_{s,\frac{\ell-1}{2}}^{\diamond}$ and $F_{s-1,\frac{\ell-1}{2}}^{\diamond}$ with $s>\tau(K)$, and the result on the row consisting of $E_{s,\frac{\ell+1}{2}}^{\diamond}$ and $F_{s,\frac{\ell+1}{2}}^{\diamond}$ for $s<-\tau(K)$ follows by a similar argument. As before, we consider length-$0$, $1$, and $2$ differentials separately.
	
	For the length-$0$ differentials, consider possible differentials leaving $ \ys_1|\xs^{\sfrac{\ell}{2}-1}_0 =\ys_1|\ve{e}^{\free}_{\frac{1}{2}+b_1,\frac{\ell-1}{2}}$  in $E^{\diamond}_{\tau(K)+\frac{1}{2},\frac{\ell-1}{2}}$ first. By the description of $\scE^{\diamond}_{*,t}$ in Figure \ref{fig:f-pm K}, the only case in which such a length-$0$ differential could exist is when $b_1 = -1$ and $b_2<0$ in the standard complex summand $C(b_1,\dots ,b_m)$ of $\cCFK_{\hat{R}}(K)$. However, for knot $K$ in $S^3$, such a standard complex is not realizable, see the last paragraph of \cite{DHSTmore}*{Section~11}. Other length-$0$ differentials vanish due to the constraint on the $s$-coordinate and the structure of $\scE^{\diamond}_{*,t}$ and $\scF^{\diamond}_{*,t}$ depicted in Figure \ref{fig:f-pm K}.

	For the length-$1$ differentials, consider the vertical length-$1$ differentials $\Phi^{\pm K} $ first. Since $m\ge 2$ in the standard complex summand $C(b_1,\dots ,b_m)$ of $\cCFK_{\hat{R}}(K)$ when $ \veps(K)\neq 0,$ there are no $\sigma $-or $\tau$-weighted terms in $\delta^1_1(\ys_1)$ in $\cX_n(b_1,\dots, b_m)^{\hat{\cK}}$. See Theorem \ref{thm:standard-cx-over-K}. Then, by the description of $f^{\pm K}$ in Figure \ref{fig:f-pm K}, we have
	\[\Phi^{\pm K}(\ys_1|\xs^{\sfrac{\ell}{2}-1}_0) = \Phi^{\pm K}(\ys_1|\xs'_0)=0.\]
	The vanishing of the remaining vertical length-$1$ differentials (i.e., $\Phi^{\pm K}$) follows from the same proof of the corresponding statement in Lemma~\ref{lem:subcomplex of X(P(K,n))}.
	
	Next, we consider horizontal length-$1$ differentials $\Phi^{\pm\mu}$. By Figure \ref{fig:maps-f-pm-mu},  $\Phi^{\mu}$ acts by $\bI\boxtimes L_{\sigma}$, and $\Phi^{-\mu}$ acts by $\bI\boxtimes L_{\tau}$.
We proceed in the same way as in the proof of corresponding part in Lemma~\ref{lem:subcomplex of X(P(K,n))},  except that we also need to show the map $\bI\boxtimes L_{\sigma} $ vanishes on $\ys_1|\xs^{\sfrac{\ell}{2}-1}_{0}\in E_{\tau(K)+\frac{1}{2},\frac{\ell-1}{2}}^{\diamond}$, which is
	\[
	L_{\sigma}\left(\xs^{\sfrac{\ell}{2}-1}_{0}\right)=0\in  \cS^{\diamond}.
	 \] By the assumptions
	\[ \ell \ge 0,  \text{ and } R_{\sfrac{\ell}{2}-1}\ge g_3(P)+\frac{\ell}{2}-1,\]
it follows from Lemma~\ref{lem:restriction on P} that the actions of ${}_{\cK} \cX^{\diamond}(L_P)_{\bF[Z]}$ satisfy 
\[
L_{\sigma}(\xs_0^{\ell/2-1})=0
\] as desired.

 For the row consisting of $E_{s,\frac{\ell+1}{2}}^{\diamond}$ and $F_{s-1,\frac{\ell+1}{2}}^{\diamond}$ with $s<-\tau(K)$, we need $\bI\boxtimes L_{\tau} $ to vanish on $\ys_{m-1}|\xs^{\sfrac{\ell}{2}+1}_{0}\in E_{-\tau(K)-\frac{1}{2},\frac{\ell+1}{2}}^{\diamond}$, which is
 \[L_{\tau}(\xs_0^{\sfrac{\ell}{2}+1})=0\in \cS^{\diamond}.\]
 Again, this follows from the assumptions and Lemma \ref{lem:restriction on P}.

	There are no length 2 differentials by Lemma~\ref{lem:no-diagonal-differential}, so the proof is complete.
	\end{proof}

It is straightforward to see that the homology of $\mathfrak{C}'''(P,K,n)$ contains an infinite $Z$-tower. By the same technique as in Theorem \ref{them:tau for epsilon=1}, we compute $\tau(P(K,n))$ starting from the truncated $\mathfrak{C}'''(P,K,n)$, and  summarize the result in the following theorem.
\begin{thm}
	Suppose $L_P$ is a $2$-component $L$-space link, $K$ is a companion knot with $\veps(K)=-1$. Let $\ell\ge0$ be the non-negative linking number $\ell =\lk(\mu,P)$ of the two components of $L_P$. Suppose the following condition holds for the $H$-function of $L_P$: 
	\begin{equation}
		 R_{\sfrac{\ell}{2}-1} \ge g+\frac{\ell}{2}-1
		\label{eq:extra assumption}
	\end{equation} Then, we have the following formula for $\tau(P(K,n))$:
	\[\tau(P(K,n)) =
	\begin{cases}
		\max\{ R_{\sfrac{\ell}{2}-1} + \frac{\ell}{2}, R_{\sfrac{\ell}{2}} - \frac{\ell}{2}\}+\frac{(\ell-1)\ell}{2}n+\ell\tau(K), & \text{if $ n<2\tau(K),$}\\
		\max\{ R_{\sfrac{\ell}{2}-1} + \frac{\ell}{2}, R_{\sfrac{\ell}{2}+1} - \frac{\ell}{2}\} +\frac{(\ell-1)\ell}{2}n+\ell\tau(K), & \text{if $n=2\tau(K)$. }\\
		\min\{ R_{\sfrac{\ell}{2}-1} + \frac{\ell}{2}, R_{\sfrac{\ell}{2}+1} + \frac{\ell}{2}\} +\frac{(\ell-1)\ell}{2}n+\ell\tau(K), & \text{if $n=2\tau(K)+1$. }\\
		\min\{ R_{\sfrac{\ell}{2}-1} +\frac{\ell}{2},g+\ell\}+\frac{(\ell-1)\ell}{2}n+\ell\tau(K), & \text{if $n>2\tau(K)+1$. }
	\end{cases}       \]	
	In the above, $g=g_3(P)$ is the genus of the pattern $P$ viewed as a knot in $S^3$.
	\label{them:formula of tau when epsilon K =-1, simple case}
\end{thm}

\begin{proof}
	As in the case when $\veps(K)=1$, the truncation of $\mathfrak{C}'''(P,K,n)_{\bF[Z]}$ takes slightly different form depending on the relative values of $n$ and $2\tau(K)$, as illustrated in Figures~\ref{fig:direct summand when epsilon=-1, n<2tau(K)} and \ref{fig:direct summand when epsilon=-1, n>2tau(K)}.

	The answer divides into the following four cases:

	\begin{enumerate}
		\item When $n<2\tau(K)$, the truncated $\mathfrak{C}'''(P,K,n)$ is homotopy equivalent to the following complex (canceling two pairs of identity maps in Figure \ref{fig:direct summand when epsilon=-1, n<2tau(K)}),
\[
			 \begin{tikzcd}[column sep={1.2cm,between origins}]
			 	\ys_{m-1}|\xs_0^{\sfrac{\ell}{2}+1} \arrow[dr, swap,"L_{\sigma}"] & & \ys'|\xs_0^{\sfrac{\ell}{2}}  \arrow[dl,"L_{\tau}"] \arrow[dr, swap,"L_{\sigma}"] 	\ar[rrrr,start anchor=north west,end anchor=north east, decorate,decoration={calligraphic brace,amplitude=7pt}, no head,"2\tau(K)-n"{yshift=19pt},swap]  & & \cdots \arrow[dl,"L_{\tau}"] \arrow[dr, swap,"L_{\sigma}"]& & \ys'|\xs_0^{\sfrac{\ell}{2}} \arrow[dl,"L_{\tau}"] \arrow[dr, swap,"L_{\sigma}"]& & \ys_1|\xs_0^{\sfrac{\ell}{2}-1}  \arrow[dl,"L_{\tau}"] \\
			 	&\ys'|\xs'_0 & & \ys'|\xs'_0 &\cdots &\ys'|\xs'_0 & & \ys'|\xs'_0 &
			 \end{tikzcd}
		\]
	where \[
	L_{\sigma}(\xs^{t}_0) =\xs'_0 \cdot Z^{R_{t}-\frac{\ell}{2}-g} \quad \text{ for }t\in \left\{\tfrac{\ell}{2}, \tfrac{\ell}{2}+1\right\},
	\]
	and 
	 \[L_{\tau}(\xs^{t}_0) =\xs'_0 \cdot Z^{R_{t}+\frac{\ell}{2}-g} \quad \text{ for }t\in \left\{\tfrac{\ell}{2}-1,\tfrac{\ell}{2}\right\}.\]
	There are $2\tau(K)-n$ copies of  $\ys'|\xs_0^{\sfrac{\ell}{2}}$ in the top layer of the zigzag. The homology contains an infinite $Z$-tower, with the generator given by an appropriate linear combination of the top layer generators.  Assuming $\ell\ge0$,  the generator of the tower has the same Alexander grading as $\ys_1|\xs^{\sfrac{\ell}{2}-1}_0$ if $ R_{\sfrac{\ell}{2}-1}+\frac{\ell}{2}\ge  R_{\sfrac{\ell}{2}}-\frac{\ell}{2},$ or $\ys_1|\xs^{\sfrac{\ell}{2}-1}_0 \cdot Z^{ R_{\sfrac{\ell}{2}}- R_{\sfrac{\ell}{2}-1}-\ell}$ otherwise. This generator $\ys_1 | \xs^{\sfrac{\ell}{2}-1}_0 =\ys_1|\ve{e}^{\free}_{\frac{1}{2}+b_1,\frac{\ell-1}{2}}$ lies in $   E^{\diamond}_{\tau(K)+\frac{1}{2},\frac{\ell-1}{2}}$. By the grading shift formula Equation~$(9.5)$ in \cite{CZZ}*{Section~9.5}, we have 
	\begin{equation}
		 A\left(\ys_1|\ve{e}^{\free}_{\frac{1}{2}+b_1,\frac{\ell-1}{2}}\right)= R_{\sfrac{\ell}{2}-1} + \frac{\ell}{2}+\frac{(\ell-1)\ell}{2}n+\ell\tau(K).
		 \label{eq:grading of the end of zigzag}
	\end{equation}

	Therefore, in this case,
	\[\tau(P(K,n))=\max\{ R_{\sfrac{\ell}{2}-1} + \tfrac{\ell}{2}, R_{\sfrac{\ell}{2}} - \tfrac{\ell}{2}\}+\frac{(\ell-1)\ell}{2}n+\ell\tau(K)\]
		
	\item If $n=2\tau(K)$, then the truncated $\mathfrak{C}'''(P(K,n))$ is homotopy equivalent to 
	\[ 
			\begin{tikzcd}[column sep={1.2cm,between origins}]
				\ys_{m-1}|\xs_{0}^{\sfrac{\ell}{2}+1} \arrow[dr,swap, "L_{\sigma}"]& & 	\ys_1|\xs_{0}^{\sfrac{\ell}{2}-1}\arrow[dl,"L_{\tau}"]\\
				&\ys'|\xs'_0 &
			\end{tikzcd},
\]	
	with $\ys_1|\xs^{\sfrac{\ell}{2}-1}_0=\ys_1|\ve{e}^{\free}_{\frac{1}{2}+b_1,\frac{\ell-1}{2}} \in E^{\diamond}_{\tau(K)+\frac{1}{2},\frac{\ell-1}{2}}$. In a similar manner to the previous case, we compute
		\[
		\tau(P(K,n)) = \max\left\{ R_{\sfrac{\ell}{2}-1} + \tfrac{\ell}{2}, R_{\sfrac{\ell}{2}+1} - \tfrac{\ell}{2}\right\} +\frac{(\ell-1)\ell}{2}n+\ell\tau(K).
		\]
		\item If $n=2\tau(K)+1$, then after truncating $\mathfrak{C}'''(P(K,n))^{\bF[Z]}$ as in Figure~\ref{fig:direct summand when epsilon=-1, n>2tau(K)}, we see that it is homotopy equivalent to
\[
\begin{tikzcd}[column sep={1.2cm,between origins}]
			 		& 	\ys_{0}|\xs_{0}^{\sfrac{\ell}{2}} \arrow[dl,swap, "L_W"]\arrow[dr,"L_Z"] &\\
			 		\ys_1|\xs_{0}^{\sfrac{\ell}{2}-1} & & 	\ys_{m-1}|\xs_{0}^{\sfrac{\ell}{2}+1}
			 	\end{tikzcd}, 
			 	\]
		where \[
		L_W(\xs^{\sfrac{\ell}{2}}_0) = \xs^{\sfrac{\ell}{2}-1}_0\cdot Z^{ R_{\sfrac{\ell}{2}}- R_{\sfrac{\ell}{2}-1}}, \quad L_Z(\xs^{\sfrac{\ell}{2}}_0) = \xs^{\sfrac{\ell}{2}+1}_0\cdot Z^{ R_{\sfrac{\ell}{2}}- R_{\sfrac{\ell}{2}+1}}. \]
		The bottom left element $\ys_1|\xs^{\sfrac{\ell}{2}-1}_0$ lies in $ E_{\tau(K)+\frac{1}{2},\frac{\ell-1}{2}}^{\diamond}$, whose Alexander grading we have computed before in Equation \eqref{eq:grading of the end of zigzag}, and the bottom right element  $\ys_{m-1}|\xs_{0}^{\sfrac{\ell}{2}+1}=\ys_{m-1}|\ve{e}^{\free}_{-\frac{1}{2}+b_m,\frac{\ell+1}{2}}$ lies in $E_{-\tau(K)-\frac{1}{2}, \frac{\ell+1}{2}}^{\diamond}$, and has Alexander grading
		\[A(\ys_{m-1}|\ve{e}^{\free}_{-\frac{1}{2}+b_m,\frac{\ell+1}{2}})= R_{\sfrac{\ell}{2}+1} +\frac{\ell^2}{2}+\tau(K)\ell^2 = R_{\sfrac{\ell}{2}+1} + \frac{\ell}{2}+\frac{(\ell-1)\ell}{2}n+\ell\tau(K).\]
      Therefore, in this case, 
		\[
		\tau(P(K,n)) = \min\left\{ R_{\sfrac{\ell}{2}-1} + \tfrac{\ell}{2}, R_{\sfrac{\ell}{2}+1} + \tfrac{\ell}{2}\right\} +\frac{(\ell-1)\ell}{2}n+\ell\tau(K).
		\]
		
		\item If $n>2\tau(K)+1$, then the truncated $\mathfrak{C}'''(P,K,n)_{\bF[Z]}$ in Figure \ref{fig:direct summand when epsilon=-1, n>2tau(K)} is homotopy equivalent to
\[
\begin{tikzcd}[column sep={1.2cm,between origins}]
	& 
	\ys_0|\xs_{0}^{\sfrac{\ell}{2}}
		\arrow[dl,swap, "L_W"]
		\arrow[dr,"L_{\sigma}"] 
			\ar[rrrrrr,start anchor=north west,end anchor=north east, decorate,decoration={calligraphic brace,amplitude=7pt}, no head,"n-2\tau(K)"{yshift=19pt},swap]  
	&&
	\ys_0|\xs_{0}^{\sfrac{\ell}{2}}
		\arrow[dl,swap,"L_{\tau}"]
		\arrow[dr, "L_{\sigma}"]
	&&
	 \ys_0|\xs_{0}^{\sfrac{\ell}{2}}
	 	\arrow[dl,swap,"L_{\tau}"]
	 	\arrow[dr, "L_{\sigma}"]
	 &&
	 \ys_0|\xs_{0}^{\sfrac{\ell}{2}}
	 	\arrow[dr, "L_Z"]
	 	\arrow[dl,swap, "L_{\tau}"]
	 	&
	\\
	\ys_1|\xs_{0}^{\sfrac{\ell}{2}-1} 
	&&
	\ys_0|\xs'_0
	&&
	\cdots 
	&&
	\ys_0|\xs'_0 &&
	\ys_{m-1}|\xs_{0}^{\sfrac{\ell}{2}+1}
\end{tikzcd},
\]
		with maps described as before.  In particular, 
		\[L_Z(\xs^{\sfrac{\ell}{2}}_0)= \xs^{\sfrac{\ell}{2}+1}_0\cdot Z^{ R_{\sfrac{\ell}{2}}- R_{\sfrac{\ell}{2}+1}}, \quad L_{\tau}(\xs^{\sfrac{\ell}{2}}_0) =\xs'_0 \cdot Z^{ R_{\sfrac{\ell}{2}}+\frac{\ell}{2}-g}, \] 
		and by Lemma \ref{lem:shape of top generators}, 
		\[ R_{\sfrac{\ell}{2}+1}\ge g-\frac{\ell}{2}.\]
		There are $n-2\tau(K)$ copies of $ \ys_0|\xs^{\sfrac{\ell}{2}}_0$ in the top layer of the zigzag. Recall we have also assumed $\ell\ge 0.$ Then, it is easy to show that the Alexander grading of $\ys_{m-1}|\xs^{\sfrac{\ell}{2}+1}_0$ is greater than or equal to each of the $\ys_0|\xs'_0$ generators in the bottom layer. 
		
		Therefore, the generator of an infinite $Z$-tower in the homology is either $\ys_1|\xs_{0}^{\sfrac{\ell}{2}-1}$ in $E^{\diamond}_{\tau(K)+\frac{1}{2},\frac{\ell-1}{2}}$, or the leftmost copy of $\ys_0|\xs'_0$, which is $\ys_0|\ve{f}^{\free}_{\frac{1}{2},\frac{\ell-1}{2}}$ in $ F^{\diamond}_{\tau(K)+\frac{1}{2},\frac{\ell-1}{2}}.$ 
		The Alexander grading of $\ys_1|\xs_{0}^{\sfrac{\ell}{2}-1}$ has been computed before in Equation \eqref{eq:grading of the end of zigzag},  and the Alexander grading of this copy of $\ys_0|\xs'_0$ is 
		\[A(\ys_0|\ve{f}^{\free}_{\frac{1}{2},\frac{\ell-1}{2}}) = g+\ell+\frac{(\ell-1)\ell}{2}n+\ell\tau(K).\]
		Therefore, in this case,
		\[\tau(P(K,n)) = \min\left\{ R_{\sfrac{\ell}{2}-1} +\tfrac{\ell}{2},g+\ell\right\}+\frac{(\ell-1)\ell}{2}n+\ell\tau(K).\]
	\end{enumerate}

\end{proof}

\begin{rem}
	\label{rem:extra assumption}
	Here are two special families of L-space satellite operator that the extra assumption in Equation (\ref{eq:extra assumption})  is satisfied:
	\begin{enumerate}
		\item If the linking number $|\ell|\le 1,$ then by Lemma \ref{lem:shape of top generators}, we have
		\[ R_{\sfrac{\ell}{2}-1} \ge g_3(P)-\frac{\ell}{2} \ge g_3(P)+\frac{\ell}{2}-1.\]
		\item If the pattern knot $P$ is unknotted when viewed as a knot in $S^3$, then $g=g_3(P)=0$. By Equation~\eqref{eq:width}, we have 
		\[
		H_{L_P}(N_{L_P}, r) = H_{L_P}(\infty,r) = H_U\left(r-\tfrac{\ell}{2}\right)=\begin{cases}
			\frac{\ell}{2}-r & \text{if $r< \frac{\ell}{2},$}\\
			0, & \text{if $r\ge \frac{\ell}{2}$,}
		\end{cases}
		\]
	while \[
	H_{L_P}\left(\tfrac{\ell}{2}-1, R_{\sfrac{\ell}{2}-1}\right) =H_{L_P}\left(\tfrac{\ell}{2}-1,\infty\right) =1.
	\]
	Since \[\frac{\ell}{2}-1 <\frac{\ell}{2} \le N_{L_P},\]
	by the monotonicity of the $H$-function in the horizontal direction, we have 
	\[ R_{\sfrac{\ell}{2}-1} \ge \frac{\ell}{2}-1 = g + \frac{\ell}{2}-1.\]
	\end{enumerate}
\end{rem}

\section{The $\veps$-invariant and non-surjectivity}

\label{sec:epsilon}
In this section, we give an inequality for the $\veps$-invariant of satellite knots under L-space satellite operations. More explicitly, we have the following.

\begin{thm}
	\label{thm:epsilon}
	Let $P$ be an L-space satellite operator, and $g_3(P)$ denote the Seifert genus of the pattern knot $P$. Orient the corresponding two-component link $L_P$ such that the linking number $\ell \ge 0$. Then, we have the following inequality of $\veps(P(K,n))$:
	\begin{enumerate}
		\item If $\veps(K)=1$, or $\veps(K)=0$ and  $n\ge0$, then 
		\[\veps(P(K,n)) \neq -1.\]
		\item Suppose further that the H-function of the two-component link $L_P$ satisfies 
		\begin{equation}
			 R_{\sfrac{\ell}{2}-1} \ge g_3(P) + \frac{\ell}{2}.
			\label{eq:condition for epsilon}
		\end{equation}
		Then, if $\veps(K)=0$ and $ n<0$, or $\veps(K)=-1$, we have
		\[\veps(P(K,n)) \neq -1.\]
	\end{enumerate}
	
	In particular, if the L-space satellite operator $P$ satisfies the condition in Equation~ \eqref{eq:condition for epsilon}, then for any $n\in \mathbb{Z}$, the map on the smooth concordance group 
	\begin{align*}
		P: \scC &\longrightarrow \scC\\
		K &\longmapsto P(K,n)
	\end{align*} is not surjective.
\end{thm}

\begin{rem}
The condition in Equation~\eqref{eq:condition for epsilon} is slightly stronger (by 1) than the condition required for our computation of $\tau(P(K,n))$ in Theorem~\ref{thm: tau}. It is not satisfied for cabling operations, but it is satisfied by all two-bridge links $L(rq-1,q)$ with $r\ge q\ge3$, where $r$ and $q$ are odd integers. This family includes the case of Whitehead double and Mazur pattern. It is also satisfied when $L_P$ has linking number 0. It seems complicated to obtain an explicit formula of $\veps(P(K,n))$.  However, recall that the $\tau$-invariant of a knot must be $0$ if the $\veps$-invariant is $0$, so if we have $\tau(P(K,n))\neq 0$, then we can conclude that $\veps(P(K,n))=1$ under the assumption of the above theorem. 
\end{rem}

To study the $\veps$-invariant, we need to consider the version of knot Floer homology over the ring $\hat{R}$.  Recall the discussion after Definition \ref{def:local map}, the $\veps$-invariant of a knot $K$ is in $\{0,1\}$ if there is a $\hat{R}$-local map from $\hat{R}$ to $\cCFK_{\hat{R}}(K)$, and it is in $\{-1,0\}$ if there is a $\hat{R}$-local map in the opposite direction. 

For our purposes, there is another description of the $\veps$ invariant which is more convenient. There is a natural map
\begin{equation}
Q\colon \cCFK_{\hat{R}}(K)\to \cCFK_{\bF[Z]}(K)
\label{eq:map-Q}
\end{equation}
which sets $W$ to $0$. Then $\veps(K)\in \{0,1\}$ if and only if the induced map on homology
\[
Q_*\colon \cHFK_{\hat{R}}(K)\to \cHFK_{\bF[Z]}(K) 
\]
has a $\bF[Z]$-tower generator of $\cHFK_{\bF[Z]}(K)$  in its image.

Additionally, one defines $\veps(K)\in \{-1,0\}$ if and only if $\veps(-K)\in \{0,1\}$. 

It is helpful to rephrase the above condition into the language of bimodules. We will view the knot Floer complex over $\hat{R}$ as a type-$D$ module over $\hat{R}$, denoted $\cCFK(K)^{\hat{R}}$. We view $\hat{R}$ as a bimodule ${}_{\hat{R}} \hat{R}_{\bF[Z]}$, where $Z$ acts on the right by inclusion of $\bF[Z]$ into $\hat{R}$. There is a bimodule morphism 
\[
q\colon {}_{\hat{R}} \hat{R}_{\bF[Z]}\to  {}_{\hat{R}} \bF[Z]{}_{\bF[Z]}
\]
gotten by setting $W$ to be equal to zero. The morphism $q$ has $q_{i|1|j}=0$ unless $i=j=0$. Then the map $Q$ can be described as
\begin{equation}
Q=\bI_{\cCFK(K)}\boxtimes q.\label{eq:quotient-map-Q}
\end{equation}

To study the $\veps$-invariant, we need to study the $DA$-bimodule ${}_{\cK} \cX(L_P)^{\hat{R}}$. For computational purposes, it is helpful to study ${}_{\cK} \cX(L_P)^{\hat{R}}\boxtimes {}_{\hat{R}} \hat{R}_{\hat{R}}$. Analogously to Lemma~\ref{lem:X-min-non-trivial-actions}, this bimodule admits a relatively simple minimal model:

\begin{lem} 
\label{lem:R-hat-modules}
\item
\begin{enumerate}
\item  The bimodule ${}_{\cK} \cX(L_P)^{\hat{R}}\boxtimes {}_{\hat{R}} \hat{R}_{\hat{R}}$ is homotopy equivalent to an $AA$-bimodule ${}_{\cK} \cX^{\ddagger}(L_P)_{\hat{R}}$ which has only $m_{1|1|0}$ and $m_{0|1|1}$ non-trivial. The bimodule ${}_{\cK} \cX(L_P)^{\hat{R}}\boxtimes {}_{\hat{R}} \hat{R}_{\bF[Z]}$ is homotopy equivalent to the same module, with scalars restricted on the right.
\item There is an isomorphism of $(\cK,\bF[Z])$-bimodules: 
\[
\cX^{\diamond}(L_P)\iso \cX^{\ddagger}(L_P)/ (\cX^{\ddagger}(L_P)\cdot W).
\]
\item With respect to the homotopy equivalences
\[
\cCFK(P(K,n))_{\bF[Z]}\simeq \bX^{\diamond}(P,K,n)_{\bF[Z]} \quad \text{and} \quad 
\cCFK(P(K,n))_{\hat{R}}\simeq \bX^{\ddagger}(P,K,n)_{\hat{R}},
\]
the map 
\[
Q\colon \cCFK_{\hat{R}}(P(K,n))\to \cCFK_{\bF[Z]}(P(K,n))
\]
from Equation~\eqref{eq:quotient-map-Q} is chain homotopic to
\[
\bI_{\cX(K)\boxtimes \cH_-}\boxtimes Q_{L_P}
\]
where $Q_{L_P}$ is the quotient map
\[
Q_{L_P}\colon\cX^{\ddagger}(L_P)\to \cX^{\ddagger}(L_P)/(\cX^{\ddagger}(L_P)\cdot W)\iso \cX^{\diamond}(L_P).
\]
\end{enumerate}
\end{lem}
\begin{proof}The proof is similar to the proof of Lemma~\ref{lem:X-min-non-trivial-actions}. Write $\cX^{\ddagger}(L_P)$ for the right $\hat{R}$-module constructed by replacing each staircase $\cC_s$ or $\cS$ with $H_*(\cC_s\otimes_R \hat{R})$.

Before proving the main claims, we describe some basic facts about staircase complexes. If $\cC$ is a staircase, write $\cC^\ddagger$ for $H_*(\cC\otimes_R \hat{R})$. 

It is straightforward to verify that if $\cC$ is a staircase, then $\cC^\ddagger$ is in fact isomorphic to $H_*(\cC)/U$.  This follows from a more general fact. Suppose $S$ is a ring, $a\in S$ and $C$ is a free chain complex over $S$ such that there is a quasi-isomorphism from $C$ to its homology $H_*(C)$. Suppose further that the action of $a$ is an injection on $C$ and $H_*(C)$. Then we claim that $C/a$ is quasi-isomorphic to $H_*(C)/a$. This follows from the following sequence of quasi-isomorphisms:
\begin{equation}
\begin{split}
C/a\simeq  &(C\xrightarrow{a} C)\\
= &C\otimes (R\xrightarrow{a} R)\\
\simeq &H_*(C)\otimes (R\xrightarrow{a} R)\\
=&( H_*(C)\xrightarrow{a} H_*(C))\\
\simeq &H_*(C)/a.
\end{split}
\label{eq:he-C/a}
\end{equation}
The first quasi-isomorphism follows from the general fact that if $f\colon A\to B$ is an injective chain map, then $\Cone(f)$ is quasi-isomorphic to $B/f(A)$.  The second quasi-isomorphism (second to third line) follows from the fact that there is a quasi-isomorphism $C\to H_*(C)$ and the fact that $R\xrightarrow{a} R$ is a free chain complex. The final quasi-isomorphism follows from the fact that $a$ acts injectively on $H_*(C)$.

We now address the claims in the statement. The proof is essentially the same as the proof of Lemma~\ref{lem:X-min-non-trivial-actions}. For each staircase $\cC_t$, we pick a strong deformation retraction of chain complexes over $\bF$  
\[
\begin{tikzcd}
\ar[loop left, "H"]\cC_t\otimes_R \hat{R} \ar[r,"\Pi",shift left] & \cC_t^\ddagger \ar[l, "I",shift left]
\end{tikzcd}
\]
such that $H$ increases the algebraic grading by 1. This can be achieved by writing $\cC_t$ as a mapping cone $\cC_t=\Cone(d\colon\cC_t^1\to \cC_t^0)$. Then $\Pi$ is projection of $\cC_t^0$ onto $\cC_t^0/\im d$. The short exact sequence
\[
0\to \cC^1_t\xrightarrow{d} \cC^0_t\xrightarrow{\Pi} \cC^\ddagger\to 0
\]
splits over $\bF$. The map $I$ is the corresponding map from $\cC^\ddagger$ to $\cC^0_t$ in this splitting, while $H$ is the map from $\cC^0_t$ to $\cC^1_t$ in this splitting. The homological perturbation lemma equips $\cX^\ddagger(L_P)$ with the structure of an $AA$-bimodule, and the same argument as in Lemma~\ref{lem:X-min-non-trivial-actions} implies that only $m_{0|1|1}$ and $m_{1|1|0}$ are non-trivial.

For the second claim, it suffices to show the analogous claim for a single staircase $\cC$. In this case, we repeatedly apply Equation~\eqref{eq:he-C/a} to see that 
\[
\cC^\diamond:= H_*(\cC/W)\iso H_*(\cC)/W\iso \left( H_*(\cC)/U \right)/W\iso \cC^\ddagger/W.
\]

To prove the final claim, one verifies that the following diagram commutes:
\[
\begin{tikzcd}{}_{\cK} \cX(L_P)^{\hat{R}} \boxtimes {}_{\hat{R}} \hat{R}_{\hat{R}} \boxtimes {}^{\hat{R}} [\iota]_{\bF[Z]} \ar[d, "\bI\boxtimes q"] \ar[r, "\Pi\boxtimes \bI"] & {}_{\cK} \cX^{\ddagger}(L_P)_{\hat{R}}\boxtimes {}^{\hat{R}}[\iota]_{\bF[Z]} \ar[d, "Q_{L_P}"]\\
{}_{\cK} \cX(L_P)^{\bF[Z]} \boxtimes {}_{\bF[Z]}\bF[Z]_{\bF[Z]} \ar[r, "\Pi\boxtimes \bI"] & {}_{\cK} \cX^{\diamond}(L_P)_{\bF[Z]}.
\end{tikzcd}
\]
We leave the verification to the reader, as it is straightforward. In the above, $\iota\colon \bF[Z]\to \hat{R}$ is the canonical inclusion, and $[\iota]$ is the induced bimodules. Here we write $\Pi$ for both the homotopy equivalence from ${}_{\cK} \cX(L_P)^{\hat{R}}\boxtimes {}_{\hat{R}} \hat{R}_{\hat{R}}$  to ${}_{\cK} \cX^{\ddagger}(L_P)_{\hat{R}}$ , as well as the homotopy equivalence of the $\diamond$-labeled modules. These are homotopy equivalences of $AA$-bimodules, and their only non-trivial term is $\Pi_{0|1|0}$. 
\end{proof}

In Figure~\ref{fig:2}, we illustrate the modules $H_*(\cC)$, $\cC^{\ddagger}$ and $\cC^\diamond$ when $\cC=\cCFK(T_{3,4})$, shown below:
\begin{equation}
\begin{tikzcd}
\xs_0 & \xs_1 \ar[l, "W",swap] \ar[r, "Z^2"] & \xs_2& \xs_3 \ar[l, "W^2",swap] \ar[r, "Z"] & \xs_4.
\end{tikzcd}
\label{eq:CFK-T34}
\end{equation}

\begin{figure}[h]
\begingroup%
  \makeatletter%
  \providecommand\color[2][]{%
    \errmessage{(Inkscape) Color is used for the text in Inkscape, but the package 'color.sty' is not loaded}%
    \renewcommand\color[2][]{}%
  }%
  \providecommand\transparent[1]{%
    \errmessage{(Inkscape) Transparency is used (non-zero) for the text in Inkscape, but the package 'transparent.sty' is not loaded}%
    \renewcommand\transparent[1]{}%
  }%
  \providecommand\rotatebox[2]{#2}%
  \newcommand*\fsize{\dimexpr\f@size pt\relax}%
  \newcommand*\lineheight[1]{\fontsize{\fsize}{#1\fsize}\selectfont}%
  \ifx\svgwidth\undefined%
    \setlength{\unitlength}{244.38129873bp}%
    \ifx\svgscale\undefined%
      \relax%
    \else%
      \setlength{\unitlength}{\unitlength * \real{\svgscale}}%
    \fi%
  \else%
    \setlength{\unitlength}{\svgwidth}%
  \fi%
  \global\let\svgwidth\undefined%
  \global\let\svgscale\undefined%
  \makeatother%
  \begin{picture}(1,0.33918147)%
    \lineheight{1}%
    \setlength\tabcolsep{0pt}%
    \put(0,0){\includegraphics[width=\unitlength,page=1]{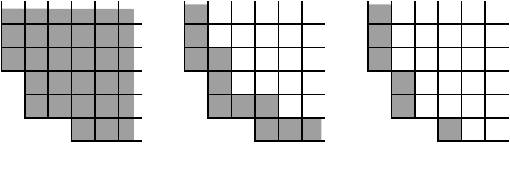}}%
    \put(0.13017386,0.00445268){\makebox(0,0)[t]{\lineheight{1.25}\smash{\begin{tabular}[t]{c}$H_*(\cC)$\end{tabular}}}}%
    \put(0.49496144,0.00702064){\makebox(0,0)[t]{\lineheight{1.25}\smash{\begin{tabular}[t]{c}$\cC^\ddagger$\end{tabular}}}}%
    \put(0.85974911,0.00702064){\makebox(0,0)[t]{\lineheight{1.25}\smash{\begin{tabular}[t]{c}$\cC^\diamond$\end{tabular}}}}%
  \end{picture}%
\endgroup%

\caption{The modules $H_*(\cC)$, $\cC^\ddagger$ and $\cC^\diamond$, when $\cC$ is the staircase complex $\cC=\cCFK(T_{3,4})$, shown above in Equation~\eqref{eq:CFK-T34}. In these diagrams, each shaded square denotes a generator over $\bF$. The action of $W$ moves left, while the action of $Z$ moves upwards.}
\label{fig:2}
\end{figure}
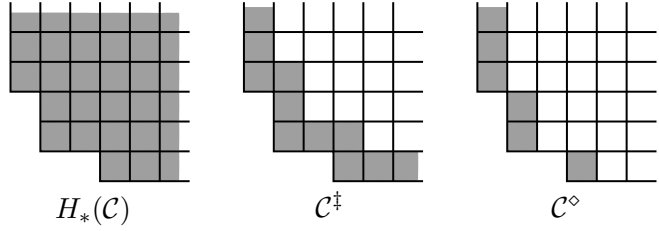

\begin{proof}[Proof of Theorem~9.1]
We will use the description of $\veps$ in terms of the map
\[
Q\colon \cCFL_{\hat{R}}(K)\to \cCFL_{\bF[Z]}(K)
\] 
from Equation~\eqref{eq:map-Q}.

By Lemma~\ref{lem:R-hat-modules}, it suffices to show that the map
\[
\bI_{\cX(K)\boxtimes \cH_-}\boxtimes Q_{L_P}\colon \cHFK_{\hat{R}}(P(K,n))\to \cHFK_{\bF[Z]}(P(K,n))
\]
has in its image (on the level of homology groups) a generator of the $\bF[Z]$-tower of $\cHFK_{\bF[Z]}(P(K,n))$. 

 In our analysis of $\tau(P(K,n))$, we identified a generator of the $\bF[Z]$-tower of $\bX(P,K,n)$. We will call this $\xs_{\tau}$. The $\hat{R}$-module generators of $\bX^{\ddagger}(P,K,n)_{\hat{R}}$ are the same as the $\bF[Z]$-module generators of $\bX^{\diamond}(P,K,n)_{\bF[Z]}$. Namely, these consist of the generators $\xs^t_{2i}$ in idempotent 0 for $i\in \Z$ and $t\in \Z+\ell/2$, as well as the generators $T^j\otimes \xs'_{2i}$ in idempotent 1. In particular, we may view $\xs_{\tau}$ as an element also in $\bX^{\ddagger}(P,K,n)$. To show that $\veps(P(K,n))\ge 0$, it suffices to show that $\xs_\tau$ is a cycle in $\bX^{\ddagger}(P,K,n)$. As with $\bX^{\diamond}(P,K,n)$ in Equation~\eqref{eq:X-P-K-n-top-cube}, the complex $\bX^{\ddagger}(P,K,n)$ can be decomposed into a hypercube
\[
\bX^{\ddagger}(P,K,n)_{\hat{R}}=\begin{tikzcd}[column sep=1.1cm, row sep=1.1cm]
 \bE^{\ddagger} \ar[r, "\Phi^{\mu}+\Phi^{-\mu}"] \ar[d, "\Phi^{K}+\Phi^{-K}",swap] & \bF^{\ddagger}\ar[d, "\Phi^{K}+\Phi^{-K}"]\\
\bJ^{\ddagger} \ar[r, swap,"\Phi^{\mu}+\Phi^{-\mu}"]& \bM^{\ddagger},
\end{tikzcd}
\]
where $\bE^{\ddagger}$, $\bF^{\ddagger}$, $\bJ^{\ddagger}$ and $\bM^{\ddagger}$ are defined via the obvious modification of $\bE^{\diamond}$, $\bJ^{\diamond}$, and so forth. Similar to Lemma~\ref{lem:no-diagonal-differential}, there is no length 2 map to the hypercube.  The maps $\Phi^{\pm K}$ are obtained by composing the $\sigma$ and $\tau$ outputs of $\delta^1$ of $\cX_n(K)^{\cK}$ into the $\delta_2^1$ of $\cH_-$. The maps $\Phi^{\pm \mu}$ are obtained by composing the $\sigma$ and $\tau$-weighted components of $\delta_1^1$ of $\cH_-$ with the map $\delta_2^1$ of ${}_{\cK} \cX^{\ddagger}(L_P)_{\hat{R}}$. Similar to the proof of Lemma~\ref{lem:X-min-non-trivial-actions}, the structure map $m_{1|1|0}(a,-)$ of ${}_{\cK} \cX^{\ddagger}(L_P)_{\hat{R}}$ is obtained by first including a module generator into the full module ${}_{\cK}\cX(L_P)^{R}\boxtimes R$, applying $\delta_2^1(a,-)$ of this module, and then projecting $\cX(L_P)\otimes R$ to $\cX^{\ddagger}(L_P)$. The relevant terms of the differential are of ${}_{\cK} \cX(L_P)^R$ are therefore those which are weighted by either $Z^i$ or $W^i$, for some $i\ge 0$.

 We can interpret our proof that $\xs_\tau$ was a cycle in Section~\ref{sec: tau computation} amounted to showing that there are no $Z^i$-weighted differentials leaving $\xs_\tau$ (under the respective hypotheses of these theorems).
 
 In this lemma, it suffices to show that there are no $W^i$ weighted differentials leaving $\xs_\tau$ which survive when projected to $\bX^{\ddagger}(P,K,n)$.

We will in fact show that there are no $W$-weighted differentials from any generators in the subspaces $\frC^{\ddagger}$, $(\frC')^{\ddagger}$, $(\frC'')^{\ddagger}$ or $(\frC''')^{\ddagger}$ to other generators of $\bX^{\ddagger}(P,K,n)$.

	We showed in Section \ref{sec: tau computation} shows that there are no differentials weighted by $Z^i$ from $\mathfrak{C}$ to other generators of $\bX^{\diamond}(P,K,n)$. 
Our work shows the same result for $\mathfrak{C}'$ when $\veps(K)=0$ and $n\ge 0$, and $\mathfrak{C}''$ when $n<0$ as long as
		 \[
		R_{\ell/2-1} \ge g_3(P)+\frac{\ell}{2}-1
		\] holds. 
	A similar statement holds for $\mathfrak{C}'''$ when $\veps(K)=-1$ and the above condition holds. 
	
	To disambiguate notation, we write  $\frC^{\ddagger},$ $(\frC')^{\ddagger}$ (and so forth) for the $\hat{R}$-span of the generators that we used to define $\frC$, $\frC'$ (and so forth).

 Therefore, to conclude $\veps(P(K,n))\neq -1$, it suffices to show there are no $W$-weighted differentials from $\mathfrak{C}^\ddagger$ (respectively  $(\mathfrak{C}')^{\ddagger},(\mathfrak{C}'')^{\ddagger},(\mathfrak{C}''')^{\ddagger}$) to other generators in $\bX^{\ddagger}(P,K,n)$. This will imply that $\xs_\tau$ is also a cycle in $\bX^{\ddagger}(P,K,n)$ so $\veps(P(K,n))\ge 0$.

	 Suppose first that  $\veps(K)=1$, and consider the complex $\mathfrak{C}^\ddagger(P,K,n)$ defined in Definition \ref{def:direct summand when epsilon=1}. 
	 
	 We start with length-$0$ maps first. For elements of the form $\ys'|\xs^{\sfrac{\ell}{2}}_{0}\in J^{\ddagger}_{s,\frac{\ell-1}{2}}$ and $\ys'|\xs'_0\in M^{\ddagger}_{s,\frac{\ell-1}{2}}$ , it is obvious, as each $J^{\ddagger}_{s,\frac{\ell-1}{2}}$ (resp. $M^{\ddagger}_{s,\frac{\ell-1}{2}}$) is a copy of $\cC^{\ddagger}_{\sfrac{\ell}{2}}$ (resp. $\cS^{\ddagger}$), and there are no $W$-arrows leaving $\xs^{\sfrac{\ell}{2}}_0$ (resp. $\xs'_0$) in $\cC^{\ddagger}_{\sfrac{\ell}{2}}$ (resp. $\cS^{\ddagger}$). For   $\ys_0| \xs^{\sfrac{\ell}{2}}_{0} \in E_{s,\frac{\ell-1}{2}}$ for $s>\tau(K),$ since $\delta^1(\ys_0)=0$ in the standard complex of the companion knot $K$, and there are no $W$-arrows leaving $\xs^{\sfrac{\ell}{2}}_0$ in $\cC^{\ddagger}_{\sfrac{\ell}{2}}$, there are no $W$-arrows leaving $\ys_0| \xs^{\sfrac{\ell}{2}}_{0}$ as well. The situation is similar for $\ys_0| \xs'_{0} \in F^{\ddagger}_{s,\frac{\ell-1}{2}}$ for $s>-1+\tau(K),$ $\ys_m|\xs^{\sfrac{\ell}{2}}_0 \in E^{\ddagger}_{s,\frac{\ell+1}{2}}$  for $s<-\tau(K)$ and $\ys_m| \xs'_{0} \in F^{\ddagger}_{s,\frac{\ell+1}{2}}$ for $s<-\tau(K).$
	 
	 For length-$1$ maps, we start with horizontal ones. In the row of $J^{\ddagger}_{s,\frac{\ell-1}{2}},M^{\ddagger}_{s,\frac{\ell-1}{2}}$, we observe that 
	 \[L_{\sigma}(\xs^{\sfrac{\ell}{2}}_0) = \xs'_0\otimes Z^{ R_{\sfrac{\ell}{2}}-g_3(P)-\sfrac{\ell}{2}},\quad L_{\tau}(\xs^{\sfrac{\ell}{2}}_0) = \xs'_0\otimes Z^{ R_{\sfrac{\ell}{2}}-g_3(P)+\sfrac{\ell}{2}},\] 
	 by the computation in Lemma \ref{lem:quotient bimodule in Lp}. Therefore, there are no $W$-arrows leaving them.	  In the rows containing $E_{s,\frac{\ell\pm 1}{2}}^\ddagger$ and $F^\ddagger_{s,\frac{\ell\pm 1}{2}}$, the argument is identical.
	
	   We now consider the maps $\Phi^K$ and $\Phi^{-K}$ (i.e. the vertical differentials). The maps $\Phi^K$ or $\Phi^{-K}$ vanish on $J_{s,\frac{\ell-1}{2}}^\ddagger,M_{s,\frac{\ell-1}{2}}^\ddagger $, so we only need to consider the rows of $E_{s,\frac{\ell\pm 1}{2}}^\ddagger$ and $F_{s,\frac{\ell\pm 1}{2}}^\ddagger$. For the row of $E_{s,\frac{\ell-1}{2}}^\ddagger,F_{s,\frac{\ell-1}{2}}^{\ddagger},$ we have already included all the downward maps $f^{K}$ in the definition of $\mathfrak{C}(P,K,n)$. For the upward maps $\Phi^{-K},$ since there is no $\tau$-term in $\delta^1(\ys_0)$, there are no upward maps leaving $\ys_0|\xs^{\sfrac{\ell}{2}}_0$ or $ \ys_0| \xs'_0$. The situation for the row of $E_{s,\frac{\ell+1}{2}}^\ddagger,F_{s,\frac{\ell+1}{2}}^\ddagger$ is similar. This finishes the discussion when $\veps(K)=1$.

	 When $\veps(K)=0$ and $n \ge 0$, consider the complex $\mathfrak{C}'(P,K,n)$ in Definition \ref{def:direct summand when epsilon=0}. There are no $W$-weighted arrows leaving $(\mathfrak{C}')^\ddagger$, because it is easily verified that there are no internal differentials of any of the $\cC_t^\ddagger$ or $\cS^{\ddagger}$ leaving any of the generators of $(\mathfrak{C}')^\ddagger$, and there are no possible $W$-weighted length differentials leaving $(\mathfrak{C}')^\ddagger$ since all $L_\tau$ or $L_\sigma$ arrows leaving a generator in this complex are already contained in the complex due to the shape of the complex, and they are all weighted by powers of $Z$. This finishes the discussion when $\veps(K)=0$ and $n > 0$.
	 
	 When $\veps(K)=0$ and $n< 0$, consider the complex $(\mathfrak{C}'')^{\ddagger}$, whose $\hat{R}$-generators are the same as the ones listed in Definition~\ref{def:direct summand when epsilon=0}. In this case, there are possibly non-trivial differentials from the copies of $\xs_0^{\sfrac{\ell}{2}-1}$ and $\xs_0^{\sfrac{\ell}{2}+1}$. These are the left-most and right-most generators in Figure~\ref{fig:direct summand when epsilon=0}. The possible differentials are
	 \[
	 L_\sigma(\xs^{\sfrac{\ell}{2}-1})\in \cS^{\ddagger}\quad \text{and} \qquad L_\tau(\xs^{\sfrac{\ell}{2}+1})\in \cS^{\ddagger}.
	 \]
	 Using computations similar to those in Lemma \ref{lem:restriction on P} and the assumption in Equation~\eqref{eq:condition for epsilon}, we can choose a model of $DA$-bimodule ${}_{\cK} \cX(L_P)^{\bF[W,Z]}$, such that two maps to be
\[\delta_2^1(\tau,\xs^{\sfrac{\ell}{2}+1}_0) = \xs'_0 \otimes UZ^{ R_{\sfrac{\ell}{2}+1}-g_3(P)+\frac{\ell}{2}},\quad \delta_2^1(\sigma,\xs^{\sfrac{\ell}{2}-1}) = \xs'_0\otimes UZ^{ R_{\sfrac{\ell}{2}-1}-g_3(P)-\frac{\ell}{2}}.\]
Observe that $\delta_2^1(\tau,\xs^{\sfrac{\ell}{2}+1})$ vanishes once we set $U=0$ by our assumption in Equation~\eqref{eq:condition for epsilon}. Also, we can see that $\delta_2^1(\sigma,\xs^{\sfrac{\ell}{2}-1})$ vanishes once we set $U=0$ by  observing that
\[ R_{\sfrac{\ell}{2}+1} \ge g_3(P)+\frac{\ell}{2}\ge g_3(P)-\frac{\ell}{2}\]
by Lemma \ref{lem:shape of top generators} and the assumption $\ell \ge 0$. So these maps vanish as well when working over the ring $\bF[W,Z]/(U)$. This finishes the case $\veps(K)=0, n<0$.

When $\veps(K)=-1$, consider the complex $(\mathfrak{C}''')^{\ddagger}$ which is generated over $\hat{R}$ by the generators appearing in Definition~\ref{def:direct summand when epsilon=-1}. Our argument is similar to the argument when $\veps(K)=1$, in the previous paragraphs, though we must pay extra attention to the four extra generators.

To see that there are no length 0 differentials leaving the generator $\xs_\tau$, we argue similarly to our proof in Lemma~\ref{lem:direct summand when epsilon=-1} that $\mathfrak{C}'''$ is a subcomplex. There are no changes in the row containing $J_{s,\frac{\ell-1}{2}}^{\ddagger}$ and $M_{s,\frac{\ell-1}{2}}^{\ddagger}$ between $\frC^\ddagger$ and $(\frC''')^\ddagger$. For the complexes $E_{s,\frac{\ell-1}{2}}^{\ddagger}$ and $F_{s,\frac{\ell-1}{2}}^{\ddagger}$, the argument in Lemma~\ref{lem:direct summand when epsilon=-1} extends without significant change, as we now describe. We note that the complexes $\scE_{*,t}^{\ddagger}$, $\scF_{*,t}^{\ddagger}$, $\scJ_{*,t}^{\ddagger}$, $\scM_{*,t}^{\ddagger}$ and maps $f^{\pm \mu}$ and $f^{\pm K}$ take the same form as the corresponding $\diamond$ labeled maps in Figures~\ref{fig:f-pm K} and ~\ref{fig:maps-f-pm-mu}. That is, we add no new arrows to these diagrams, and only replace the complexes and elementary maps ($L_W$, $L_Z$, $L_\sigma$ and $L_\tau$) with the $\ddagger$ versions instead of $\diamond$ version. The situation is the same as in the proof of Lemma~\ref{lem:direct summand when epsilon=-1}: the only case in  which a length-$0$ differential could leave $(\frC''')^{\ddagger}$ from generators in $E_{s,\frac{\ell-1}{2}}^{\ddagger}$ and $F_{s,\frac{\ell-1}{2}}^{\ddagger}$ would be if  
   $b_1 = -1$ and $b_2<0$ in the standard complex summand $C(b_1,\dots,b_m)$ of $\cCFK_{\hat{R}}(K)$. However, for knot $K$ in $S^3$, such a standard complex is not realizable, see the last paragraph of \cite{DHSTmore}*{Section~11}. The situation for the row of $E_{s,\frac{\ell+1}{2}}^{\ddagger}$ and $F_{s,\frac{\ell+1}{2}}^{\ddagger}$ is the same.

For length-$1$ arrows, we only need to consider those starting from the four extra generators $\ys_1|\xs_0'$, $\ys_1|\xs_0^{\sfrac{\ell}{2}-1}$, $\ys_{m-1}|\xs_0^{\sfrac{\ell}{2}+1}$ and $\ys_{m-1}|\xs_0^{\sfrac{\ell}{2}+1}$. Since there are no $\sigma$ or $\tau$-terms in $\delta^1_1(\ys_1)$ or $\delta^1_1(\ys_{m-1})$ in $ \cX_n(b_1,\dots, b_m)^{\hat{\cK}}$, we do not need to consider the differentials $\Phi^{K}$ or $\Phi^{-K}$. For the differentials $\Phi^{\mu}$ and $\Phi^{-\mu}$, the possible arrows leaving $(\mathfrak{C}''')^{\ddagger}$ are the leftward one from $\ys_{m-1}|\xs^{\sfrac{\ell}{2}+1}_0\in E^\ddagger_{-\tau(K)-\frac{1}{2},\frac{\ell+1}{2}}$,
 represented by 
\[ L_{\tau}\left(\xs^{\sfrac{\ell}{2}+1}_0\right)\in  \cS^{\ddagger}, \]
and rightward one from $\ys_1 | \xs^{\sfrac{\ell}{2}-1}_0\in  E_{\tau(K)+\frac{1}{2},\frac{\ell-1}{2}}^{\ddagger}$, represented by
\[L_{\sigma}\left(\xs^{\sfrac{\ell}{2}-1}_0\right)\in \cS^{\ddagger}.\]
Under the assumption that $\ell\ge 0$ and that Equation~\eqref{eq:condition for epsilon} holds, these two elements vanish by a small modification to the argument in Lemma~\ref{lem:restriction on P}.
\end{proof}

\section{Thurston norm and slice genus bounds}
\label{sec: Thurston norm}

In this section, we investigate the relationship between the Thurston norm of L-space links and the 4-genus of satellites by L-space patterns.  We introduce some notations. We write $g_4(K)$ for the 4-ball genus of a knot $K\subset S^3$. If $P$ is a winding number $\ell$ satellite, then we write $g_3^{\rel}(P)$ for the minimum genus of a smoothly embedded surface in $S^1\times D^2$ which has boundary equal to $P$ and $|\ell|$ parallel translates of $S^1\times \{pt\}$. 

In this section, we prove the following, which is Theorem~\ref{thm:4-ball-genus-intro} in the Introduction: 

\begin{thm}
	\label{thm:4-ball-genus} Let $K\subset S^3$ be a knot such that $g_4(K)=\tau(K)>0$ and let $P$ be an L-space satellite operator with winding number $\ell$. Then
	\[
	g_4(P(K,0))=\tau(P(K,0))=|\ell|g_4(K)+g_3^{\rel}(P).
	\]
\end{thm}

We note that the inequality 
\[
g_4(P(K,0))\le |\ell|g_4(K)+g_3^{\rel}(P)
\]
is easy, since we can construct a smoothly embedded surface in $B^4$ which bounds $P(K,0)$ by taking a genus $g_3^{\rel}(P)$ surface in $S^1\times D^2$ which bounds $P$ and $|\ell|$ copies of a 0-framed longitude of $S^1\times D^2$, and then capping each longitude in $B^4$ with a parallel copy of a genus $g_4(K)$ slice genus for $K$.

Since $\tau(P(K,0))\le g_4(P(K,0))$ by \cite{OS4ballgenus}, it suffices to show that 
\begin{equation}
\tau(P(K,0))=|\ell|g_4(K)+g_3^{\rel}(P). \label{eq:sufficient-condition-Thurston-norm}
\end{equation}

The key additional result that we need to prove is the following:

\begin{prop}\label{prop:g3-P=Rll/2} Suppose that $P$ is an L-space satellite operator with winding number $\ell$. Then 
	\[
	R_{\sfrac{\ell}{2}}=g_3^{\rel}(P)+|\ell|/2.
	\]
\end{prop} 

Our proof of Proposition~\ref{prop:g3-P=Rll/2} uses two key results. The first is Ozsv\'{a}th and Szab\'{o}'s proof that link Floer homology detects the Thurston norm of the link complement \cite{OSThurstonNorm}, and the second is Liu's computation \cite{BLiuCFK-LSpace} of $\widehat{\HFL}(L)$ for a 2-component L-space link in terms of the $H$-function of $L$.

We additionally consider the slice genus of knots of the form $P(K,n)$ when $P$ is an L-space operator. We will write $g_3^{\rel}(P,n)$ for the minimum genus of a smoothly embedded surface in $S^1\times D^2$ with boundary equal to $P$ and $|\ell|$ parallel copies of a curve with homology class $-\lambda-n \mu$, where $\lambda$ is a 0-framed longitude (oriented so that $P$ is homologous to $\ell \lambda$) and $\mu$ is an oriented meridian. Therefore $g_3^{\rel}(P,0)=g_3^{\rel}(P)$. To understand  the slice genus of knots of the form $P(K,n)$, we will restrict to patterns which have width $N_{L_P}$ equal to $|\ell|/2$.

We will prove the following:
\begin{prop} Suppose that $P$ is an L-space satellite operator with winding number $\ell$.
	\begin{enumerate}
		\item The width $N_{L_P}$ is equal to $|\ell|/2$ if and only if $P$ has wrapping number $|\ell|$ (i.e. $P$ intersects a meridianal disk of $S^1\times D^2$ in exactly $|\ell|$ points).
		\item If $N_{L_P}=\ell/2\ge 0$ and $n\ge 0$, then
		\[
		g_3^{\rel}(P,-n)=g_3(P)+\frac{\ell(\ell-1)}{2} n.
		\]
	\end{enumerate}
\end{prop}

One immediate corollary of the above result, together with our formula for $\tau(P(K,n))$ in Theorem~\ref{thm: tau}, is the following:

\begin{cor} If $K$ is a knot in $S^3$ with $\tau(K)=g_4(K)>0$, and $n\ge 0$, then
	\[
	g_4(P(K,n))=g_3^{\rel}(P,-n)+|\ell| g_4(K).
	\]
\end{cor}

\subsection{Thurston norm and the $H$-function}

Liu \cite{BLiuCFK-LSpace} proved that the $H$-function of a 2-component link detects the Thurston norm of the link complement. In this section, we use this result to prove several results which relate the $H$-function of L-space satellite operators to certain minimal genus problems.

Our most general result is the following:

\begin{prop}
	\label{prop:relative-3-genus-H-function} Suppose that $L=L_1\cup L_2$ is a 2-component L-space link in $S^3$, with linking number $\ell$. Let 
	\[
	M=\min\{s_1: H_L(s,s_2)=H_L(s+1,s_2) \quad \forall s\ge s_1 \textrm{ and } s_2\in \Z+\ell/2\}.
	\]
	Then $M-|\ell|/2$ is equal to the minimum genus of a embedded surface $F$ in $S^3$ which has $F\cap L_1=\d F\cap L_1=L_1$, and such that $F\cap L_2$ consists of $|\ell|$ transverse intersection points.
\end{prop}

To simplify the notation, we will write $g_3^{\rel}(L_1; L_2)$ for the minimal genus in the statement.

There are two steps in our proof of the above. The first is to translate Ozsv\'{a}th and Szab\'{o}'s result about the Thurston norm and link Floer homology \cite{OSThurstonNorm} to obtain a statement about the minimal genus $g_3^{\rel}(L_1; L_2)$ and the support of $\widehat{\HFL}(L)$. As a second step, we translate this condition on $\widehat{\HFL}(L)$ to a statement about the $H$-function of $L$.

\begin{lem}\label{lem:g3-rel=support-HFL}  If $L=L_1\cup L_2$ is a two-component link in $S^3$, then
	\[
	g_3^{\rel}(L_1;L_2)=\max_{(s_1,s_2)\in \bH(L)| \widehat{\HFL}(L,s_1,s_2)\neq 0} s_1-|\ell|/2. 
	\]
\end{lem}
\begin{proof} The result will follow from Ozsv\'{a}th and Szab\'{o}'s result \cite{OSThurstonNorm} that link Floer homology detects the Thurston norm, which we now recall.
	
	If $L\subset S^3$ is a link, and $h\in H_2(S^3,L)$ is a homology class, then $h$ may always be represented by a smoothly embedded surface $F$ in $S^3\setminus \nu(L)$ with boundary on $\d \nu(L)$, such that $[F]=h$. Given $h\in H_2(S^3,L; \Z)$, we define the \textit{Thurston norm} of $h$ to be
	\[
	x(h):=\min_{F\hookrightarrow S^3\setminus \nu(L)| [F]=h} \chi_-(F)
	\]
	where $\chi_-(F)$ is obtained by writing $F$ as a disjoint union of connected surfaces $F_1\cup \cdots \cup F_n$ and setting
	\[
	\chi_-(F):=\sum_{F_i|\chi(F_i)\le 0} -\chi(F_i).
	\]

	Ozsv\'{a}th and Szab\'{o} encode the Thurston norm as follows. We can naturally view $\bH(L)$  as being a subset of $\R^n$, where $n=|L|$. Furthermore, $H_1(S^3\setminus L;\R)$ can be identified with $\R^n$, by identifying the $i$-th meridian $\mu_i$ of $L$ with the standard $i$-th unit vector in $\R^n$. Therefore, we view
	\[
	\bH(L)\subset H_1(S^3\setminus L; \R).
	\]
	
	We set
	\[
	y(h):= \max_{\ve{s}\in \bH(L)| \widehat{\HFL}(L,\ve{s})\neq 0} |\langle \ve{s}, h\rangle |. 
	\]

	Ozsv\'{a}th and Szab\'{o} \cite{OSThurstonNorm} prove that if $L$ is an $n$-component link with no trivial components, then
	\begin{equation}
		x(h)+\sum_{i=1}^n |\langle h, \mu_i \rangle|=2 y(h).\label{eq:Thurston-norm}
	\end{equation}
	In the above, we identify $h\in H_2(S^3,L)$ with its Poincar\'{e} dual, which lies in $H^1(S^3\setminus \nu(L))$. 
	
	Note that if $L$ contains a trivial component, then the claim is straightforward, so we assume that $L$ does not contain a trivial component.

	We now return to the setting of two-component links. We consider the homology class $h\in H_2(S^3,L)$ of surfaces which have boundary equal to $L_1$, and which intersect $L_2$ algebraically in $\ell$-points. Equivalently, we can view these as embedded surfaces $F$ in $S^3\setminus \nu(L)$ which have boundary equal to a longitude of $L_1$ and some number of meridians of $L_2$. If two meridians along $L_2$ have different orientations, then we may tube them together without changing $\chi_-(F)$. Furthermore, any closed component may be deleting without increasing $\chi_-(F)$. Therefore, we may assume that $F$ is a connected surface with boundary equal to a longitude of $L_1$ and exactly $|\ell|$ meridians of $L_2$. The quantity $\chi_-(F)$ is minimized when $g(F)$ is minimized. We compute
	\[
	\chi_-(F)=-\chi(F)=2g(F)-1+|\ell|. 
	\]
	On the other hand, we have
	\[
	\sum_{i=1}^2 |\langle h, \mu_i\rangle|=1 \quad \text{and} \quad 2y(h)=\max_{(s_1,s_2)\in \bH(L)| \widehat{\HFL}(L,s_1,s_2)\neq0} 2|s_1|. 
	\]
	Applying Equation~\eqref{eq:Thurston-norm}, we obtain that
	\[
	2g(F)-1+|\ell|+1=\max_{(s_1,s_2)\in \bH(L)| \widehat{\HFL}(L,s_1,s_2)\neq 0} 2|s_1|.
	\]
	Rearranging, and noting that 
	\[
	\max_{(s_1,s_2)\in \bH(L)| \widehat{\HFL}(L,s_1,s_2) \neq 0} s_1=\max_{(s_1,s_2)\in \bH(L)| \widehat{\HFL}(L,s_1,s_2) \neq 0} |s_1|
	\] by the symmetry of link Floer homology, we obtain
	\[
	g(F)=\max_{(s_1,s_2)\in \bH(L)| \widehat{\HFL}(L,s_1,s_2)\neq 0} s_1- |\ell|/2,
	\]
	which was the claim.
\end{proof}

Next, we relate the support of $\widehat{\HFL}(L)$ to the $H$-function of $L$, when $L$ is an L-space link.

\begin{lem}\label{lem:H-function-v-HFL-support} Suppose that $L$ is a 2-component L-space link with linking number $\ell$. Let 
	\[
	N=\min\{s_1: H_L(s,s_2)=H_L(s+1,s_2) \quad \forall s\ge s_1 \textrm{ and } s_2\in \Z+\ell/2\}.
	\]
	(This coincides with our definition in Definition~\ref{def:width N} of $N_{L_P}$ when $L=L_P$). 
	Then 
	\[
	N=\max_{(s_1,s_2)\in \bH(L)| \widehat{\HFL}(L,s_1,s_2)\neq 0} s_1.
	\]
\end{lem}
\begin{proof} To align our exposition with our previous results about L-space links, we describe our proof in terms of the formality of $\widehat{\HFL}(L)$, though the result may also be derived from Liu's argument in \cite{BLiuCFK-LSpace} using a spectral sequence which converges to $\widehat{\HFL}(L)$, which is defined using the $H$-function of $L$.

	Consider the set $S=\{s_2: H_L(N-1,s_2)=H_L(N,s_2)+1\}$. The set $S$ is non-empty by definition of $N$. Let $h$ be the minimum value of $H_L(N,s_2)$ for $s_2\in S$, and let $n_2$ be the minimum value of $s_2\in S$ so that $H_L(N,s_2)=h$. 
	
	We claim that
	\begin{equation}
		H_L(N+1,n_2)=h,\quad H_L(N,n_2+1)=h, \quad H_L(N,n_2-1)=h+1, \quad H_L(N-1,n_2)=h+1.
		\label{eq:H-function-special-values-support}
	\end{equation}
	To establish the above equation, write
	\[
	H_L(N,n_2+1)=\scN,\quad 	H_L(N+1,n_2)=\scE,\quad H_L(N,n_2-1)=\scS, \quad H_L(N-1,n_2)=\scW.
	\]
	The fact that $\scE=h$ follows from the definition of $N$. Similarly the fact that $\scW=h+1$ follows from the definition of $n_2$. We observe that $\scN\in \{h,h-1\}$ by monotonicity and the bounded gap property of the $H$-function. If $\scN=h-1$, then the monotonicity and bounded gap properties of the $H$-function imply that $H_L(N-1,n_2+1)=h$, which contradicts the definition of $h$. Hence $\scN=h$. Similarly, we know that $\scS\in \{h,h+1\}$. If $\scS=h$, then $H_{L}(N-1,n_2-1)=h+1$ by the aforementioned properties, which contradicts the definition of $n_2$. Therefore Equation~\eqref{eq:H-function-special-values-support} follows. A schematic is shown in Equation~\ref{eq:h-function-N}. 
	\begin{equation}
		\begin{array}{|c|cccccc|}
			\hline
			\hbox{\diagbox{$s_2$}{$s_1$}}&\cdots&N-1&N&N+1&\cdots &\\
			\hline n_2+1&\cdots&* &h&*&\cdots &  \\
			n_2&\cdots&h+1 &h&h&\cdots &\\
			n_2-1&\cdots&* &h+1&*&\cdots& \\
			\hline
		\end{array}
		\label{eq:h-function-N}  	
	\end{equation} 
	
	We now use the above results to conclude 
	\[
	\begin{cases} \bigoplus_{s_2\in \Z+\ell/2}\widehat{\HFL}(L,s_1,s_2)\neq 0 & \text{ if } s_1=N,\\ 
	\bigoplus_{s_2\in \Z+\ell/2}\widehat{\HFL}(L,s_1,s_2)=0 &  \text{ if } s_1>N.\end{cases}
	\]
	Since $L$ is a 2-component L-space link, the complex $\cCFL(L)$ is formal by \cite{CZZ}*{Theorem~1.2}. Therefore, $\cCFL(L)$, viewed as a type-$A$ module over $R_2:=\bF[W_1,Z_1,W_2,Z_2]$, is homotopy equivalent to ${}_{R_2} \cHFL(L)$, viewed as a type-$A$ module with $m_j=0$ if $j\neq 2$. In particular, since we can realize $\widehat{\CFL}(L)$ as the tensor product of $\cCFL(L)$ with $\bF$ (viewed as an $R_2$-module with trivial $W_1$, $Z_1$, $W_2$ and $Z_2$ actions), we can write $\widehat{\CFL}(L)$ as the tensor product over $R_2$ of $\cHFL(L)$ with a free resolution of $\bF$ over $\bF[W_1,Z_1,W_2,Z_2]$. A free resolution of $\bF$ is given by the Koszul complex obtained by tensoring the four complexes of the form
	\[
	\begin{tikzcd} a^* \ar[r, "a"] &1^*\end{tikzcd}
	\] where $a$ is one of $W_1,Z_1,W_2,Z_2$. Here, the element $a^*$ is put in Alexander grading $A(a)$. We will write $\scK_{W_1,Z_1,W_2,Z_2}$ for the tensor product of these Koszul complexes. We can view $\scK_{W_1,Z_1,W_2,Z_2}$ as a type-$D$ module over $R_2:= \bF[W_1,Z_1,W_2,Z_2]$. Therefore we can write
	\begin{equation}
		\widehat{\CFL}(L)\simeq \scK^{R_2}_{W_1,Z_1,W_2,Z_2}\boxtimes {}_{R_2} \cHFL(L).
		\label{eq:Koszul-complex-model-for-CFL-hat}
	\end{equation} 
	Compare to \cite{BLZLattice}*{Equation~6.6}.

	If $\ve{s}\in \bH(L)$, let $\xs_{\ve{s}}$ denote $\bF[U]$-tower generator of $\cHFL(L_P,\ve{s})$. Note that $1^*\otimes \xs_{(N,n_2)}$ is a cycle. We claim that it is not a boundary. This follows from the fact that $\xs_{(N,n_2)}$ is not in the image of any of $W_1$, $Z_1$, $W_2$, $Z_2$ by the form of the $H$-function shown in Equation~\eqref{eq:h-function-N}. In particular, $\bigoplus_{s_2\in \Z+\ell/2}\widehat{\HFL}(L,N,s_2)\neq 0$.
	
	We now claim that $\widehat{\HFL}(L,s_1,s_2)=0$ for all $s_2$ and $s_1>N$. We can rewrite the model of $\widehat{\CFL}(L_P)$ from Equation~\eqref{eq:Koszul-complex-model-for-CFL-hat} as
	\[\Cone\left(\begin{tikzcd}\scK_{W_1,W_2,Z_2}\boxtimes \cHFL(L) \{0,1\}\ar[r, "\id\otimes Z_1"] &\scK_{W_1,W_2,Z_2}\boxtimes \cHFL(L) \end{tikzcd}\right).
	\] 
	Here $\{0,1\}$ denotes an upward shift in the multivariate Alexander grading by $(0,1)$.
	
	We recall an elementary fact that if $A$ and $B$ are chain complexes and $f\colon A\to B$ is an injective chain map, then $\Cone(A\xrightarrow{f} B)$ is quasi-isomorphic to $B/f(A)$. Since $Z_1$ acts injectively on $\cHFL(L)$, the above complex is quasi-isomorphic to
	\begin{equation}
\scK_{W_1,W_2,Z_2}\boxtimes (\cHFL(L)/Z_1 \cHFL(L)). \label{eq:quotient-of-chain-complexes-Koszul}
	\end{equation}
	However, the image of the action of $Z_1$ on $\cHFL(L)$ contains the entire subspace of elements which lie in $A_1$-Alexander grading above $N$. This is because the $H$-function satisfies $H_L(s_1+1,s_2)=H_L(s_1,s_2)$ for any $s_1\ge N$.  Since $\scK_{W_1,W_2,Z_2}$ is supported only in non-positive $A_1$-Alexander gradings, it follows that Equation~\eqref{eq:quotient-of-chain-complexes-Koszul} has no elements in $A_1$-Alexander gradings above $N$, completing the proof. 
\end{proof}

Note that Proposition~\ref{prop:relative-3-genus-H-function} follows immediate from Lemmas~\ref{lem:g3-rel=support-HFL} and ~\ref{lem:H-function-v-HFL-support}. 

We now apply Proposition~\ref{prop:relative-3-genus-H-function} to a link of the form $L_P$, for  an L-space pattern $P$. There are two ways to apply it, depending on which component we label as $L_1$ or $L_2$. We obtain the following:

\begin{cor}
	\label{cor:Thurston-norm-N-R}
	Suppose that $P$ is an L-space satellite operator.
	\begin{enumerate}
		\item The quantity $N_{L_P}-|\ell|/2$ is equal to the minimum genus of an embedded, oriented surface in $S^3$ with boundary $\mu$ and which intersects $P$ in exactly $|\ell|$ points, transversely.
		\item The quantity $R_{\sfrac{\ell}{2}}-|\ell|/2$ is equal to the minimum genus of an embedded, oriented surface in $S^1\times D^2$ which has boundary equal to $P$ and to $|\ell|$ parallel copies of a 0-framed longitude on the boundary of $S^1\times D^2$. (Previously, we wrote $g_3^{\rel}(P)$ for this quantity). 
	\end{enumerate}
\end{cor}
The second part of the above Corollary is stated as Proposition~\ref{prop:R_ell/2-intro} in the introduction and Proposition \ref{prop:g3-P=Rll/2} in this section.

\subsection{Slice genus bounds}

We now prove the main theorem of this section:

\begin{proof}[Proof of Theorem~\ref{thm:4-ball-genus}]
	Let $K$ be a knot with $\tau(K)=g_4(K)>0$, and $P$ be an L-space satellite operator. We may assume, by possibly reversing the string orientation of $P$, that $\ell=w(P)\ge 0$.
	
	As observed in Equation~\eqref{eq:sufficient-condition-Thurston-norm}, it is sufficient to show that $\tau(P(K,0))$ is equal to $g_3^{\rel}(P)+\ell g_4(K)$.
	
	By \cite{HomTauCables}*{Corollary~4}, we know that $\veps(K)=1$. Therefore our Theorem~\ref{thm:main-intro} implies that
	\[
	\tau(P(K,0))= R_{\sfrac{\ell}{2}}-\ell/2+\ell\tau(K).
	\]
	By Corollary~\ref{cor:Thurston-norm-N-R} the above is equal to
	\[
	g_3^{\rel}(P)+\ell\tau(K)= g_3^{\rel}(P)+\ell g_4(K),
	\]
	completing the proof.
\end{proof} 

We also have the following statement when $\ell=0$, which is stated as Proposition~\ref{prop:slice-genus-bounds-winding-number-0-intro} in the introduction:
\begin{prop}
 Suppose that $K\subset S^3$ is a knot with $g_4(K)=\tau(K)>0$, and $P$ is an L-space satellite operator with winding number 0. Then
\[
g_4(P(K,n))=g_3^{\rel}(P)
\]
for all $n<2\tau(K)$. 
\end{prop}
\begin{proof} Since the winding number is zero, $g_3^{\rel}(P)$ is the minimum genus of a surface in $S^1\times D^2$ which has boundary equal to $P\subset S^1\times D^2$. Therefore $g_4(P(K,n))\le g_3^{\rel}(P)$ for any $n$. Furthermore, since $\tau(K)=g_4(K)>0$,  we have that $\veps(K)=1$, as before. When $n<2\tau(K)$, the  Theorem~\ref{them:tau for epsilon=1} and Corollary~\ref{cor:Thurston-norm-N-R} imply that
\[
\tau(P(K,n))=g_3^{\rel}(P).
\]
Therefore $g_4(P(K,n))=\tau(P(K,n))$, completing the proof.
\end{proof}

\subsection{Patterns $P$ with minimal wrapping number}

If $P$ is a satellite pattern, the \emph{wrapping number} of $P$ is defined to be the minimal geometric intersection number of $P$ with a meridianal disk of $S^1\times D^2$. In this section, we study L-space patterns $P$ with wrapping number $|\ell|/2$. By Corollary~\ref{cor:Thurston-norm-N-R}, these are exactly the patterns with width $N_{L_P}$ equal to $|\ell|/2$. In particular, any pattern where $P$ is braided about $\mu$ satisfies this condition (in particular cables and 1-bridge braids satisfy this condition). 

Recall that $g_3^{\rel}(P,n)$ is the minimum genus of a smoothly embedded surface in $S^1\times D^2$ with boundary equal to $P$ and $|\ell|$ parallel copies of a curve with homology class $-\lambda-n \mu$, where $\lambda$ is a 0-framed longitude (oriented so that $P$ is homologous to $\ell \lambda$) and $\mu$ is an oriented meridian. In this section, we prove the following:

\begin{lem}\label{lem:3-genus-n-framed} Suppose that $P$ is an L-space pattern with $\ell\ge 0$ and with width $N=\ell/2$. Then, for any $n\ge 0$, we have
	\[
	g_3^{\rel}(P,-n)=g_3(P)+\frac{\ell(\ell-1)}{2}n.
	\]
\end{lem}
\begin{proof}
	We first claim that $g_3^{\rel}(P,-n)$ is equal to the maximum value of $s_2+n\ell s_1$ for $(s_1,s_2)$ in the support of $\widehat{\HFL}(L_P)$.
	
	The quantity $g_3^{\rel}(P,-n)$ measures embedded surfaces $F$ in $S^1\times D^2$ with boundary equal to the union of $P$ and $\ell$ parallel curves which have homology class $-\ell \lambda+n \mu$, where $-\lambda$ is a 0-framed longitude (which we can think of as a meridian of the component $\mu\subset L_P$). If $h$ is the class of this surface, we observe, using similar reasoning as in Lemma~\ref{lem:g3-rel=support-HFL}, that a smoothly embedded surface $F$ representing $h$ minimizes $\chi_-(h)$ if and only if it minimizes $g(F)$. Therefore we have
	\[
	x(h)=2g(F)+\ell-1.
	\]
	We also compute that
	\[
	\sum_{i=1}^2 |\langle h,\mu_i\rangle |=1+n\ell \quad \text{and} \quad \langle \ve{s},h\rangle=s_1 n\ell+s_2. 
	\]
	Therefore $y(h)$ is the maximum of $s_1n\ell+s_2$ ranging over $(s_1,s_2)$ which are supported by $\widehat{\HFL}(L_P)$. 
	
	Note that since $N=\ell/2$, we must have $R_{\sfrac{\ell}{2}}=g_3(P)+\ell/2$ by the stabilization property of the $H$-function, stated in Lemma~\ref{lem: properties of H function}.
	
	It follows from Lemma~\ref{lem:H-function-v-HFL-support} (applied once with $L_1=\mu$, $L_2=P$ and once with these roles reversed) that if $L_P$ has width $N=|\ell|/2$, there are no points $(s_1,s_2)$ in the support of $\widehat{\HFL}(L_P)$ with $s_1>\ell/2$ or $s_2>g_3(P)+\ell/2$. Since $n,\ell\ge 0$, we conclude that if $(s_1,s_2)$ is in the support of $\widehat{\HFL}(L_P)$, then
	\begin{equation}
	s_1 n\ell+s_2\le (\ell/2) n \ell+ g_3(P)+\ell/2. \label{eq:inequality-support}
	\end{equation} 
	On the other hand, $(\ell/2,g_3(P)+\ell/2)$ is in the support of $\widehat{\HFL}(L_P)$ by the proof of Lemma~\ref{lem:H-function-v-HFL-support}, (because the point $(N,n_2)$ in the proof of the aforementioned lemma is clearly equal in this case to $(\ell/2,g_3(P)+\ell/2)$), and $(\ell/2,g_3(P)+\ell/2)$ achieves equality in Equation~\eqref{eq:inequality-support}, so we conclude
	\[
	y(h)=\ell^2 n/2+g_3(P)+\ell/2.
	\]
	Using Ozsv\'{a}th and Szab\'{o}'s Thurston norm detection result, stated above in Equation~\eqref{eq:Thurston-norm}, we have
	\[
	2g(F)=\ell^2 n+2g_3(P)+\ell-1-n\ell-\ell+1=\ell(\ell-1)n+g_3(P),
	\]
	as claimed.
\end{proof}

\begin{rem}\label{rem:n=0-3-genus-and-R_l} The case that $n=0$ is of particular importance. Combining Corollaries~\ref{cor:Thurston-norm-N-R} and Lemma~\ref{lem:3-genus-n-framed}, we obtain
	\[
	R_{\sfrac{\ell}{2}}-\ell/2=g_3^{\rel}(P)=g_3(P).
	\]This is stated as Proposition~\ref{prop:minimally-wrapped-seifert-genus-intro} in the introduction. 
\end{rem}

\section{Examples}

\label{sec:examples}

In this section, we illustrate the formulas for $\tau(P(K,n))$ using three families of examples: the cabling operation, $1$-bridge braid, and the generalized Mazur patterns.

\subsection{Cables}
If $p,q\in \Z$ are non-zero integers with $\gcd(p,q)=1$ and $p>0$, then we write $K_{p,q}$ for the $(p,q)$-cable of $K$. We take the convention that $K_{p,q}$ winds $p$ times longitudinally and $q$ times meridianally in the boundary of a tubular neighborhood of $K$. We write $q=np+r$ where $0<r<p$. We let $P$ be the cabling operator where $L_P$ is obtained by taking the $(p,r)$ cable of one component of the Hopf link. Then
\[
K_{p,q}=P(K,n).
\]
Since the cabling operator is braided in $S^1\times D^2$, we know by Corollary~\ref{cor:Thurston-norm-N-R} that
\[
N=\ell/2=p/2, \quad R_{\ell/2}-\ell/2=g_3(P)=\frac{(p-1)(r-1)}{2}.
\] 
If $\veps(K)>0$, Theorem~\ref{thm: tau} implies
\[
\tau(P(K,n))=g_3(P)+\frac{(p-1)p}{2} n+ p \tau(K)=\frac{(p-1)(q-1)}{2}+p\tau(K).
\]
When $\veps(K)<0$, Theorem~\ref{thm: tau} does not apply directly, since the condition $ R_{\sfrac{\ell}{2}-1}\ge g_3(P)+\ell/2-1$ is typically not satisfied. Instead, we can use duality by observing that
\[
-K_{p,q}=(-K)_{p,-q}
\]
and that $\veps(-K)>0$. Using the result in this case, we obtain
\[
\tau(K_{p,q})=-\tau(-K_{p,q})=-\tau( (-K)_{p,-q})=\frac{(p-1)(q+1)}{2}+p\tau(K),
\]
recovering Hom's formula from \cite{HomTauCables}. Note that the case that $\veps(K)=0$ is straightforward because the knot Floer complex of $K$ is locally equivalent to the knot Floer complex of the unknot, so we see that
\[
\tau(K_{p,q})=\tau(U_{p,q})=\tau(T_{p,q})=\begin{cases}
\frac{(p-1)(q-1)}{2} & \text{ if } q>0 \\
\frac{(p-1)(q+1)}{2} & \text{ if } q<0.
\end{cases}
\]
in this case.

\subsection{1-bridge braids}

	Let $p,q,b$ be integers such that $p\ge2$, $q\neq np$ or $-1+np$ for some $n$, (otherwise we are guaranteed to get a link for the following braid closure), and $0< b<p-1$. (If $b=0$ or $b=p-1$, then $B(p,q,b)$ is a torus knot). Let $B(p,q,b)$ be the braid closure of $(\sigma_b\sigma_{b-1}\cdots \sigma_1) (\sigma_{p-1}\sigma_{p-2}\cdots\sigma_{1})^q$.  Greene, Lewallen and Vafaee proved in \cite{GLVLspace} that when the braid closure is a knot, then $B(p,q,b)$ is an L-space knot when $q\ge1$. From now on, we only consider the case when the braid closure $B(p,q,b)$ gives a knot. These knots are called \emph{1-bridge braids}. This family (together with torus knots) includes the \emph{Berge-Gabai knots} \cite{Gabai_1-bridge_braids}, which are the knots in the solid torus which admit a solid torus surgery.

	Let $L(p,q,b) = B(p,q,b) \cup \mu$ be the corresponding $2$-component link by viewing  $B(p,q,b)$ as lying in the solid torus, and $\mu$ denotes the meridian of the solid torus. Note that adding a positive full twist along $\mu$ changes $B(p,q,b)$ to $B(p,q+p,b)$. Therefore, by \cite{CZZ}*{Lemma~6.2} (see also \cite{LiuLSpaceLinks}*{Lemma~2.6}), $L(p,q,b)$ is an L-space link when $q\ge p+1$. See Figure \ref{fig:1bridgebraid} for an example.
	\begin{figure}
\scalebox{.5}{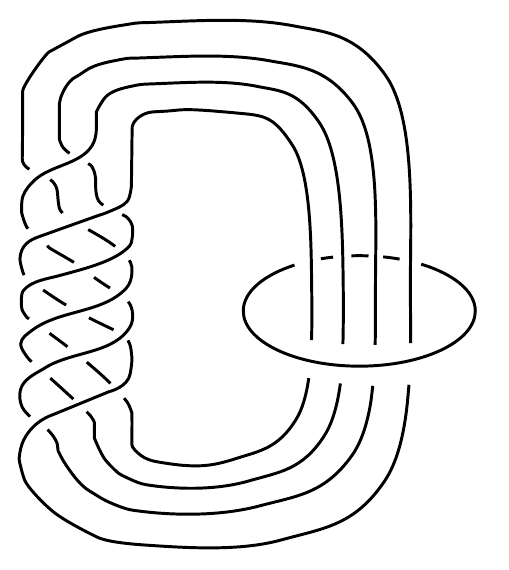}
\caption{The link $L(4,5,2)$.}\label{fig:1bridgebraid}
\end{figure}

	Let $P_{p,q,b}$ be the satellite operation with the pattern $B_{p,q,b}$. Note that 
	\[P_{p,q,b}(K,n) =P_{p,q+pn,b}(K,0), \]
	so it suffices to consider the case when $ p<q<2p$, and other values of $q$ could be deduced by adjusting the framing $n$ appropriately. Write $ q= r+pn$, with $p< r<2p$.  
	
	For the simplicity of notation, define
	\[K_{p,q,b} := P_{p,q,b}(K,0) = P_{p,r,b}(K,n). \] 
	In this example, we compute the formula for $\tau(K_{p,q,b})$ using Theorem \ref{thm: tau}, which recovers the result in \cite{ChenHanselmanSatellites}*{Theorem~6.8}.
	
	By definition, $B(p,q,b)$ is braided in $S^1\times D^2$, so by Corollary~\ref{cor:Thurston-norm-N-R}, we have 
	\[N=\ell/2 = p/2,\quad R_{\sfrac{\ell}{2}}-\ell/2=g_3(B(p,r,b))=\frac{(p-1)(r-1)+b}{2}, \] where the Seifert genus $g_3(B(p,r,b))$ is easy to obtain as it is a positive knot.
	
	If $\veps(K)>0$, Theorem~\ref{thm: tau} implies
	\[
	\tau(K_{p,q,b})=g_3(B(p,r,b))+\frac{(p-1)p}{2} n+ p \tau(K)=\frac{(p-1)(q-1)+b}{2}+p\tau(K).
	\]
	
	If $\veps(K)<0$, again we can't apply Theorem~\ref{thm: tau} directly for the same reason as in the cabling case. But we can use the same mirroring trick. Note that the mirror of $B(p,q,b)$ is $B(p,-q-1,p-b-1),$ so 
	\[-(K_{p,q,b}) = (-K)_{p,-q-1,p-b-1},\] where $-K$ denote the mirror of a knot $K$.  Therefore, when $\veps(K)<0$, we have 
	\[\tau(K_{p,q,b}) = -\tau(-(K_{p,q,b})) = -\tau((-K)_{p,-q-1,p-b-1}) =\frac{(p-1)(q+1)+b}{2}+p\tau(K). \]
	
	If $\veps(K)=0$, we have
	\[\tau(K_{p,q,b}) = \tau(U_{p,q,b}) =\tau(B(p,q,b))=\begin{cases}
		\frac{(p-1)(q-1)+b}{2} & \text{ if } q>0, \\
		\frac{(p-1)(q+1)+b}{2} & \text{ if } q<-1,
	\end{cases}\] using the fact that $B(p,q,b)$ is an L-space knot when $q>0$, and the same mirror trick for the $q<-1$ case. (When $q=0$ or $-1$, the braid closure $B(p,q,b)$ is a link instead of a knot.) 

\subsection{Satellites by a family of $2$-bridge links}\label{ss:2bridgelinks}

Consider $2$-bridge links of the form $L(rq-1,q)$, where $r$ and $q$ are positive odd integers (excluding the case $r=q=1$). Here we follow the convention in \cite{rationalknotskl} for the slope of rational tangles, so that $L(rq-1,q)$ corresponds to $K(\frac{rq-1}{q})$ in their notation. By \cite{SchubertTwobridgeknots}, the (unoriented) link $L(rq-1,q)$ is isotopic to $L(rq-1,r)$, so we may and will assume $r\ge q$ by switching $r$ and $q$ if necessary in the rest of this example. See Figure \ref{fig:twobridgefamily}.
	
We denote by $P_{r,q}$ the satellite operation such that the corresponding $2$-component link $L_{P_{r,q}}$ is $L(rq-1,q)$. For example, the identity operator (for which the corresponding $2$-component link is the positive Hopf link) correspond to the case $r=3,q=1$, the (positive clapsed) Whitehead double corresponds to the case $r=q=3$, the Mazur pattern corresponds to $r=5,q=3$.  
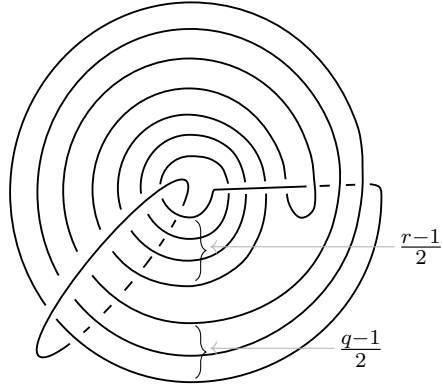
\begin{figure}
\begin{tikzpicture}
    \node at (0,0){\includegraphics[scale=0.5]{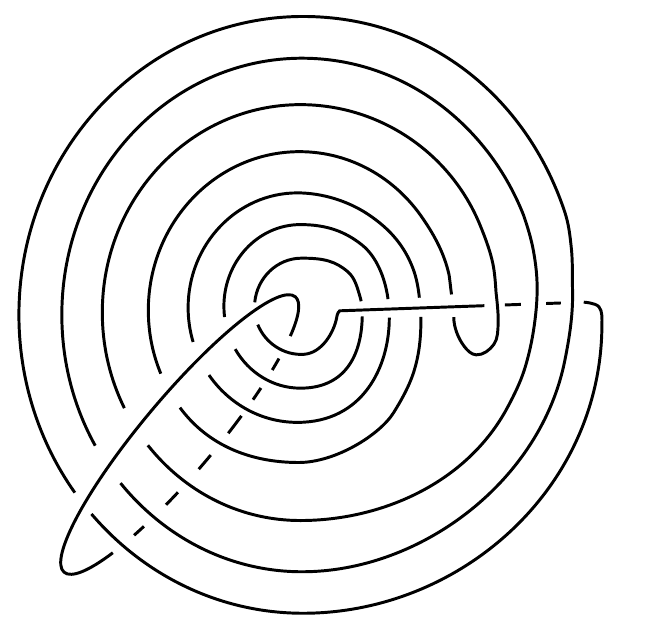}};
     \draw [decorate,decoration={brace,amplitude=4pt},xshift=0.5cm,yshift=0pt]
        (-0.7,-0.4) -- (-0.7,-1.2);
\node at (2.8,-0.75) {$\frac{r-1}{2}$};
     \draw [decorate,decoration={brace,amplitude=4pt},xshift=0.5cm,yshift=0pt]
        (-0.7,-1.8) -- (-0.7,-2.5);
\node at (2,-2.1) {$\frac{q-1}{2}$};
\draw [<- ,   gray!60] (0,-0.75) -- (2.45,-0.75);
\draw [<- ,   gray!60] (0,-2.1) -- (1.65,-2.1);
\end{tikzpicture}
\caption{The link $L(rq-1,q)$ for odd integers $r \geq q>1$. The picture depicts the case when $r=9$ and $q=7$.}\label{fig:twobridgefamily}
\end{figure}

 In \cite{LiuLSpaceLinks}*{Theorem~3.8}, Liu proved that each $L(rq-1,q)$, where $q$ and $r$ are positive odd integers, is a $2$-component L-space link; in his notation, $L(rq-1,q)$ is denoted by $b(rq-1,-q)$. Since each component of $L(rq-1,q)$ is an unknot, the extra condition in Equation (\ref{eq:extra condition when epsilon=0}) in Theorem \ref{thm: tau} is satisfied by Remark \ref{rem:extra assumption}. In this example, we will compute the formula for $\tau(P_{q,r}(K,n))$ using Theorem \ref{thm: tau}. The main effort lies in writing down the multivariable Alexander polynomial of $L(rq-1,q)$, which we can fortunately use an algorithm developed by Hoste in \cite{HosteAlexander}.
 
 In \cite{TulerLinking}, Tuler gives a formula to the compute the linking number of $2$-bridge links. In particular, the linking number of $L(rq-1,q)$, with a specific orientation specified by the bridge presentation, is given by 
 \[\ell = \frac{r-q}{2}, \]
 which is non-negative since we have assumed $r\ge q$.

 In \cite{HosteAlexander}, Hoste gives an algorithm to compute the multivariable Alexander polynomials of $2$-bridge links $L(p,q)$ with the same orientations specified by the bridge presentation as above. Note that $L(p,q)$ corresponds to $K_{q/p}$ in his convention. (The only parity of $p$ and $q$ such that $L(p,q)$ is a two-component link instead of a knot is when $p$ is even and $q$ is odd, so there should be no danger of confusion with the convention used in \cite{HosteAlexander}). The closed-form formula is \[\Delta_{L(p,q)}(x_1,x_2) \doteq \sum_{i=1}^{p}\eta_{2i-1}x_1^{\sum_{j=1}^{i-1}\eta_{2j}}x_2^{\frac{\eta_{2i-1}-1}{2}+\sum_{k=1}^{i-1}\eta_{2k-1}},\]
where $\eta_{i} = (-1)^{\lfloor \frac{iq}{p}\rfloor}.$ 

More interestingly, he gives a combinatorial description of the algorithm in terms of walks on a $2$-dimensional grid. See \cite{HosteAlexander}*{Algorithm~3}. In the special case when $p=rq-1$, as considered here, the sequence of signs $\eta_i= (-1)^{\lfloor \frac{iq}{rq-1}\rfloor} $, $i=1,\dots,rq-2$, is given by $r-1$ many $+$'s, followed by an alternating sequence of $r$ many $-$'s and $r$ many $+$'s, repeated $q - 2$ times, and ending with $r - 1$ many $+$'s. (When $q=1$, this sequence is just $r-1$ many $+$'s.) For example, when $r=5,q=3$, this sequence $\left\{\eta_i\right\}_{i=1}^{13}$ is given by 
\[++++-----++++.\] 
(The linking number is obtained by summing up $\eta_{2k+1}$ for $k=0,\dots,(rq-3)/2$.) The corresponding walk on the $2$-dimensional grid is as follows:
\begin{enumerate}
	\item Start at (0,0), then move $(r-3)/2$ steps of $(1,1)$ (one unit right and one unit up).
	\item Move one step of $(1,0)$ (one unit right), followed by $(r-1)/2$ steps of $(-1,-1)$ (one unit left and one unit down), then one step of $(1,0)$, followed by $(r-3)/2$ many steps of $(1,1)$.
	\item Repeat step (2) for $(q-3)/2$ more times. (If $q=1$, then stop after step (1).)
\end{enumerate}
Then, if $k$ is the number of times the walk visits the point $(i,j)$, the (unsymmetrized) Alexander polynomial will contain the term $(-1)^{i+j}kx_1^ix_2^j$, up to a global sign. In the case $p=rq-1$, each point is visited either $0$ or $1$ time, so every monomial in $\Delta_{L(rq-1,q)}(x_1,x_2)$ has coefficient $\pm 1$, as it should be. Below is the example of the walk when $r=5,q=3$, which is the Mazur link. 
	\begin{figure*}[h!]
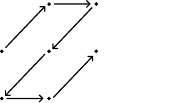
\end{figure*}

The (unsymmetrized) Alexander polynomial of this example is then given by 
\[\Delta_{L(14,3)}(x_1,x_2)\doteq 1+x_1x_2-x_1^2x_2-x_1\\ -x_2^{-1}+x_1x_2^{-1}+x_1^2.\]

In general, the support of the multivariable Alexander polynomial takes the following form in Figure \ref{fig:support of the Alexander polynomial of 2-bridge link}, assuming $r\geq q \geq 3$. When $r>3,q=1$, the walk simply consists of $(r-3)/2$ steps of $(1,1)$, so the support of the multivariable Alexander polynomial is a single line segment of slope $1$.  When $r=3$ and $q=1$, the walk is the empty one, the Alexander polynomial consists of a single monomial $1$, and the link is the positive Hopf link.	\begin{figure}[h!]
\begingroup%
  \makeatletter%
  \providecommand\color[2][]{%
    \errmessage{(Inkscape) Color is used for the text in Inkscape, but the package 'color.sty' is not loaded}%
    \renewcommand\color[2][]{}%
  }%
  \providecommand\transparent[1]{%
    \errmessage{(Inkscape) Transparency is used (non-zero) for the text in Inkscape, but the package 'transparent.sty' is not loaded}%
    \renewcommand\transparent[1]{}%
  }%
  \providecommand\rotatebox[2]{#2}%
  \newcommand*\fsize{\dimexpr\f@size pt\relax}%
  \newcommand*\lineheight[1]{\fontsize{\fsize}{#1\fsize}\selectfont}%
  \ifx\svgwidth\undefined%
    \setlength{\unitlength}{208.92950876bp}%
    \ifx\svgscale\undefined%
      \relax%
    \else%
      \setlength{\unitlength}{\unitlength * \real{\svgscale}}%
    \fi%
  \else%
    \setlength{\unitlength}{\svgwidth}%
  \fi%
  \global\let\svgwidth\undefined%
  \global\let\svgscale\undefined%
  \makeatother%
  \begin{picture}(1,0.85045864)%
    \lineheight{1}%
    \setlength\tabcolsep{0pt}%
    \put(0,0){\includegraphics[width=\unitlength,page=1]{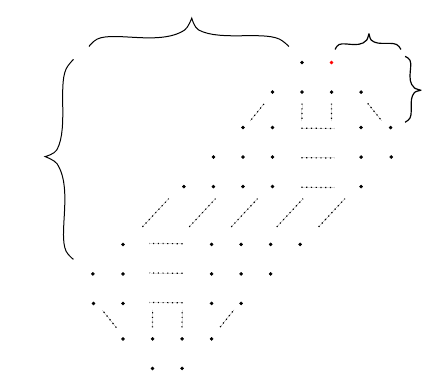}}%
    \put(0.40,0.85){\color[rgb]{0,0,0}\makebox(0,0)[lt]{\lineheight{1.25}\smash{\begin{tabular}[t]{l}$ \frac{r-3}{2}$\end{tabular}}}}%
    \put(0.81,0.83253418){\color[rgb]{0,0,0}\makebox(0,0)[lt]{\lineheight{1.25}\smash{\begin{tabular}[t]{l}$\frac{q-3}{2} $\end{tabular}}}}%
    \put(0.98375654,0.63344528){\color[rgb]{0,0,0}\makebox(0,0)[lt]{\lineheight{1.25}\smash{\begin{tabular}[t]{l}$\frac{q-3}{2} $\end{tabular}}}}%
    \put(-0.00174781,0.48045196){\color[rgb]{0,0,0}\makebox(0,0)[lt]{\lineheight{1.25}\smash{\begin{tabular}[t]{l} $ \frac{r-3}{2}$\end{tabular}}}}%
  \end{picture}%
\endgroup%

	\caption{Support of the Alexander polynomial of $2$-bridge links $L(rq-1,q)$}
	\label{fig:support of the Alexander polynomial of 2-bridge link}
\end{figure}
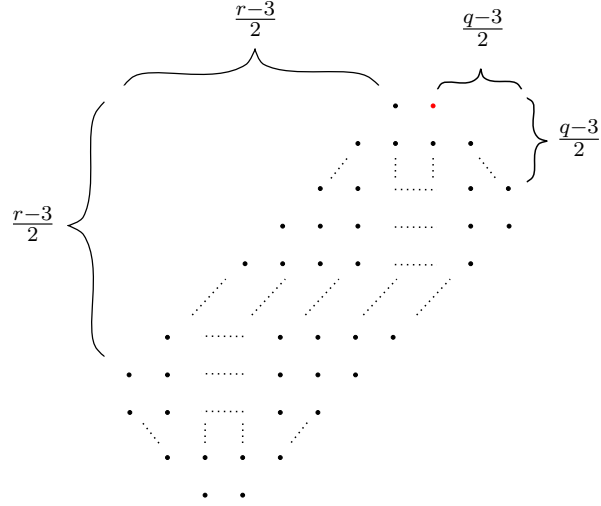

After symmetrizing and multiplying by $x_1^{\sfrac{1}{2}}x_2^{\sfrac{1}{2}}$, the monomial represented by the red dot in Figure \ref{fig:support of the Alexander polynomial of 2-bridge link}, which has the highest power of $x_2$ in $\tilde{\Delta}_{L(rq-1,q)}(x_1,x_2)$, corresponds to the monomial \[-x_1^{(r-q+4)/4}x_2^{(r+q-2)/4}.\]
(The sign is chosen such that the $H$-function obtained below takes non-negative values.) 

Recall the algorithm to obtain the $H$-function from $\tilde{\Delta}$ from Proposition \ref{prop:GNH-function}. We start with the stabilized state 
\[
J(t,r)=H_{L_1}(t-\ell/2)+H_{L_2}(t-\ell/2),
\]
where in this case, both $L_1$ and $L_2$ are unknot. To form $H_L(t,r)$, for each non-trivial summand $a_{j,k}x_1^jx_2^k$ in $\tilde{\Delta}_L(x_1,x_2)$, we add $-a_{j,k}$ to each of $J(t,r)$ for $t< j,$ and $r< k$. From the support of the nomalized alexander polynomial $\tilde{\Delta}_{L(rq-1,q)}$, it is easy to see that 
\[\begin{split}
	&H_{L(rq-1,q)}\left(\tfrac{\ell}{2}-1, \infty \right)= H_{L(rq-1,q)}\left(\tfrac{\ell}{2}-1, \tfrac{r+q-6}{4}\right)=1,
	\\ 
	 &H_{L(rq-1,q)}\left(\tfrac{\ell}{2}-1, \tfrac{r+q-10}{4}\right)=2,  \\
	  &H_{L(rq-1,q)}\left(\tfrac{\ell}{2}, \infty\right)= H_{L(rq-1,q)}\left(\tfrac{\ell}{2}, \tfrac{r+q-2}{4} \right)=0,\\ &H_{L(rq-1,q)}\left(\tfrac{\ell}{2}, \tfrac{r+q-6}{4}\right)=1, \\	 
	  &H_{L(rq-1,q)}\left(\tfrac{\ell}{2}+1,\infty \right)=H_{L(rq-1,q)}\left(\tfrac{\ell}{2}+1, \tfrac{r+q-6}{4} \right)=0,\\
	  &H_{L(rq-1,q)}\left(\tfrac{\ell}{2}+1, \tfrac{r+q-10}{4}\right)=1,
\end{split}
\]
where $\ell = (r-q)/2$. Therefore, we have 
\[R_{\sfrac{\ell}{2}-1} =R_{\sfrac{\ell}{2}+1}= \frac{r+q-6}{4},\quad  R_{\sfrac{\ell}{2}}=\frac{r+q-2}{4}.\]

By Theorem \ref{thm: tau}, we have the following formula for $\tau(P_{r,q}(K,n))$ (recall we have assumed $r\ge q$; otherwise, simply switch $r$ and $q$):
\begin{enumerate}
	\item If $\veps(K) =1$, then 
	\[\tau(P_{r,q}(K,n) ) = \begin{cases}
		\frac{q-1}{2}+ \frac{(\ell-1)\ell}{2}n+\ell\tau(K) & \text{if $n<2\tau(K)$,}\\
		\frac{(\ell-1)\ell}{2}n+\ell\tau(K) & \text{if $n\ge 2\tau(K)$. }
	\end{cases} 
	\] 
	\item If $\veps(K) =0$, then 
	\[\tau(P_{r,q}(K,n) ) = \begin{cases}
		\max\{\frac{q-1}{2}, \frac{r-3}{2}\}+ \frac{(\ell-1)\ell}{2}n+\ell\tau(K) & \text{if $n<0$,}\\
		\frac{(\ell-1)\ell}{2}n+\ell\tau(K) & \text{if $n\ge 0$.}
	\end{cases} 
	\] 
	\item If $\veps(K) =-1$, then 
	\[\tau(P_{r,q}(K,n) ) = \begin{cases}
		\max\{\frac{q-1}{2}, \frac{r-3}{2}\}+ \frac{(\ell-1)\ell}{2}n+\ell\tau(K) & \text{if $n<2\tau(K)$,}\\
		\frac{r-3}{2}+\frac{(\ell-1)\ell}{2}n+\ell\tau(K) & \text{if $n=2\tau(K)$ or $2\tau(K)+1$,}\\
		\min\{ \frac{r-3}{2},\frac{r-q}{2}\}+ \frac{(\ell-1)\ell}{2}n+\ell\tau(K) & \text{if $n>2\tau(K)+1$.}\\
	\end{cases} 	
	\] 
\end{enumerate}

The family of two-bridge links we consider in this example is the same as the one studied in \cite{PX-twobridge}, where $P_{r,q}(K,0)$ corresponds to the satellite knot $Q_{\frac{r-1}{2},\frac{q-1}{2}}(K)$ in their notation. See \cite{PX-twobridge}*{Proposition~4.3} for their parameterization of this family of two-bridge links. The difference is that we also take into account the effect of different framings of the companion knot, while they focus on the case when the framing $n=0$. In particular, when $r=q=3$, we recover the formula of $\tau$ under Whitehead double in \cite{HeddenWhiteheadDoubles}. When $r=5,q=3$, we recover the formula of $\tau$ for Mazur pattern in \cite{LevineMazur}. For $r= 2j+3,q=3$, we recover the formula for $\tau$ under the family $Q^{0,j}$ in \cite{Bodish-Satellites}.

\section{Homomorphism obstructions for satellite operators}

\label{sec: non homomorphism}
In this section, we use our formulas for $\tau(P(K,n))$ in Theorem~\ref{thm: tau} to partially confirm a conjecture by Hedden and Pinz\'on-Caicedo on when satellite operators act as a group homomorphism on the smooth concordance group for knots in $S^3$, under the extra assumption that $P$ is an $L$-space satellite operation. We prove the following, which is stated as Theorem \ref{thm:concordance-homomorphism-intro} in the introduction:

\begin{thm}
	Let $\scC$ denote the smooth concordance group of knots in $S^3$. If $L_P = \mu \cup P$ is an $L$-space link, and the map 
\begin{align*}
		P: \scC &\longrightarrow \scC\\
		K &\longmapsto P(K,n)
	\end{align*}
is a group homomorphism for some $n\ge 0$, then 
$L_P$ is the $2$-component unlink $U_2$, the positive Hopf link $\cH_+$, or the negative Hopf link $\cH_-$. Therefore, $P(-,n)$ induces the trivial map, the identity map, or the orientation-reversing map on the concordance group.
\end{thm}

\begin{proof}
	
 By \cite{BLiuCFK-LSpace}, we can compute $\widehat{\HFL}$ of a two-component L-space link from its $H$-function, and $\widehat{\HFL}$ detects the $2$-component unlink, the positive Hopf link $\cH_+$, and the negative Hopf link $\cH_-$. Therefore, it is enough to show the $H$-function $H_{L_P}$ agrees with one of the $H$-functions $H_{U_2}, H_{\cH_+},H_{\cH_-}.$ See Figure \ref{fig:h-function-examples} for the $H$-functions of the unlink $U_2$ and positive/negative Hopf links $\cH_{\pm}$ respectively.
	
    By reversing the orientation of $P$ if necessary, we can assume the linking number $\ell = \lk (\mu,P)$ to be non-negative, and we will prove $H_L = H_{U_2}$ or $H_{\cH_+}$ under the assumptions in the theorem. 
	
	We will prove the case when $n=0$ first. The proof for $n>0$ follows from the same strategy, and we will indicate the necessary changes at the end of the proof.

	Suppose the operator $P(-,0)$ is a group homomorphism. We break the proof into a sequence of subclaims.

	 \begin{claim}
	 $g_3(P)=0$ and $R_{\sfrac{\ell}{2}}=\frac{\ell}{2}$. In particular, when viewed as a knot in $S^3$, $P$ is an unknot.
	 \label{claim:1}
	 \end{claim}
	 
	 The above claim is proven by applying Theorem~\ref{thm: tau} to knots of the form $\#^j K$ which have $\veps(K)=1$ and $\tau(K)<0$ as well as on knots which have $\veps(K)=1$ and $\tau(K)>0$. This implies that the constant term in the formula for $\tau(P(K,n))$ in Theorem~\ref{them:tau for epsilon=1} must vanish, when applied to knots with $\veps(K)=1$. The proves Claim~\ref{claim:1}.

\begin{claim}
\label{claim:2}
The width $N$ satisfies $N=\ell/2$.
\end{claim}
To prove the above claim, we consider  $\frac{\ell}{2}\le t$. From $ R_{\sfrac{\ell}{2}}=\frac{\ell}{2}$ (Claim~\ref{claim:1}), together with the monotonicity result
\[
 R_{\sfrac{\ell}{2}}\ge R_{\sfrac{\ell}{2}+1}\ge \cdots \ge R_{N-1}\ge R_{N} = g_3(P)+\frac{\ell}{2}=\frac{\ell}{2} \]
from Lemma \ref{lem:shape of top generators}, we conclude
\[
 R_{\sfrac{\ell}{2}}= R_{\sfrac{\ell}{2}+1}=\cdots=R_{N-1}= R_{N} = \frac{\ell}{2}.
 \]
 Here $N$ denotes the width of $L_P$ (Definition~\ref{def:width N}). 
Using the definition of $R_t$, this implies that 
\begin{equation}
H_{L_P}(t,r)=H_{L_P}(t,\infty) = H_\mu(t-\tfrac{\ell}{2})=0  \,\text{ for } \, t\ge \ell/2 \,\text{ and }\, r\ge \ell/2. \label{eq:values-H-function-1}
\end{equation}

Now, we consider $H_{L_P}(t,r)$ for $ \frac{\ell}{2}\le t$, and $ r<R_t = \frac{\ell}{2}. $ By the bounded gap property of $H$-function in the $r$-coordinate (the vertical direction), we get 
\begin{equation}
H_{L_P}(t,r) \le H_{L_P}(t,\ell/2) + \frac{\ell}{2}-r = \frac{\ell}{2}-r, \text{ for }\frac{\ell}{2}\le t  \, \text{ and } \, r<\frac{\ell}{2}.
\label{eq:claim:2-comparison-1}
\end{equation}
By the monotonicity property of $H$-function in the $t$-coordinate (the horizontal direction), we get 
\begin{equation}
H_{L_P} (t,r)\ge H_{L_P}(N,r) =\frac{\ell}{2}-r, \text{ for }\frac{\ell}{2}\le t,\,\text{ and }\,  r<\frac{\ell}{2}.
\label{eq:claim:2-comparison-2}
\end{equation}
Combining Equations~\eqref{eq:claim:2-comparison-1} and~\eqref{eq:claim:2-comparison-2}, we conclude
\begin{equation} 
H_{L_P} (t,r)= \frac{\ell}{2}-r,\, \text{ for }\,\frac{\ell}{2}\le t ,\,\text{ and }\, r<\frac{\ell}{2}.
 \label{eq:values-H-function-2}
\end{equation}

From the definition of the width $N$ (see Equation~\eqref{eq:width}), we see that Equations~\eqref{eq:values-H-function-1} and~\eqref{eq:values-H-function-2} imply that  $N=\ell/2$, completing the proof of the claim.

\begin{claim}
\label{claim:3}
 $R_{\sfrac{\ell}{2}-1}\ge \frac{\ell}{2}-1$. 
\end{claim}

By the definition of $R_t$, we have 
\begin{equation} H_{L_P}(\tfrac{\ell}{2}-1, R_{\sfrac{\ell}{2}-1}) = H_{L_P}(\tfrac{\ell}{2}-1,\infty) = H_{\mu}(-1) = 1.
\label{eq:R-ell/2-1-inequality}
\end{equation}
By monotonicity of the $H$-function in the $t$-direction, we have 
\[1= H_{L_P}(\tfrac{\ell}{2}-1, R_{\sfrac{\ell}{2}-1}) \ge H_{L_P}(\tfrac{\ell}{2}, R_{\sfrac{\ell}{2}-1}) = \frac{\ell}{2} - R_{\sfrac{\ell}{2}-1}.\]
Note that in the right-most equality in the above equation, we are using the fact that $N=\ell/2$, that $P$ is an unknot, and the fact that $ R_{\sfrac{\ell}{2}-1}\le  R_{\sfrac{\ell}{2}}$, which follows from Lemma \ref{lem:shape of top generators}. Rearranging Equation~\eqref{eq:R-ell/2-1-inequality} yields $R_{\sfrac{\ell}{2}-1} \ge \ell/2-1$, which was the claim.

\begin{claim}
\label{claim:4} $R_{\sfrac{\ell}{2}-1}=R_{\sfrac{\ell}{2}-2}=\cdots =R_{-N}=-\tfrac{\ell}{2}$. 
\end{claim}

Claims~\ref{claim:1} and~\ref{claim:3} imply that
\[
R_{\sfrac{\ell}{2}-1}\ge g_3(P)+\frac{\ell}{2}-1.
\]
Therefore, the condition in Equation~\eqref{eq:extra condition when epsilon=0} in the $\veps(K)<0$ case of Theorem~\ref{thm: tau} is satisfied, so we can apply Theorem~\ref{thm: tau} to knots $K$ with $\veps(K)<0$. We apply Theorem~\ref{thm: tau} to knots of the form $\#^j K$ where $\tau(K)>0$ and $\veps(K)<0$, and we conclude that 
\[\max\{ R_{\sfrac{\ell}{2}-1} + \tfrac{\ell}{2}, R_{\sfrac{\ell}{2}} - \tfrac{\ell}{2}\} =\max\{ R_{\sfrac{\ell}{2}-1} + \frac{\ell}{2},0\}=0.\] 
Therefore
\begin{equation}
R_{\sfrac{\ell}{2}-1}\le -\frac{\ell}{2}. 
\label{eq:R_l/2-1-inequality-1}
\end{equation}
On the other hand, we apply Lemma~\ref{lem:shape of top generators} and use the fact that $g_3(P)=0$ to see that
\begin{equation}
 R_{\sfrac{\ell}{2}-1} \ge R_{\frac{\ell}{2}-2}\ge \cdots \ge R_{-N+1}\ge R_{-N} = g_3(P)-\frac{\ell}{2} = -\frac{\ell}{2}.
 \label{eq:R_l/2-1-inequality-2}
 \end{equation}
Combining Equations~\eqref{eq:R_l/2-1-inequality-1} and ~\eqref{eq:R_l/2-1-inequality-2}
we conclude
\[
 R_{\sfrac{\ell}{2}-1} = R_{\frac{\ell}{2}-2}= \cdots = R_{-N+1}= R_{-N} = -\frac{\ell}{2},\]
 as claimed.

\begin{claim}
\label{claim:5} The linking number $\ell$ is either 0 or 1.
\end{claim}

The above claim follows immediately from Claims~\ref{claim:3} and \ref{claim:4} and the fact that $R_{\sfrac{\ell}{2}-1}\le R_{\sfrac{\ell}{2}}$ by Lemma~\ref{lem:shape of top generators}.

\begin{claim} \label{claim:6} If $\ell=0$, then $H_{P_L}=H_{U_2}$. If $\ell=1$, then $H_{P_L}=H_{\cH_+}$.
\end{claim}

In both cases, Claim~\ref{claim:1} implies that the width $N$ is $\ell/2$, which is either $0$ or $1/2$. If $N=\ell=0$, then Equations~\eqref{eq:values-H-function-1} and~\eqref{eq:values-H-function-2} imply that $H_{L_P}(r,t)=H_{U_2}(r,t)$ for $r\ge 0$. By symmetry of the $H$-function (see Lemma~\ref{lem: properties of H function}) this implies that $H_{L_P}(r,t)=H_{U_2}(r,t)$ for all $(r,t)\in \Z\times \Z$. If $N=\ell/2=1/2$, the same argument works, since $H_{L_P}(r,t)=H_{\cH_+}(r,t)$ for all $r>0$ by Equations~\eqref{eq:values-H-function-1} and~\eqref{eq:values-H-function-2}, and symmetry implies equality for all $(r,t)\in (\Z+\tfrac{1}{2})\times (\Z+\tfrac{1}{2}).$

This concludes the proof of the main theorem when $n=0$.

\begin{claim} The main theorem holds also when $n>0$.
\end{claim}
We suppose that $P(-,n)$ is a homomorphism for some $n>0$. By considering companion knots of the form $\#^j K$ with $\veps(K)=1$ and $\tau(K)>0$, as well as knots $K$ with $\veps(K)=1$ and $\tau(K)<0$, we conclude that \[	 R_{\sfrac{\ell}{2}}-\frac{\ell}{2} +\frac{(\ell-1)\ell}{2}n=0 \quad \text{and}\quad g_3(P) +\frac{(\ell-1)\ell}{2}n=0.\]
	Recall we have assumed that the linking number $\ell$ was non-negative by choosing the orientation on $P$ appropriately, and because 
	\[ R_{\sfrac{\ell}{2}}-\frac{\ell}{2} \ge g_3(P) \ge 0,\]
	the only possibility is 
	\[ R_{\sfrac{\ell}{2}}-\frac{\ell}{2} = g_3(P) =0, \text{ and } \ell =0 \text{ or }1.\]
	Then, the contribution of $n$ in the formula for $\tau(P(K,n))$, which is $\frac{(\ell-1)\ell}{2}n$, becomes trivial, and we return to the same situation as when $n=0$. 
\end{proof}

\section{A local inequality}
\label{sec:inequality}

In this section, we prove a result about the full knot Floer complex of L-space satellites of a knot $K$. We recall that if $K_1$ and $K_2$ are knots, a \emph{local map} from $\cCFK(K_1)$ to $\cCFK(K_2)$ consists of a grading preserving chain map
\[
F\colon \cCFK(K_1)\to \cCFK(K_2)
\]
which becomes a homotopy equivalence after inverting $U=WZ$. 

\begin{thm} 
\label{thm:local-map}
Suppose that $P$ an L-space satellite operation with linking number $\ell\ge 0$, and that $n$ is an integer. Then for any knot $K\subset S^3$, there exists a $(\gr_{\ws},\gr_{\zs})$-grading preserving local map
\[
F\colon \cCFK(P(K,n))\to \cCFK(K_{\ell,\ell n+1}).
\]
\end{thm}

As immediate corollaries, we obtain the following:
\begin{cor}\label{cor:inequality} The above implies inequalities for most concordance invariants. For example, we obtain that
\begin{enumerate}
\item $\tau(P(K,n))\ge \tau(K_{\ell,\ell n+1})$.
\item $V_k(P(K,n))\ge V_k(K_{\ell,\ell n+1})$ for all $k\in \Z$.
\item $\Upsilon_{P(K,n)}(t)\le \Upsilon_{K_{\ell,\ell n+1}}(t)$ for all $t\in [0,2]$. 
\end{enumerate}
\end{cor}
This is stated as Proposition~\ref{prop:inequality-intro} in the introduction. In the above, $V_k$ denotes Rasmussen's local $h$-invariant \cite{RasmussenGodaTeragaito}, and $\Upsilon_K(t)$ denotes the concordance invariant of Ozsv\'{a}th, Stipsicz and Szab\'{o} \cite{OSSUpsilon}.

The following is stated as Corollary~\ref{cor:linking-number-0-intro} in the introduction:
\begin{cor} If $P$ is a winding number 0 L-space satellite operator, then
\[
\tau(P(K,n))\ge 0,\quad \text{and}\quad 
 \Upsilon_{P(K,n)}(t)\le 0. 
\]
\end{cor}
\begin{proof} When $\ell=0$, the knot $K_{\ell,\ell n+1}$ is the unknot, so the result follows from Corollary~\ref{cor:inequality}. 
\end{proof}

We now prove Theorem~\ref{thm:local-map}:

\begin{proof}[Proof of Theorem~\ref{thm:local-map}]
The key step in our proof is to construct a morphism of $DA$-bimodules
\[
f_*^1 \colon {}_{\cK} \cX(L_P)^{\bF[W,Z]}\to {}_{\cK} \cX(T(2,2\ell))^{\bF[W,Z]}.
\]
This map will be analogous to a local map between knot like complexes in that it becomes a homotopy equivalence after we invert $U$ on the right (i.e. tensor with the rank 1 $DA$ bimodule corresponding to the inclusion of $\bF[W,Z]$ into $\bF[W,Z,W^{-1},Z^{-1}]$). In particular, the map $F$ will be 
\[
F=\bI\boxtimes f_*^1\colon \cX_n(K)^{\cK}\boxtimes {}_{\cK} \cH_-^{\cK}\boxtimes {}_{\cK}\cX(L_P)^{\bF[W,Z]}\to \cX_n(K)^{\cK}\boxtimes {}_{\cK} \cH_-^{\cK}\boxtimes {}_{\cK}\cX(T_{2,2\ell})^{\bF[W,Z]}.
\]

To construct $f_*^1$, we have to recall several important facts about staircase complexes. Let $\cA$ and $\cB$ be staircase complexes. Assume that on $\cA$ and $\cB$, the $(\gr_{\ws},\gr_{\zs})$-gradings take values in $\Z\times \Z$, and that on $H_*(\cA)$ and $H_*(\cB)$ they take values in $2\Z\times 2 \Z$ (though we allow $\cA$ and $\cB$ to take non-standard gradings, in the sense of Definition~\ref{def:standard-gradings}).  Since $\cA$ and $\cB$ are free-resolutions of their homology, there exists a grading preserving local map from $\cA$ to $\cB$ if and only there is a grading preserving inclusion of $\bF[W,Z]$-modules $H_*(\cA)\hookrightarrow H_*(\cB)$.

 We can restate the condition that $H_*(\cA)\subset H_*(\cB)$ in terms of an $H$-function of $\cA$ and $\cB$, as follows. We define $H_{\cA}\colon \Z\to \Z$ by setting $H_{\cA}(s)$ to be the $-\frac{1}{2}$ times the maximum grading of a $\bF[U]$-non-torsion element in $H_*(\cA)$ which lies in Alexander grading $s$. We define $H_{\cB}$ similarly. We observe that there is a grading preserving inclusion of $H_*(\cA)$ into $H_*(\cB)$ if and only if
 \begin{equation}
 H_{\cA}(s)\ge H_{\cB}(s)\label{eq:inclusion-homologies-H-function}
 \end{equation}
 for all $s\in \Z$.
 
 We now define the component $f_1^1$ of $f_*^1$. Write $\cC_{s_1}$ for the idempotent 0 complexes of ${}_{\cK} \cX(L_P)^{\bF[W,Z]}$, and write $\cC_{s_1}'$ for the idempotent 0 complexes of ${}_{\cK} \cX(T(2,2\ell))^{\bF[W,Z]}$. By our above observations about staircase complexes, in order to construct $f_1^1$ on left idempotent 0, it suffices to show that there is a grading preserving local map from each $\cC_{s_1}$ into $\cC_{s_1}'$. By the observation in Equation~\eqref{eq:inclusion-homologies-H-function}, it suffices to show that
 \begin{equation}
 H_{L_p}(s_1,s_2)\ge H_{T(2,2\ell)}(s_1,s_2)\label{eq:inequality-H-function-local-map}
 \end{equation}
 for all $(s_1,s_2)\in \bH(L_P)=\bH(T(2,2\ell)).$ It is not hard to see that this also allows for the construction of the map $f_1^1$ on left idempotent 1, because $\ve{I}_1\cdot {}_{\cK} \cX(L_P)^{\bF[W,Z]}$ can be identified with $\bF[T,T^{-1}] \otimes \cC_{s_1}$ for $s_1\gg 0$, and similarly for $\ve{I}_1 \cdot {}_{\cK} \cX(T(2,2\ell))^{\bF[W,Z]}$.

 To prove Equation~\eqref{eq:inequality-H-function-local-map}, we use Lemma~\ref{lem:properties-H-function} and a direct computation of $H_{T(2,2\ell)}(s_1,s_2)$. We recall that 
 \begin{equation}
 H_{T(2,2\ell)}(s_1,s_2)=\max\big\{H_U(s_1-\ell/2), -s_1-s_2+H_U(-s_1-\ell/2) \big\}.
\label{eq:T22q-computation}
 \end{equation}
 See \cite{CZZ}*{Section~6.1} for a description of $H_{T(2,2\ell)}$ and the module ${}_{\cK} \cX(T(2,2\ell))^{\bF[W,Z]}$. Lemma~\ref{lem:properties-H-function} implies that
 \[
 H_{L_P}(s_1,s_2)=H_U(s_1-\ell/2)
 \] 
 for $s_2\gg 0$. By the monotonicity property in Lemma~\ref{lem:properties-H-function}, we conclude that
 \[
 H_{L_P}(s_1,s_2)\ge H_U(s_1-\ell/2)
 \]
 for all $(s_1,s_2)$.
 Using the symmetry property, we conclude that if $s_2\ll 0$, then
 \[
 H_{L_P}(s_1,s_2)=-s_1-s_2+H_{L_P}(-s_1,-s_2)=-s_1-s_2+H_U(-s_1-\ell/2).
 \]
 Using the bounded gap property, we conclude that
 \[
 H_{L_P}(s_1,s_2)\ge -s_1-s_2+H_U(-s_1-\ell/2)
 \]
 for all $(s_1,s_2)$. Using Equation~\eqref{eq:T22q-computation}, the inequality in Equation~\eqref{eq:inequality-H-function-local-map} follows.

We now describe how to construct the map $f_j^1$ for $j>1$. The construction is essentially the same as the construction of the module ${}_{\cK} \cX(L_P)^{\bF[W,Z]}$ from \cite{CZZ}*{Section~5}. If $a$ is a monomial in the algebra, we want $f_2^1(a,-)$ to satisfy
\[
\delta_1^1 \circ f_2^1(a,-)+f_2^1(a, \delta_1^1(-))=f_1^1 \circ \delta_2^1(a,-)+\delta_2^1(a,f_1^1(-)).
\]
  The right hand side of the above equation is the sum of two chain maps from one staircase to another. These two chain maps have the same $(\gr_{\ws},\gr_{\zs})$-grading, and are both local maps (i.e. become homotopy equivalences after we localize at $U$). By \cite{CZZ}*{Lemma~5.6}, the sum $f_1^1\circ \delta_2^1(a,-)+\delta_2^1(a,f_1^1(-))$ is null-homotopic, so we pick $f_2^1(a,-)$ to be any suitably $(\gr_{\ws},\gr_{\zs})$-graded null-homotopy which increases the algebraic grading by 1. We now consider the $A_\infty$ relations with two algebra inputs. Let $a$ and $b$ be monomials in $\cK$. Then
  \[
  \delta_3^1(a,b,f_1^1(-))+f_1^1(\delta_3^1(a,b,-))+f_2^1(a, \delta_2^1(b,-))+\delta_2^1(a,f_2^1(b,-))
  \]
  is a chain map between two staircase complexes. Furthermore, this map increases algebraic grading by 1. By \cite{CZZ}*{Lemma~5.8} this map is zero. Therefore the length 2 relations are satisfied. The $A_\infty$ relations with $j>2$ algebra inputs are similarly satisfied because the structure relation vanishes by algebraic grading considerations.
  
  Finally, we note that the map $f_*^1$ becomes a homotopy equivalence after inverting $U$, because the maps $f_1^1$ do.
\end{proof} 

\begin{rem} More generally, it is straightforward to adapt the above proof to show the following. If $P$ and $Q$ are two L-space satellite operators with the same linking number $\ell$, and $H_{L_P}(t,r)\le H_{L_Q}(t,r)$ for all $t,r$, then
\[
\tau(P(K,n))\ge \tau(Q(K,n))
\]
for all knots $K$ and integers $n$. Similar remarks hold for the $V_k$ and $\Upsilon_{K}(t)$ invariants. 
\end{rem}

\bibliographystyle{custom}
\def\MR#1{}
\bibliography{biblio}

\end{document}